\documentclass[11pt,reqno]{amsart}
\usepackage{fullpage,enumerate}
\usepackage{amsfonts}
\usepackage[all]{xy} 
\usepackage{amssymb}
\usepackage{comment}
\usepackage{wrapfig}
\usepackage{times}
\usepackage{graphicx}
\usepackage{amsmath,bm}
\usepackage[hidelinks]{hyperref}
\usepackage{tikz-cd}
\usepackage{tikz}
\usepackage{appendix}
\usepackage{color}\usepackage{float}
\usepackage[margin=1.3in,marginparwidth=1.1in,marginparsep=0.1in]{geometry}
\usepackage{marginnote,color}
\usepackage{import}
\usepackage[normalem]{ulem}
\usepackage{graphicx}

\usepackage{caption}
\usepackage{mathrsfs}
\usepackage{geometry}
\usepackage{epsfig,amscd,amsthm,mathtools,fourier}
\numberwithin{equation}{section}

\newtheorem{thm}{Theorem}[section]

\newtheorem{lem}[thm]{Lemma}
\newtheorem{prop}[thm]{Proposition}
\newtheorem*{Thm}{Main Theorem}

\theoremstyle{definition}
\newtheorem{ex}[thm]{Example}

\newtheorem{defn}[thm]{Definition}

\newtheorem{constr}[thm]{Construction}
\newtheorem{cl}[thm]{Claim}

\newcommand{\sing}{\mathrm{sing}}

\newcommand{\irr}{\mathrm{irr}}

\newcommand{\trop}{\mathrm{trop}}

\newcommand{\rank}{{\rm rank}}

\newcommand{\val}{{\rm val}}

\newcommand{\Spec}{{\rm Spec}}

\newcommand{\RR}{{\mathbb R}}
\newcommand{\ZZ}{{\mathbb Z}}
\newcommand{\PP}{{\mathbb P}}
\newcommand{\NN}{{\mathbb N}}
\newcommand{\GG}{{\mathbb G}}

\newcommand{\CL}{{\mathcal L}}

\newcommand{\CO}{{\mathcal O}}

\newcommand{\CX}{{\mathcal X}}

\newcommand{\Star}{{\rm Star}}

\def\:{\colon}
\def\val{\nu}

\newcommand{\cX}{{\mathcal{X}}}

\newcommand{\pg}{\mathrm{p_g}}

\newcommand{\ev}{\mathrm{ev}}

\newcommand{\m}{\mathfrak w}

\theoremstyle{remark}
\newtheorem{rem}[thm]{Remark}

\makeatletter
\DeclareRobustCommand{\cev}[1]{%
  {\mathpalette\do@cev{#1}}%
}
\newcommand{\do@cev}[2]{%
  \vbox{\offinterlineskip
    \sbox\z@{$\m@th#1 x$}%
    \ialign{##\cr
      \hidewidth\reflectbox{$\m@th#1\vec{}\mkern4mu$}\hidewidth\cr
      \noalign{\kern-\ht\z@}
      $\m@th#1#2$\cr
    }%
  }%
}
\makeatother

\title{Degeneration of curves on some polarized toric surfaces}
\author{Karl Christ, Xiang He, and Ilya Tyomkin}

\thanks{IT is partially supported by the Israel Science Foundation (grant No. 821/16). KC was partially supported by the Israel Science Foundation (grant No. 821/16) and by the Center for Advanced Studies at BGU. XH is supported by the ERC Consolidator Grant 770922 - BirNonArchGeom.}

\address[Christ]{$~^1$Department of Mathematics\\
	Ben-Gurion University of the Negev\\P.O.Box 653 \\Be'er Sheva\\ 84105\\  Israel and $~^2$Institute of Algebraic Geometry\\
	Leibniz University Hannover\\Welfengarten 1 \\30167 Han\-no\-ver\\  Germany }\email{kchrist@math.uni-hannover.de}

\address[He]{$~^1$Einstein Institute of Mathematics\\
	The Hebrew University of Jerusalem\\ Giv'at Ram\\ Jerusalem\\ 91904\\ Israel and $~^2$Yau Mathematical Sciences Center, Ningzhai, Tsinghua University, Hai Dian District, Beijing, China, 100084
 }
	\email{xianghe@mail.tsinghua.edu.cn}

\address[Tyomkin]{Department of Mathematics\\
	Ben-Gurion University of the Negev\\P.O.Box 653 \\Be'er Sheva\\ 84105\\  Israel }\email{tyomkin@math.bgu.ac.il}

\begin{document}
	
\begin{abstract}
We address the following question: {\em Given a polarized toric surface $(S,\CL)$, and a general integral curve $C$ of geometric genus $g$ in the linear system $|\CL|$, do there exist degenerations of $C$ in $|\CL|$ to general integral curves of smaller geometric genera?} We give an affirmative answer to this question for surfaces associated to $h$-transverse polygons, provided that the characteristic of the ground field is large enough. We give examples of surfaces in small characteristic, for which the answer to the question is negative. In case the answer is affirmative, we deduce that a general curve $C$ as above is nodal. In characteristic $0$, we use the result to show the irreducibility of Severi varieties of a large class of polarized toric surfaces with $h$-transverse polygon.
\end{abstract}

\maketitle
	
\setcounter{tocdepth}{1}
\tableofcontents

\section{Introduction}
In algebraic geometry, degenerations are one of the main tools in the study of algebraic curves. They play a central role in enumerative questions \cite{caporaso1998counting}, Brill-Noether theory \cite{K76, GH80}, irreducibility problems \cite{Har86}, etc.
In the current paper, we investigate degenerations of curves on toric surfaces. Our main motivation comes from the Severi problem and Harris' approach to it, which we briefly recall. Roughly speaking, Severi varieties parametrize integral curves of a fixed geometric genus in a given linear system. The problem of the irreducibility of Severi varieties in the planar case, the classical Severi problem, had been a long-standing open problem until settled in characteristic $0$ by Harris in 1986 \cite{Har86}. In our preceding paper \cite{CHT20a}, we extended his result to arbitrary characteristic. The present paper is concerned with the analogous problem for toric surfaces.

Harris' approach goes back to the original ideas of Severi, and consists of two parts: first, one proves that the closure of any irreducible component of the Severi variety contains the Severi variety $V_{0,d}^\irr$ parametrizing integral rational curves; this amounts to a statement about the possible degenerations of curves in the given linear system. Second, one describes the branches of the original Severi variety along $V_{0,d}^\irr$, and uses monodromy arguments to show that all branches belong to the same irreducible component. A similar approach was used by the third author in \cite{Tyo07} to prove the irreducibility of Severi varieties on Hirzebruch surfaces in characteristic $0$.

In this paper, we prove the degeneration part of the argument. We restrict our attention to the case of polarized toric surfaces associated to so-called $h$-transverse polygons. This is a rich class of surfaces that includes the projective plane, Hirzebruch surfaces, and many other examples. In particular, this class includes examples of toric surfaces admitting reducible Severi varieties \cite{Tyo14, LT20}; see Example~\ref{ex:kites irreduciblity}. Our main result asserts that any irreducible component of a Severi variety on such a surface contains in its closure the Severi variety parametrizing integral rational curves, provided that the characteristic of the ground field is either $0$ or large enough. Moreover, any integral curve of geometric genus $g$ in the linear system associated to an $h$-transverse polygon can be degenerated to an integral curve of any genus $0\le g'\le g$.

In the planar case, the monodromy part of the argument is relatively easy, and the main difficulty lies in the degeneration step. We shall mention, however, that in general the monodromy argument is subtle, and there are examples of toric surfaces for which it fails. Moreover, some toric surfaces admit reducible Severi varieties, and the structure of the monodromy groups seems to be the main reason for their reducibility. An interested reader can find more about such examples in the recent paper of Lang and the third author \cite{LT20}, and about the computation of the monodromy groups in the case of toric surfaces in Lang's \cite{lang2020monodromy}.

\subsection{The main results}
We work over an algebraically closed field $K$. To formulate the main result, we need the notion of the width of an $h$-transverse polygon. Recall that an $h$-transverse polygon in $\RR^2$ is a convex lattice polygon $\Delta$, whose horizontal slices at integral heights $y=k\in\ZZ$ are either empty or intervals with integral boundary points. We define the {\em width} of an $h$-transverse polygon $\Delta$ to be the maximal length of its horizontal slices and denote it by  $\mathtt w(\Delta)$.

\begin{Thm}\label{Thm:main thm}
Let $\Delta\subset\RR^2$ be an $h$-transverse lattice polygon, and $(S_\Delta,\CL_\Delta)$ the associated polarized toric surface. If $\mathrm{char}(K)=0$ or $\mathrm{char}(K)>\mathtt w(\Delta)/2$, then any irreducible component of the Severi variety parametrizing integral curves of geometric genus $g\ge 0$ in the linear system $|\CL_\Delta|$ contains in its closure the Severi variety parametrizing integral rational curves in $|\CL_\Delta|$.
\end{Thm}

We refer the reader to Theorem~\ref{thm:main thm} for a stronger version of this statement, in which we allow for some fixed tangency conditions with the toric boundary of $S_\Delta$. The bound $\mathrm{char}(K)>\mathtt w(\Delta)/2$ is sharp for $\mathrm{char}(K) \geq 3$: In Proposition~\ref{prop:low characteristic counterexample} we provide for every prime $p$ an explicit example of a polarized toric surface with $h$-transverse polygon of width $\max\{6, 2p\}$ and a component of the corresponding Severi variety parametrizing genus $1$ curves, that violates the assertion of the theorem if $\mathrm{char}(K)=p$. 

We give two applications of the Main Theorem. First, we generalize Zariski's Theorem in Theorem~\ref{thm:Zariski}; that is, we show under the assumptions of the Main Theorem, except requiring $\mathrm{char}(K)>\mathtt w(\Delta)$ instead of $\mathrm{char}(K)>\mathtt w(\Delta)/2$, that a general element in any irreducible component of a Severi variety is a nodal curve. This was previously known only in characteristic $0$ \cite{KS13}, and there are examples of polarized toric surfaces, for which the assertion fails in positive characteristic \cite{Tyo13}. In fact, the bound $\mathrm{char}(K)>\mathtt w(\Delta)$ in Zariski's Theorem is sharp as Example~\ref{ex:severi with tangency} shows. Our argument goes as follows: we prove by explicit calculations that a general integral rational curve is nodal. By the Main Theorem, any irreducible component of a Severi variety contains in its closure the general integral rational curves, and the claim follows since being nodal is an open condition.

The second application returns to our original motivation, and we deduce irreducibility of Severi varieties whenever the assumptions of the Main Theorem are satisfied and the results of Lang \cite{lang2020monodromy} ensure that the monodromy acts as a full symmetric group on the set of nodes of a general integral rational curve; that is when also the monodromy part in Harris' approach is known. This is the case when $\mathrm{char}(K) = 0$, the primitive outer normals to the sides of an $h$-transverse polygon $\Delta$ generate the whole lattice $\ZZ^2$, and the polarization is ample enough; see Theorem~\ref{thm:monodromy} for a precise statement. In \cite[Corollary 1.11]{B16}, Bourqui obtains irreducibility for some smooth polarized toric surfaces. Note however the crucial difference that, unlike in {\it loc. cit.}, the ampleness condition in our Theorem~\ref{thm:monodromy} is independent of the fixed geometric genus $g$. 

The tangency conditions we allow for in Theorem~\ref{thm:main thm} recover for $S_\Delta = \PP^2$ some of the generalized Severi varieties of Caporaso and Harris \cite{caporaso1998counting}; namely, the ones where the position of the point with tangency to the line is not fixed; see Remark~\ref{rem:generalized severi} for details. Recently, the irreducibility of generalized Severi varieties of irreducible curves has been proven in characteristic $0$ by Zahariuc \cite{Z19}. We expect that our methods can be used to obtain a more detailed understanding of degenerations of curves with fixed tangency profiles on toric surfaces; this however was not the focus of the current paper and the tangencies we allow for in Theorem~\ref{thm:main thm} are the ones we get for free in an argument designed for curves that intersect the toric boundary transversally.
Note also, that we ensure in the current paper that the curves of smaller geometric genus constructed in each degeneration are irreducible, which we did not require in \cite{CHT20a}; this is a second aspect, in which we get a more detailed picture even in the planar case.

The degeneration argument of \cite{CHT20a} works for a larger class of toric surfaces, and not just $\PP^2$. While this class does not contain all surfaces with $h$-transverse polygons considered in the present paper, it does contain for example Hirzebruch surfaces and toric del Pezzo surfaces. In an upcoming third part \cite{CHT20c} of the series, we give a characteristic-free tropical proof that involves no monodromy arguments in these cases, thus establishing the irreducibility of Severi varieties for such surfaces in any characteristic. As a consequence, we obtain the irreducibility of Hurwitz schemes in any characteristic.

\subsection{The approach}
Our approach is based on tropical geometry and develops further the techniques introduced in \cite{CHT20a}. The main idea is to control the degenerations by studying the tropicalizations of one-parameter families of algebraic curves and the induced maps to the moduli spaces of tropical curves. The polygon $\Delta$ being $h$-transverse allows us to work with floor decomposed tropical curves -- a particularly convenient class of tropical curves introduced by Brugall\'e and Mikhalkin \cite{BM08}. Let us explain our approach in some detail.

For the closure $V$ of a given irreducible component of a Severi variety, we consider the locus $V^{FD}$ of algebraic curves, whose tropicalizations are floor decomposed, weightless, 3-valent, and immersed except at contracted legs; see \S~\ref{sec:nottropcur} for the definition of weight function and weightless tropical curves. We analyze the tropicalizations of one-parameter families of curves in $V$ and show that in the closure of $V^{FD}$ there are curves, whose tropicalization contains either a contracted loop or a contracted edge that belongs to a cycle of the tropicalization. Proceeding similarly to \cite{CHT20a}, we show that the length of the contracted edge/loop can be increased indefinitely, and conclude that the closure of $V^{FD}$ contains general irreducible curves of smaller genera. From this, the Main Theorem follows by induction and a dimension count.

Although the strategy of the proof is similar to the proof in \cite{CHT20a}, it is more complicated. In the planar case discussed in \emph{loc. cit.}, the one-parameter family connecting a general tropical curve with one that contains a contracted edge can be constructed such that all its fibers are weightless, $3$-valent tropical curves, except for possibly one $4$-valent vertex. In the more general setting of the current paper, this is no longer the case, and we have to allow for additional types of tropical curves. For each of them, we have to prove a local liftability result to ensure that the curves we construct in the tropical family are tropicalizations of curves in the closure of $V^{FD}$. Among the new types considered in this paper is the case of tropical curves, which are weightless and $3$-valent except for a unique $2$-valent vertex of weight $1$. To show local liftability in this case, we introduce a local version of Speyer's well-spacedness condition \cite{S14} in Proposition~\ref{prop:weight one midpoint}. Then we use an argument that uses properties of modular curves parametrizing elliptic curves together with a torsion point of a fixed order, that depends on the characteristic; see Lemma~\ref{lem:alpha at weight one vertex} for details. This is the reason why unlike the result in \cite{CHT20a}, the Main Theorem is not stated for arbitrary characteristic. As mentioned above, the theorem may fail in low characteristic: a related construction (also involving modular curves) gives explicit examples when this happens; see Proposition~\ref{prop:low characteristic counterexample} for details.

\subsection{The structure of the paper}

In \S \ref{sec:notterm} we recall basic definitions and fix some terminology.  In \S \ref{sec:tropical curves} we discuss $h$-transverse polygons, floor decomposed curves, and the tropicalization of parametrized curves. We establish in Proposition~\ref{prop:weight one midpoint} a local version of Speyer's well-spacedness condition for realizable tropical curves. In \S \ref{sec:liftability}
we recall the local liftability results established in \cite{CHT20a} and prove in Lemmas~\ref{lem:alpha at flattened cycle} and \ref{lem:alpha at weight one vertex} the additional ones needed in the degeneration argument for $h$-transverse $\Delta$. In \S \ref{sec:degeneration} we prove Theorem~\ref{thm:main thm}, which immediately implies the Main Theorem. In Proposition~\ref{prop:low characteristic counterexample} we show that in low characteristic the claim of the theorem may fail. In \S \ref{sec:zar} we prove Zariski's Theorem, and in \S \ref{sec:irr} we combine the results with those of \cite{lang2020monodromy} to establish irreducibility of Severi varieties in some cases.

\medskip
	
\noindent {\bf Acknowledgements.} We are very grateful to an anonymous referee, whose comments helped us to improve the presentation and to fix some inaccuracies in an earlier version of the paper.

\section{Notation and terminology}\label{sec:notterm}

We follow the conventions of \cite{CHT20a}. For the convenience of the reader, we recall them here and refer to \textit{loc. cit.} for further details.

\subsection{Families of curves}

\label{subsec: Families of curves}

By a \emph{family of curves} we mean a flat, projective morphism of finite presentation and relative
dimension one. By a collection of \emph{marked points} on a family of curves, we mean a collection of
disjoint sections contained in the smooth locus of the family. A family of curves with marked
points is \emph{prestable} if its fibers have at-worst-nodal singularities. It is called
\emph{(semi-)stable} if so are its geometric fibers. A prestable curve with marked points defined over a field is
called \emph{split} if the irreducible components of its normalization are geometrically irreducible and
smooth and the preimages of the nodes in the normalization are defined over the ground field.
A family of prestable curves with marked points is called split if all of its fibers are so.  If $U \subset Z$ is open, and $(\cX , \sigma_\bullet)$ is a family of curves with marked points over $U$, then by a
\emph{model} of $(\cX , \sigma_\bullet)$ over $Z$ we mean a family of curves with marked points over $Z$, whose restriction
to $U$ is $(\cX , \sigma_\bullet)$.

\subsection{Toric varieties and parametrized curves}
For a toric variety $S$, we denote the lattice of characters by $M$, of cocharacters by $N$, and the monomial functions by $x^m$ for $m\in M$. The polarized toric variety associated to a lattice polytope $\Delta\subset M_\RR \coloneqq M \otimes_{\ZZ} \RR$ is denoted by $(S_\Delta, \mathscr L_\Delta)$, and the facets of $\Delta$ are denoted by $\partial \Delta_i$. In the current paper, we mostly work with toric surfaces. For simplicity, we will always fix an identification $M=N=\ZZ^2$ such that the duality between the lattices is given by the usual dot product pairing.

A \emph{parametrized curve} in a toric variety $S$ is a smooth projective curve with marked points $(X,\sigma_\bullet)$ together with a map $f\: X \to S$ such that $f(X)$ does not intersect orbits of codimension greater than one, and the image of $X \setminus \left( \bigcup_i \sigma_i \right)$ under $f$ is contained in the dense torus $T \subset S$.
A \emph{family of parametrized curves} $f\colon \cX \dashrightarrow S_{\Delta}$ over $K$ consists of the following data: 
	\begin{enumerate}
	    \item a smooth, projective base curve $(B, \tau_{\bullet})$ with marked points,
	    \item a family of stable marked curves $(\cX \to B, \sigma_{\bullet})$, smooth over $B' \coloneqq B \setminus \left( \bigcup_i \tau_i \right)$, and
	    \item a rational map $f\colon \cX \dashrightarrow S_{\Delta}$, defined over $B'$, such that for any closed point $b \in B'$ the restriction $\cX_b \to S_{\Delta}$ is a parametrized curve. 
	\end{enumerate}

\subsection{Tropical curves}\label{sec:nottropcur}
The graphs we consider have vertices, edges, and half-edges, called {\em legs}. We denote the set of vertices of a graph $\mathbb G$ by $V(\GG)$, of edges by $E(\GG)$, and of legs by $L(\GG)$. We set $\overline E(\mathbb G):=E(\GG)\cup L(\GG)$. By a {\em tropical curve} $\Gamma$ we mean a finite graph $\mathbb G$ with ordered legs equipped with a {\em length function} $\ell\colon E(\mathbb G)\rightarrow \mathbb R_{>0}$ and a {\em weight (or genus) function} $g\colon V(\mathbb G)\rightarrow \mathbb Z_{\geq 0}$. We say that a curve is {\em weightless} if the weight function is identically zero. For any leg $l\in L(\GG)$ we set $\ell(l):=\infty$. We usually view tropical curves as polyhedral complexes by identifying the edges of $\GG$ with bounded closed intervals of corresponding lengths in $\RR$ and identifying the legs of $\GG$ with semi-bounded closed intervals in $\RR$.
	
For $e\in \overline E(\mathbb G)$ we denote by $e^\circ$ the interior of $e$, and use $\vec e$ to indicate a choice of orientation on $e$. The legs are always oriented away from the adjacent vertex. For $v\in V(\mathbb G)$, we denote by $\Star(v)$ the {\em star} of $v$, i.e., the collection of oriented edges and legs having $v$ as their tail. In particular, $\Star(v)$ contains two oriented edges for every loop based at $v$. The number of edges and legs in $\Star(v)$ is called the {\em valence} of $v$. The {\em genus} of $\Gamma$ is defined to be $g(\Gamma) = g(\GG) :=1-\chi(\mathbb G)+\sum_{v\in V(\mathbb G)}g(v)$, where $\chi(\GG):=b_0(\GG)-b_1(\GG)$ denotes the Euler characteristic of $\GG$. Finally, a tropical curve $\Gamma$ is called {\em stable}  if every vertex of weight zero has valence at least three, and every vertex of weight one has valence at least one.

A {\em parametrized tropical curve} is a balanced piecewise integral affine map $h\colon \Gamma\rightarrow N_\mathbb R$ from a tropical curve $\Gamma$ to $N_\mathbb R$, i.e., a continuous map, whose restrictions to the edges and legs are integral affine, and for any vertex $v\in V(\Gamma)$ the following balancing condition holds: $\sum_{\vec e\in\Star(v)}\frac{\partial h}{\partial \vec e}=0.$ Note that the slope $\frac{\partial h}{\partial \vec e}\in N$ is not necessarily primitive, and its integral length is the stretching factor of the affine map $h|_e$. We call it the {\it multiplicity} of $h$ along $\vec e$, and denote it by $\m(h,e)$, or $\m(e)$ if the parametrization $h$ is clear from the context. In the literature, the multiplicity $\m(e)$ is often called the weight of $e$, but we shall not use this terminology in the current paper.

A parametrized tropical curve $(\Gamma, h)$ is called {\em stable} if so is $\Gamma$. Its {\em combinatorial type} $\Theta$ is defined to be the weighted underlying graph $\mathbb G$ with ordered legs equipped with the collection of slopes $\frac{\partial h}{\partial \vec e}$ for $e \in \overline E(\GG)$. We denote the group of automorphisms of a combinatorial type $\Theta$ by $\mathrm{Aut}(\Theta)$, and the isomorphism class of $\Theta$ by $[\Theta]$.

The {\em degree} $\nabla$ of a parametrized tropical curve is the sequence $\left(\frac{\partial h}{\partial \vec l_i}\right)$, where $l_i$'s are the legs not contracted by $h$. If $\rank(N)=2$, then we say that a degree $\nabla$ is \emph{dual to} a polygon $\Delta$, if each slope $\frac{\partial h}{\partial \vec l_j}$ is a multiple of an outer normal of $\Delta$ and the multiplicities of slopes that correspond to a given side $\partial \Delta_i$ in this way sum to the integral length of $\partial \Delta_i$.

\subsection{Moduli and families of parametrized tropical curves.}
\label{subsec: moduli and families of parametrized tropical curves}

We denote by $M_{g,n, \nabla}^{\trop}$ the moduli space of parametrized tropical curves of genus $g$, degree $\nabla$, and with $n$ contracted legs. We always assume that the first $n$ legs $l_1,\dotsc, l_n$ are the contracted ones.
The space $M_{g,n, \nabla}^{\trop}$ is a generalized polyhedral complex. It is stratified by subsets $M_{[\Theta]}$, that are indexed by combinatorial types $\Theta$ with the fixed invariants $g, \nabla,$ and $n$. If $\GG$ is the underlying graph of a combinatorial type $\Theta$, then $M_{[\Theta]} = M_{\Theta}/\mathrm{Aut}(\Theta)$, where $M_{\Theta}$ is a polyhedron in $\RR^{|E(\GG)|} \times N_\RR^{|V(\GG)|}$ parametrizing tropical curves $h\: \Gamma \to N_{\RR}$ of type $\Theta$. The $e$-coordinate for an edge $e \in E(\GG)$ is the length $\ell(e)$ of $e$; the $v$-coordinates for a vertex $v \in V(\GG)$ are the coordinates of $h(v) \in N_\RR$.

Let $\Lambda$ be a tropical curve without loops. By a \emph{family} of parametrized tropical curves over $\Lambda$ we mean a continuous family of curves $h_q \colon \Gamma_q \to N_\RR$ for $q\in\Lambda$ such that the following holds:
    \begin{enumerate}
        \item The degree and the number of contracted legs are the same for all fibers. 
        \item Along the interior of an edge/leg $e$ of $\Lambda$, the combinatorial type of the fibers is constant.
        \item If a vertex $v \in V(\Lambda)$ is adjacent to an edge/leg $e$ of $\Lambda$, there is a fixed contraction of weighted graphs with marked legs $\varphi_{\vec e}\colon \GG_e \to \GG_v$ of the underlying graphs of the fibers. 
        \item The length of an edge $\gamma \in E(\GG_e)$ in $\Gamma_q$ along an edge/leg $e$ of $\Lambda$ is given by an integral affine function. Here we identify edges over an adjacent vertex of $e$ with those over $e^\circ$ via $\varphi_{\vec e}$, setting their length to zero in case they get contracted.
        \item Similarly, the coordinates of the images $h(u)$ for a vertex $u$ of $\GG_e$ in $h(\Gamma_q)$ along an edge/leg $e$ of $\Lambda$ is given by an integral affine function, where we again identify vertices over the interior of $e$ with those over vertices adjacent to $e$ via $\varphi_{\vec e}$.
    \end{enumerate}
    Any family of parametrized tropical curves $h\:\Gamma_\Lambda \to N_\RR$ induces a piecewise integral affine map $\alpha\: \Lambda \to M_{g, n, \nabla}^{\trop}$ by sending $q \in \Lambda$ to the point parametrizing the isomorphism class of the fiber $h_q\: \Gamma_q \to N_\RR$. Furthermore, $\alpha$ lifts to an integral affine map from the interior of each edge/leg $e$ of $\Lambda$ to the corresponding stratum $M_\Theta$. We refer the reader to \cite[\S 3.1.3, 3.1.4]{CHT20a} for formal definitions.

\subsection{Valued fields}
For a valued field $K$, we denote by $K^0$ the ring of integers, by $K^{00}\subset K^0$ the maximal ideal, and by $\widetilde K$ the residue field. As a general convention, we will similarly indicate when working over $K^0$ and $\widetilde K$ by adding a superscript $0$, respectively an overlining tilde; for example, $X^0$ for a model over $K^0$ of a curve $X \to \Spec(K)$ and $\widetilde X$ for its reduction.

\section{Realizable tropical curves and \texorpdfstring{$h$}{h}-transverse polygons}\label{sec:tropical curves}

In this section and the next, we will work over the algebraic closure $K$ of a complete, discretely valued field, with valuation $\val\:K\to \RR\cup\{\infty\}$.

\subsection{$h$-transverse polygons}
	
	The toric surfaces $S_\Delta$ that we consider in this paper are the ones with $h$-transverse polygon $\Delta$. 
	We restrict ourselves to $h$-transverse polygons since the theory of floor decomposed tropical curves has been developed only for this case in \cite{BM08}. 

	\begin{defn}\label{defn:htransverse}
	A polygon $\Delta \subset M_{\RR} = \RR^2$ is called $h$-{\em transverse}, if its normal vectors all have integral or infinite slopes with respect to the $x$-coordinate.
	\end{defn}
	Clearly, whether $\Delta$ is $h$-transverse depends only on its normal fan; thus the notion encodes a property of the surface $S_\Delta$ and does not depend on the polarization $\mathscr L_\Delta$. On the other hand, it does depend on a choice of coordinates, i.e., on the identification $M\cong\ZZ^2$. For the results we will discuss, it, of course, suffices that there is one choice of coordinates, for which $\Delta$ is $h$-transverse.
	
	More explicitly, $h$-transversality of a polygon $\Delta$ is characterized by the following. Denote by $y_0 := \min_{q\in\Delta}y_q$ the minimal $y$-coordinate of lattice points in $\Delta$. Set $y_i := y_0 + i$ with $i \in \NN$ up to $y_k := \max_{q\in\Delta} y_q$, and denote by $Y_i$ the horizontal line given by $y = y_i$. Then $Y_0$ and $Y_k$ intersect $\Delta$ in a bounded interval, possibly of length $0$. Now, $\Delta$ is $h$-transverse if and only if the lines $Y_i$ for $0 \leq i \leq k$ intersect the boundary of $\Delta$ at two lattice points; or in other words, the intervals $Y_i \cap \Delta$ are lattice intervals for all $i$. See Figure~\ref{fig:htransverse1} for an illustration. We denote by $\mathtt w(\Delta)$ the {\it width} of $\Delta$, that is, the maximum length of the intersections $Y_i \cap \Delta$. If they are not of length $0$, we call the lattice intervals $Y_0 \cap \Delta$ and $Y_k \cap \Delta$ the \emph{horizontal sides} of $\Delta$.
	
	\begin{ex}\label{ex:h-transverse polygon}
	Many smooth toric surfaces are induced by $h$-transverse polygons, for example, $\PP^2$, Hirzebruch surfaces, toric del Pezzo surfaces, and similarly blow-ups of Hirzebruch surfaces at any subset of the $4$ torus invariant points. An interesting class of examples of singular polarized toric surfaces induced by $h$-transverse polygons are the surfaces associated to the so-called {\em kites} introduced in \cite{lang2020monodromy}; see also \cite[\S 1.1]{LT20}. These are lattice polygons with vertices $(0,0), (k, \pm 1)$ and $(k + k', 0)$ for non-negative integers $k, k'$ with $k' \geq k$ and $k' > 0$ (the choice of coordinates differs from {\em loc. cit.}), see Figure~\ref{fig:htransverse1}. Some of these surfaces admit reducible Severi varieties, and we will discuss them in more detail in Proposition~\ref{prop:low characteristic counterexample} and Example~\ref{ex:kites irreduciblity} below.
	\end{ex}
	
    \begin{ex}\label{ex:(non)htr}
    Let $S_\Delta$ be the weighted projective plane with $\Delta$ given by vertices $(0,1), (1,-1)$ and $(6,-1)$; see  Figure~\ref{fig:htransverse1}. Then $\Delta$ is not $h$-transverse for any choice of coordinates. Indeed, if there was such a choice of coordinates, then the $x$-coordinate would give a non-zero linear function on $N$, that would assign each of the $3$ primitive vectors along the rays of the dual fan of $\Delta$ values $\pm 1$ or $0$. The above conditions on linear functions on $N$ correspond to $3$ triples of parallel lines in $M$, and one checks that for $\Delta$ as above no $3$ have a common point of intersection other than the origin; hence $\Delta$ is not $h$-transverse for any choice of coordinates. Note also that the same holds for any $\Delta'$ whose dual fan is a refinement of the dual fan of $\Delta$. Thus, an appropriate toric blow up $S_{\Delta'}$ of $S_{\Delta}$ gives a smooth toric surface that is not $h$-transverse for any choice of coordinates.
    \end{ex}

	\tikzset{every picture/.style={line width=0.75pt}}
    \begin{figure}[ht]	
    \begin{tikzpicture}[x=0.8pt,y=0.8pt,yscale=-0.8,xscale=0.8]
    \import{./}{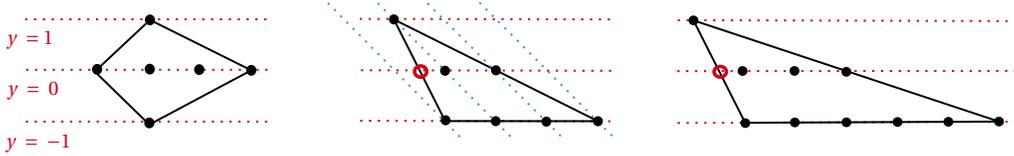}
    \end{tikzpicture}		
    \caption{The kite $(k,k')=(1,2)$ and two triangles. The right triangle is not $h$-transverse in any coordinate system, cf. Example~\ref{ex:(non)htr}, while the middle triangle is $h$-transverse with respect to the $x$-axis parallel to the blue lines, but not with respect to the usual $x$-axis.}
    \label{fig:htransverse1}
    \end{figure}
    	
    \subsection{Floor decomposed curves}\label{subsubsec:bound multiplicities} 
    Recall that a \emph{floor decomposed tropical curve} is defined to be a pa\-ra\-me\-tri\-zed tropical curve $h \colon \Gamma \to N_{\RR}=\RR^2$, in which the $x$-coordinate of the slope of each edge/leg is either $0$ or $\pm 1$. In the former case, a non-contracted edge/leg is called an {\it elevator}. If a connected component of the graph obtained from $\Gamma$ by removing the interiors of all elevators contains a non-contracted edge/leg, it is called a \emph{floor}.

    \begin{lem} \label{lem:max stretching factor}
    Let $h\: \Gamma \to N_{\RR}$ be a floor decomposed tropical curve of degree $\nabla$ dual to an $h$-transverse polygon $\Delta$. Then the multiplicity of any elevator $e\in\overline E(\Gamma)$ is bounded above by the width $\mathtt w(\Delta)$.
    \end{lem}

    \begin{proof}
    Recall that the multiplicity $\m(h,e)$ is defined to be the integral length of $\frac{\partial h}{\partial \vec e }$. Therefore, for embedded tropical curves, this bound follows from Legendre duality. In general, we argue as follows. For any $m \in M$, consider the parametrized tropical curve $h_m\:\Gamma\to\RR$ given by the composition of $h$ with the linear functional $m \: N_\RR \to \RR$. Let $d_m$ be the sum of multiplicities of the legs of $\Gamma$ having positive $h_m$-slope. If $p \in \RR$ is a point that is not an image of a vertex of $\Gamma$, then by balancing, the number of preimages $h_m^{-1}(p)$ counted with edge multiplicities equals $d_m$. Thus, to show the claim, it suffices to find an $m \in M$ for which $d_m=\mathtt w(\Delta)$ and $h_m$ does not contract elevators of $\Gamma$.

    In order to do so, note first that we can inscribe $\Delta$ in an $h$-transverse integral parallelogram $\square \subset M_\RR$ with two horizontal sides and width $\mathtt w(\square)=\mathtt w(\Delta)$. We now claim that the primitive integral vector $m$ along a non-horizontal side of $\square$ and having positive $y$-component satisfies the assertion. Indeed, $h_m$ does not contract elevators, since the $y$-coordinate of $m$ cannot be $0$. To calculate the degree of $h_m$, pick a point $p < \min_{v \in V(\Gamma)}\left \{h_m(v) \right \}$ in $\RR$. All preimages $h_m^{-1}(p)$ are contained in the set of legs   $l_i$ of $\Gamma$, whose slopes in $N_{\RR}$ have negative inner product with $m$, i.e., $\left(\frac{\partial h}{\partial \vec l_i}, m \right) < 0$.   Let $\square_l^-$ (resp., $\square_r^-$) be the connected component of $\square \setminus \Delta$ in the bottom left (respectively, bottom right) of the parallelogram $\square$; both are (possibly empty, non-convex) polygons. Note that outer normals along $\square_l^-$ and $\square_r^-$ sum to zero, and $m$ has $y$-coordinate $1$ since $\Delta$ is $h$-transverse. Thus, the sum of multiplicities $\sum \m(h_m,l_i)$ over the legs $l_i$, whose slopes have negative inner product with $m$, equals the sum of the integral lengths of the bottoms of $\Delta$, $\square_l^-$ and $\square_r^-$. The latter is nothing but the length of the bottom of $\square$, i.e., $\mathtt w (\square)=\mathtt w (\Delta)$. Thus, for the tropical curve $h_m\: \Gamma \to \RR$, we have $d_m=\mathtt w(\Delta)$, as claimed.
    \end{proof}
    
    In our setup, a vertical segment connecting two floors in a floor decomposed tropical curve can contain several elevators. It will be convenient to give a name to such segments, which are special cases of elevators in the sense of \cite{BM08}:
    
    \begin{defn}\label{defn:basic_floor_to_floor}
    Let $h\:\Gamma \to N_{\RR}$ be a floor decomposed tropical curve and $E \subset \Gamma$ a subgraph such that $h(E)$ is a bounded interval with constant $x$-coordinate in $N_{\RR}$. We will call $E$ a {\it basic floor-to-floor elevator} if it satisfies the following properties:
	\begin{enumerate}
	    \item for any edge/leg $e$ of $\Gamma\setminus E$ adjacent to $E$, $h(e) \cap h(E)$ is one of the endpoints of $h(E)$,
	    \item for each of the two endpoints of $h(E)$ there is a unique vertex $u_i$ of $E$ that gets mapped to it, and
	    \item the $u_i$ have exactly $2$ adjacent edges/legs not contained in $E$ and their slopes have $x$-coordinates $1$ and $-1$.
	\end{enumerate}
    \end{defn}

	\subsection{Tropicalization and realizability} \label{subsec:local liftability}
	In this section, we give a local version of Speyer's well-spacedness condition for the realizability of elliptic curves \cite{S14}; see also \cite[Theorem~1.1]{K12}, \cite[Theorem~6.9]{BPR16} and \cite[Theorem~D]{R17} for other generalizations.
	
    If a parametrized curve $f\: X \to S_\Delta$ is defined over the valued field $K$, one can naturally construct its \emph{tropicalization}, that is, a parametrized tropical curve $h \: \Gamma \to N_{\RR}$. Parametrized tropical curves arising this way are called \emph{realizable}. The tropicalization construction is well-known and can be found, e.g., in \cite[\S 4.2]{CHT20a}. For the convenience of the reader, we recall it in the special case of a non-constant rational function $f$ on a smooth proper marked curve $X$. Notice that to turn $(X,f)$ into a parametrized curve $f\:X\to\PP^1$, one must first add the zeroes and poles of $f$ to the collection of marked points on $X$. We always do so when talking about tropicalizations of rational functions, without necessarily mentioning it. 
    
    So, let $f\:X\to\PP^1$ be a parametrized curve, and $X^0$ the stable model of $X$ over $\Spec(K^0)$. The underlying graph of $\Gamma:=\trop(X)$ is the dual graph of the central fiber $\widetilde X$, i.e., the vertices correspond to irreducible components of $\widetilde{X}$, the edges -- to nodes, the legs -- to marked points, and the natural incidence relation holds. The weight of a vertex is given by the geometric genus of the corresponding irreducible component. Finally, the length of an edge corresponding to a node $z$ is given by the valuation $\val(\lambda)$, where $\lambda\in K^{00}$ is such that \'etale locally at $z$, the total space of $X^0$ is given by $xy=\lambda$.
    It remains to define the map $h:=\trop(f)\colon \Gamma \to \RR$. Since $h$ is affine along edges and legs, it is sufficient to specify its values at vertices together with its slopes along legs, which are given as follows. For the irreducible component $\widetilde X_v$ of $\widetilde X$ corresponding to a vertex $v\in V(\Gamma)$, let $\lambda_v \in K^\times$ by such that $\lambda_v f$ is an invertible function at the generic point of $\widetilde X_v$. Then $h(v):=\val(\lambda_v)$. The slope of $h$ along a leg $l\in L(\Gamma)$ is defined to be the order of pole of $f$ at the corresponding marked point.
    
    Let $\lambda_v$ be as above. The \emph{scaled reduction} of $f$ at $\widetilde X_v$ with respect to $\lambda_{v}$ is defined to be the non-zero rational function $(\lambda_{v} f)|_{\widetilde X_v}$ on $\widetilde X_v$. Notice that although scaled reductions of $f$ at $\widetilde X_v$ depend on $\lambda_v$, their divisors do not. In fact, it is straightforward to check that the divisor of a scaled reduction is given by $-\sum_{\vec e\in\Star(v)}\m(h,e)p_e$, where $p_e\in \widetilde X_v$ is either the node of $\widetilde X$ or the reduction of the marked point corresponding to $e$. In particular, $h$ has integral slopes, satisfies the balancing condition, and therefore $(\Gamma,h)$ is a parametrized tropical curve, cf. \cite[Lemma~2.23]{Tyo12}.
    
    We say that a parametrized tropical curve $h\:\Gamma\to N_\RR$ is {\em weakly faithful} at a vertex $v\in V(\Gamma)$ if $h$ does not contract the star of $v$ to a point. If $(\Gamma,h)$ is the tropicalization of a parametrized curve, then being not weakly faithful means that the scaled reductions of all monomial functions at $\widetilde X_v$ are constant.

	\begin{lem} \label{lem:tropical modification}
	Let $f$ be a non-constant rational function on a smooth proper marked curve $X$, and $v$ a vertex of $\trop(X)$. Then there is $\mu \in K$ such that $\trop(f - \mu)$ is weakly faithful at $v$.
	\end{lem}
	
	\begin{proof}
	Let $\eta$ be a $K$-point of $X$ at which $f$ is defined and such that the reduction of $\eta$ lies on the component $\widetilde X_v$ of the reduction of $X$ corresponding to the vertex $v$. Set $\mu \coloneqq f(\eta) \in K$. Since $f - \mu$ has a zero at $\eta$, every scaled reduction of $f - \mu$ at $\widetilde X_v$ needs to vanish at the reduction of $\eta$. However, since $f$ is non-constant, every scaled reduction of $f - \mu$ is non-zero. Therefore, the scaled reductions of $f-\mu$ at $\widetilde X_v$ are non-constant, and $\trop(f-\mu)$ is weakly faithful at $v$.
	\end{proof}
	
    A \emph{cycle} of a graph $\GG$ is a connected subgraph of $\GG$ in which each vertex has valence $2$. A cycle $O$ of a parametrized tropical curve $h\: \Gamma \to N_{\RR}$ is a subgraph of $\Gamma$ induced by a cycle of the underlying graph $\GG$, possibly containing additional legs that get contracted by $h$.
	
	\tikzset{every picture/.style={line width=0.75pt}}
    \begin{figure}[ht]	
    \begin{tikzpicture}[x=0.8pt,y=0.8pt,yscale=-0.7,xscale=0.7]
    \import{./}{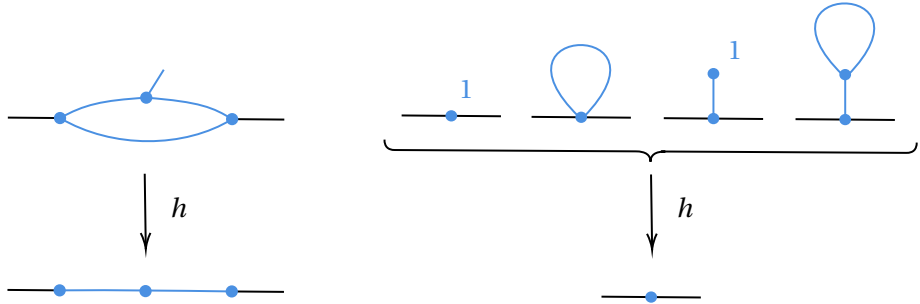}
    \end{tikzpicture}		
    \caption{From left to right: a flattened cycle, the two versions of an elliptic component, and the two versions of a contracted elliptic tail.}
    \label{fig:flattenedcycle}
    \end{figure}
	
	\begin{defn} \label{defn:flattened cycle}
	Let $h\: \Gamma \to N_\RR$ be a parametrized tropical curve. We give the following names to certain subgraphs $O$ of $\Gamma$  (see Figure~\ref{fig:flattenedcycle}):
	\begin{enumerate}
	    \item $O$ is called a \emph{flattened cycle} if it is a cycle, whose image in $N_{\RR}$ is a bounded interval, and such that for each endpoint of $h(O)$ there is exactly one vertex of $O$ mapped to it;
	    \item $O$ is called an \emph{elliptic component} if it is either a single vertex of weight $1$ or a single vertex of weight $0$ with an adjacent loop (which necessarily gets contracted by $h$);
	    \item $O$ is called a \emph{contracted elliptic tail} if $h(O)$ is a point, and $O$ contains $2$ vertices $v_1$ and $v_2$ connected by a single edge $e$, such that $v_1$ has weight zero, and all edges adjacent to $v_2$ are contained in $O$. Furthermore, $v_2$ has either weight $1$ and is adjacent only to $e$, or it has weight $0$ and is adjacent to an additional loop.
	\end{enumerate}
	\end{defn}
    
	\begin{prop} \label{prop:weight one midpoint}
    Let $h\:\Gamma \to N_{\RR}$ be a realizable tropical curve. Assume that $E \subset \Gamma$ is a subgraph that for some choice of coordinates satisfies the three axioms of Definition~\ref{defn:basic_floor_to_floor} for a basic floor-to-floor elevator. Suppose furthermore, that $O \subset E$ is a subgraph which is either a flattened cycle, an elliptic component, or a contracted elliptic tail. Finally, suppose that $\overline{E \setminus O}$ consists of at most two connected components, each of which is a weightless subgraph of\; $\Gamma$ that maps injectively to $h(E)$ away from contracted legs.
    Then either $h(O) = h(E)$ or the complement of $h(O)$ in $h(E)$ consists of two intervals of the same length.
	\end{prop}
	
	\begin{proof}
    Since $h\: \Gamma \to N_{\RR}$ is realizable, there is a parametrized curve $f\: X \to S_\Delta$, whose tropicalization is $h\: \Gamma \to N_{\RR}$. Let $\widetilde X$ be the central fiber of the stable model of $X$, and let $\widetilde X_E$ be the subcurve of $\widetilde X$ that corresponds to vertices in $E$; similarly, let $\widetilde X_{u_1}$ and $\widetilde X_{u_2}$ be the components corresponding to the (unique) vertices $u_1, u_2 \in E$ that get mapped to the endpoints of $h(E)$.

    Without loss of generality we may assume that $E$ satisfies the three axioms of Definition~\ref{defn:basic_floor_to_floor} for the standard coordinates in $N=\ZZ^2$. In particular, $h(E)$ is vertical. Set $m := (1,0) \in M=\ZZ^2$. Since $h(E)$ is vertical and the only adjacent edges/legs to $E$ whose slope has non-zero $x$-coordinate are by assumption adjacent to $u_1$ or $u_2$, we have that any scaled reduction of $f^*(x^m)$ at components of $\widetilde X_{E}$ different from the $\widetilde X_{u_i}$ has no zeroes or poles -- thus, it needs to be constant. In particular, no poles of $f^*(x^m)$ on $X$ specialize to components of $\widetilde X_{E}$ different from the $\widetilde X_{u_i}$. On the other hand, and again by assumption, for each $u_i$ there are exactly two adjacent edges/legs whose slope has non-zero $x$-coordinate, and that slope is $\pm 1$. Accordingly, any scaled reduction of $f^*(x^m)$ at $\widetilde X_{u_i}$ has a simple pole and a simple zero on $\widetilde X_{u_i}$ -- thus, it necessarily has degree $1$.

    Case 1: $O$ \textit{ is an elliptic component or a contracted elliptic tail.} Let $u$ be the (unique) vertex in $O$ that either has weight $1$, or is adjacent to a loop. Let $\widetilde X_u$ be the component of $\widetilde X$ corresponding to $u$. Note first that we must have $u \neq u_i$ since there can be no rational function of degree $1$ on a curve of arithmetic genus $1$. By Lemma~\ref{lem:tropical modification}, we may choose $\mu \in K$ such that any scaled reduction of $G \coloneqq f^*(x^m) - \mu$ at $\widetilde X_u$ is non-constant.
    The rational functions $f^*(x^m)$ and $G$ may have different zeroes on $X$, and to consider the tropicalization of $G$, we add the new zeroes as marked points to $X$. Note, however, that the poles of $f^*(x^m)$ on $X$ remain unaffected when subtracting $\mu$.

    Let $L \subset E$ be the subgraph of $E$ spanned by the edges of $E$ that do not get contracted by $h$. By the assumptions on $E$, we can naturally identify $L$ with a bounded line segment. Let $\gamma_i$ denote the edge in $L$ adjacent to $u_i$ and let $p_i \in \widetilde X_{u_i}$ denote the point corresponding to $\gamma_i$. Let $\lambda_{m, E} \in K$ be such that $\lambda_{m,E} f^*(x^m)$ is regular and invertible at generic points of $\widetilde X_E$; such a $\lambda_{m, E}$ exists since $h(E)$ has constant $x$-coordinate and $m = (1,0)$. The corresponding scaled reduction $(\lambda_{m,E} f^*(x^m))|_{\widetilde X_E}$ has, as explained above, constant value $\widetilde c$ along components of $\widetilde X_E$ different from the $\widetilde X_{u_i}$, and has degree $1$ on the $\widetilde X_{u_i}$. Furthermore, the restriction of $\lambda_{m, E} G$ to $\widetilde X_E$ is $(\lambda_{m,E} f^*(x^m))|_{\widetilde X_E} - \widetilde c$. In particular, $\lambda_{m, E} G$ is still a non-zero rational function on the $\widetilde X_{u_i}$ and we get $\trop(G)(u_1)=\trop(G)(u_2) = \val(\lambda_{m,E})$. Furthermore, the scaled reduction $(\lambda_{m, E} G)|_{\widetilde X_{u_i}}$ has degree $1$ with a simple zero at $p_i$; thus $\trop(G)$ has outgoing slope $-1$ at $u_i$ along $\gamma_i$.
    
    We can now deduce the claim from the properties of $\trop(G)$, see Figure~\ref{fig:tropg1} for an illustration. Let $u' \in L$ be the unique vertex with $h(u') = h(u)$; that is, if $O$ is an elliptic component we set $u':=u$, and if it is a contracted elliptic tail, then $u$ and $u'$ are the two vertices of $O$. Recall first that $\trop(G)$ is harmonic, no poles of $G$ on $X$ specialize to components of $\widetilde X_E$ different from the $\widetilde X_{u_i}$, and the outgoing slope of $\trop(G)$ at $u_i$ along $\gamma_i$ is $-1$. We conclude
    that the sum $s_v:=\sum_{\vec \gamma \in \Star(v) \cap L} \frac{\partial \trop(G)}{\partial \vec \gamma}$ of outgoing slopes of $\trop(G)$ at any $v\in L\setminus\{u_1, u_2\}$ along edges of $L$ is non-negative.
    It follows that $\trop(G)$ is convex along $L$, considered as a line segment.  In particular, the outgoing slope of $\trop(G)$ along the two edges of $L$ at $u'$ are both at most 1.
    On the other hand, since $\widetilde X_u$ has arithmetic genus $1$ and any scaled reduction $\widetilde G_{u}$ of $G$ at $\widetilde X_u$ is, by the construction of $G$, not constant, the negative part of $\mathrm{div}(\widetilde G_{u})$ needs to have multiplicity at least $2$; this means that $\trop(G)$ needs to have positive outgoing slopes at $u$ that sum to at least $2$. Since no poles of $G$ specialize to $\widetilde X_u$ or $\widetilde X_{u'}$, it follows that $\trop(G)$ needs to have outgoing slopes at $u'$ along edges contained in $L$ that sum to at least $2$.
    Combining these observations, we get that $\trop(G)$ has outgoing slope $1$ at $u'$ along each of the $2$ adjacent edges in $L$. Therefore, by convexity of $\trop(G)$,
    it is affine on $L$ except at $u'$, and its slope on both connected components of $L\setminus \{u'\}$ is $1$.
    Since $\trop(G)(u_1)=\trop(G)(u_2)$, this implies that $u'$ is the midpoint of $L$; balancing of $h$ then gives that also $h(u) = h(u')$ is the midpoint of $h(E) = h(L)$, as claimed.
    	
	\tikzset{every picture/.style={line width=0.75pt}}
    \begin{figure}[ht]	
    \begin{tikzpicture}[x=0.8pt,y=0.8pt,yscale=-0.7,xscale=0.7]
    \import{./}{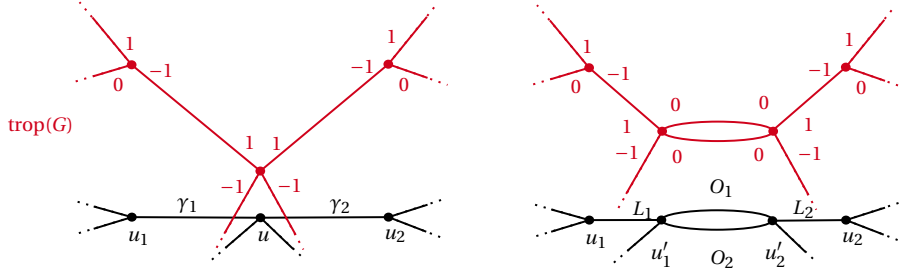}
    \end{tikzpicture}		
    \caption{In red, the graph of $\trop(G)$, where the numbers indicate outgoing slopes. On the left, Case 1, where $O$ contains a single vertex of weight $1$, $u = u'$, and $L$ contains only the two edges $\gamma_1$ and $\gamma_2$; if there are contracted legs contained in $E$, $L$ may contain more edges. On the right, Case 2, with $\trop(G)$ constant along both $O_i$; in general, it is non-constant along at most one of the $O_i$.}
    \label{fig:tropg1}
    \end{figure}

    Case 2: $O$ \textit{is a flattened cycle.}
	As in Case 1, let $u_1, u_2$ be the vertices that get mapped to endpoints of $h(E)$, and let $u_1'$ and $u_2'$ be the vertices of $O$ that get mapped to the endpoints of $h(O)$. If $u_1 = u_1'$ and $u_2 = u_2'$, there is nothing to show. So suppose $u_1 \neq u_1'$ and we first claim that then also $u_2 \neq u_2'$. Assume to the contrary that $u_2 = u_2'$. Since $O$ is a cycle and $u_1 \neq u_1'$, the closure of  $\widetilde X_E \setminus \left(\widetilde X_{u_1} \cup \widetilde X_{u_2}\right)$ intersects $\widetilde X_{u_2}$ in two points $p_1, p_2$. Since a scaled reduction $(\lambda_{m,E} f^*(x^m))|_{\widetilde X_E}$ as in Case 1 has constant value $\widetilde c$ along components of $\widetilde X_E$ different from the $\widetilde X_{u_i}$, the scaled reduction $(\lambda_{m,E} f^*(x^m))|_{\widetilde X_{u_2}}$ at $\widetilde X_{u_2}$ needs to achieve the same value $\widetilde c$ at both intersection points $p_1$ and $p_2$; since the scaled reduction has degree $1$ on $\widetilde X_{u_2}$, this is not possible, and hence $u_2 \neq u_2'$.
	
	The remaining argument is similar to Case 1 and we omit some of the details.
	Denote by $L$ the subgraph of $E$ spanned by edges that do not get contracted by $h$; this time, we may identify $L$ with two line segments $L_i$ between $u_i$ and $u_i'$ and two line segments $O_1$ and $O_2$ between $u_1'$ and $u_2'$.
    By Lemma~\ref{lem:tropical modification}, we may choose $\mu$ such that the scaled reduction of $G \coloneqq f^*(x^m) - \mu$ at $\widetilde X_{u_1'}$ is not constant.
    Arguing as in Case 1, $\trop(G)(u_1) = \trop(G)(u_2)$, $\trop(G)$ has outgoing slope $-1$ at $u_i$ along $L_i$, and $\trop(G)$ is convex on both line segments $L_1 \cup O_i \cup L_2$. It follows that $\trop(G)$ needs to be constant along one of the $O_i$; thus $\trop(G)(u_1') = \trop (G)(u_2')$. Since $\trop(G)$ is not constant at $u_1'$, it furthermore follows that $\trop(G)$ is linear along $L_1$. Combining these observations, we get $\ell(L_1) \leq \ell(L_2)$ for the lengths of the $L_i$. 
    Repeating the argument for $G'$ with scaled reduction at $\widetilde X_{u_2'}$ non-constant gives $\ell(L_2) \le \ell(L_1)$, and therefore $\ell(L_1) = \ell(L_2)$ as claimed.
	\end{proof}
	
	\begin{rem}
	In the proof of Proposition~\ref{prop:weight one midpoint} we considered the tropicalization of the function $G = f^*(x^m) - \mu$, and one may view the obtained function as an additional tropical coordinate (cf. Figure~\ref{fig:tropg1}). This is a version of a \emph{tropical modification}; we refer the interested reader to  \cite{K18} for a survey and further references.
	\end{rem}
	
	\section{Local liftability results for families of tropical curves} \label{sec:liftability}
	In \cite{CHT20a} we introduced the tropicalization of one-parameter families of parametrized cur\-ves and studied the properties of the induced moduli maps $\alpha\: \Lambda \to M_{g,n, \nabla}^{\trop}$ to the tropical moduli space over certain strata, called nice strata and simple walls. This was then used to locally lift the family of tropical curves, an essential ingredient in the degeneration argument for $\PP^2$. For the convenience of the reader, we recall this construction in \S~\ref{subsec:tropicalization for one-parameter families} and the properties of $\alpha$ in \S~\ref{subsec:simple walls and nice strata}.
	
For the remainder of this section, we fix a family of parametrized curves $f\colon \cX \dashrightarrow S_{\Delta}$ such that $\cX \to B$ admits a split stable model $\cX^0 \to B^0$ over $\Spec(K^0)$, where $B^0$ is a prestable model of the projective curve with marked points $(B, \tau_\bullet)$, cf. \S~\ref{sec:notterm}. We assume, in addition, that the irreducible components of the reduction $\widetilde B$ of $B^0$ are smooth.

	 \subsection{Tropicalization for one-parameter families.}\label{subsec:tropicalization for one-parameter families}
    
    The \emph{tropicalization of the family of parame\-trized curves} $f\colon \cX \dashrightarrow S_{\Delta}$ with respect to a fixed model $\cX^0 \to B^0$ is a family of parametrized tropical curves $h\colon \Gamma_\Lambda \to N_{\RR}$ as in \S~\ref{subsec: moduli and families of parametrized tropical curves}. We now sketch its construction, and refer to \cite[\S~4]{CHT20a} for details. 
    
    The base $\Lambda$ of the tropical family is the tropicalization of the base curve $(B, \tau_{\bullet})$ with respect to the model $B^0$. For a vertex $v \in V(\Lambda)$ denote by $\widetilde B_v$ the irreducible component of the reduction $\widetilde B$ corresponding to $v$. Let $\eta \in B(K)$ be a $K$-point, whose reduction $s \coloneqq  \widetilde \eta$ is a non-special point of $\widetilde B_v$, i.e., $s$ is neither a node of $\widetilde B$, nor a marked point. Then the fiber $h_v\: \Gamma_v \to N_\RR$ of the family $h\colon \Gamma_\Lambda \to N_{\RR}$ over $v$ is defined to be the tropicalization of $f|_{\cX_\eta}\colon \cX_{\eta} \to S_\Delta$. As long as $s\in \widetilde B_v$ is non-special, one checks that the parametrized tropical curves obtained this way are isomorphic for different choices of $\eta$. Furthermore, the splitness of the model $\cX^0 \to B^0$ ensures that they can be identified in a {\em canonical} way. 
    More generally, if $e\in\overline{E}(\Lambda)$ and $q\in e$ is a $\val(K^\times)$-rational point, then we can refine the model $B^0$ of $B$, such that $q$ becomes a vertex in the tropicalization of the base with respect to the refined model. The fiber $h_q\: \Gamma_q \to N_\RR$ over $q$ now is defined as before. One can check that the fibers over the $\val(K^\times)$-rational points of $\Lambda$ satisfy the 5 axioms of \S~\ref{subsec: moduli and families of parametrized tropical curves}, and therefore can be extended uniquely by continuity to a family over the whole of $\Lambda$.
    
    To verify the first axiom, recall that the slopes of $h_q$ along the legs of $\Gamma_q$ are given by the order of pole along the marked points of the pull-backs of monomials $(f|_{\cX_\eta})^*(x^m)$. Since marked points of $\cX_\eta$ are restrictions of marked points $\sigma_\bullet$ of $\cX$, the slopes of $h_q$ along the legs of $\Gamma_q$ are globally determined by the order of pole of $f^*(x^m)$ along the $\sigma_\bullet$, and hence independent of $q$.
    
    If $s \in \widetilde B_v$ is a special point corresponding to an edge/leg $e$ of $\Lambda$, the underlying graph $\GG_e$ of fibers over $e^\circ$ is by construction the dual graph of the fiber $\widetilde \cX_s$ over $s$. In particular, it is constant along $e^\circ$, which verifies the second axiom. Furthermore, the family $\cX^0 \to B^0$ prescribes a one-parameter degeneration of stable curves in a neighbourhood of $s$ in $\widetilde B_v$. Such a degeneration naturally induces an edge contraction $\varphi_{\vec e}\: \GG_e \to \GG_v$ between the dual graphs, as required by the third axiom. It remains to check that the edge lengths $\ell(\gamma)$ and the images of vertices $h(u)$ in the fibers are given by integral affine functions along the edges/legs $e$ of $\Lambda$. 
    
    Let us show how to verify this in the case of $\ell(\gamma)$ over an edge $e$, as the remaining cases can be verified in a similar way. Consider the toroidal embedding $B \setminus \left ( \bigcup_i \tau_i \right ) \subset B^0$, whose boundary $\widetilde B \cup \left( \bigcup_i \tau_i \right)$ contains all points over which fibers of $\cX^0 \to B^0$ fail to be smooth. All models are defined over some discretely valued subfield $F\subset K$ with a uniformizer $\pi\in F^{00}$. Let $s \in \widetilde B_v \cap \widetilde B_w$ be the node corresponding to $e$. Then in an \'etale neighbourhood $U$ of $s$ we have $ab = \pi^{k_s}$, where $a = 0$ and $b= 0$ define the components $\widetilde B_v$ and $\widetilde B_w$, respectively. Thus, the length of $e$ in $\Lambda$ is given by $\ell_\Lambda(e) = k_s \val(\pi).$
    
    Let $z$ be the node of $\widetilde \cX_s$ corresponding to $\gamma$. Then, possibly after shrinking $U$, an \'etale neighbourhood of $z$ in the total space $\cX^0$ is given by $xy = g_z$ for some regular function $g_z \in \mathcal O_U(U)$. Let $\psi\: \Spec(F) \to B$ be an $F$-point with image $\eta$ and reduction $s$. Then the length of $\gamma$ in $\trop(\cX_\eta)$ satisfies $\ell(\gamma)=\val(\psi^*(g_z))$ by the tropicalization construction, cf. \S~\ref{subsec:local liftability}. Since $g_z$ vanishes only along $\widetilde B_v \cup \widetilde B_w$ in $U$, there are $k_a, k_b, k_\pi \in \NN$ such that $a^{-k_a} b^{-k_b} \pi^{-k_\pi} g_z$ is regular and invertible on $U$. Thus, $\val(\psi^*(g_z))=k_a\val(\psi^*(a))+k_b\val(\psi^*(b))+k_\pi\val(\pi).$
    Since $ab = \pi^{k_s}$, it follows that $\val(\psi^*(a))+\val(\psi^*(b))=k_s\val(\pi)$, and therefore 
    $$\val(\psi^*(g_z))=(k_a-k_b)\val(\psi^*(a))+k_b\ell_\Lambda(e)+k_\pi\val(\pi).$$
    By the construction, $\val(\psi^*a)$ is nothing but the distance from $\trop(\eta) \in e$ from the vertex $w$, and $k_a\ell_\Lambda(e)+k_\pi\val(\pi)$ and $k_b\ell_\Lambda(e)+k_\pi\val(\pi)$ are the lengths of $\gamma$ in $\GG_v$ and $\GG_w$, respectively, identified via the contraction maps described above. Thus, the length of $\gamma$ along $e$ is given by an integral affine function of slope $k_a - k_b$, as claimed.
    
    \begin{rem}\label{rem:slopes}
    Note that $k_a - k_b$ is the vanishing order at $s$ of the scaled reduction of $g_z$ at $\widetilde B_w$ since $(ab)^{-k_b} \pi^{-k_\pi} g_z$ is regular and invertible at the generic point of $\widetilde B_w$. Thus, the slope $\frac{\partial\ell(\gamma)}{\partial \vec e}$ is given by the vanishing order at $s$ of a scaled reduction of $g_z$ at $\widetilde B_w$. Similarly, for a leg $l$ of $\Lambda$ adjacent to $w$ and corresponding to a marked point $\tau \in B$, the slope $\frac{\partial\ell(\gamma)}{\partial \vec l}$ is the vanishing order of $g_z$ at $\tau$, which  coincides with the vanishing order at $\widetilde \tau$ of a scaled reduction of $g_z$ at $\widetilde B_w$.
    
    The slopes of $\frac{\partial h(u)}{\partial \vec e}$ can be described similarly. Let $s$ be a special point of $\widetilde B_w$ corresponding to an edge/leg $e$. Then a vertex $u$ of $\GG_e$ corresponds to an irreducible component $\widetilde \cX_{s,u}$ of the fiber $\widetilde \cX_s$. There is a unique irreducible component $\widetilde \cX_{u}$ of the surface $\widetilde \cX|_{\widetilde B_w}$ that contains $\widetilde \cX_{s,u}$. Then for $m \in M$, the slope $\frac{\partial h(u)(m)}{\partial \vec e}$ is the order of pole along $\widetilde \cX_{s,u}$ of a scaled reduction of $f^*(x^m)$ at $\widetilde \cX_{u}$, which in the case of a leg $e$ is nothing but the order of pole of $f^*(x^m)$ along the irreducible component of the fiber $\cX_\tau$ that contains $\widetilde \cX_{s,u}$.
    \end{rem}
	\subsection{Simple walls and nice strata} \label{subsec:simple walls and nice strata}
	
	Let $\alpha \: \Lambda \to M_{g, n, \nabla}^{\trop}$ be the map induced by tropicalizing a one-paramter family, $v$ be a vertex of $\Lambda$, and $\Theta$ the combinatorial type such that $\alpha(v)\in M_{[\Theta]}$. The map $\alpha$ is \emph{harmonic} at $v$ if $\alpha$ lifts to a map to $M_{\Theta}$ locally at $v$ and the slopes $\frac{\partial \alpha}{\partial \vec e}$ along outgoing edges/legs $\vec e$ in $\Star(v)$ sum to zero. The map $\alpha$ is \emph{locally combinatorially surjective} at $v$, if the image $\alpha(\Star(v))$ intersects all strata $M_{[\Theta']}$ adjacent to $M_{[\Theta]}$. The following is Case 1 of Step 3 in the proof of \cite[Theorem 4.6]{CHT20a}:
	
	\begin{lem} \label{lem: harmonic when contracted}
	Let $h\:\Gamma_\Lambda \to N_\RR$ be the tropicalization of a family of parametrized curves $\cX \dashrightarrow S_\Delta$ with respect to $\cX^0 \to B^0$. Let $v \in \Lambda$ be a vertex corresponding to an irreducible component $\widetilde B_v$ of the reduction $\widetilde B$. Suppose that the moduli map $\chi\: B^0 \to \overline M_{g,n + |\nabla|}$ induced by the family $\cX^0 \to B^0$ contracts $\widetilde B_v$. Then the induced map $\alpha\: \Lambda \to M_{g, n, \nabla}^{\trop}$ is harmonic at $v$.
	\end{lem}
	
	\begin{proof}[Sketch of the proof]
    If $\chi(\widetilde B_v) = p \in \overline M_{g,n + |\nabla|}$, then every fiber of $\cX^0 \to B^0$ over $\widetilde B_v$ has the same dual graph $\GG$. This implies that the combinatorial type $\Theta$ of $h\: \Gamma_\Lambda \to \Lambda$ is locally constant around $v$, and thus $\alpha$ lifts to a map $\Star(v) \to M_{\Theta}$. As explained in \S~\ref{subsec: moduli and families of parametrized tropical curves}, we may view $M_\Theta$ naturally as a subset of $\mathbb R^{|E(\GG)| + 2 |V(\GG)|}$, and harmonicity can be checked coordinate-wise. 
    
    We first consider the $\gamma$-coordinate for an edge $\gamma \in E(\GG)$. The locus in $\overline M_{g,n + |\nabla|}$ in which the node $z$ corresponding to $\gamma$ persists, is cut out locally at $p$ by a single function. Since we can pull this function back from $\overline M_{g,n + |\nabla|}$ to $B^0$, the node $z$ is given by $xy = g_z$ in the total space of $\cX^0$, where $g_z$ is a regular function in a neighborhood of $\widetilde B_v$ in $B^0$ (and not just in a neighborhood of a closed point of $\widetilde B_v$ as it is in general). By Remark~\ref{rem:slopes}, the slopes of $\alpha$ at $v$ in the $\gamma$-coordinate are the orders of vanishing of a scaled reduction of $g_z$ at $\widetilde B_v$, and hence their sum vanishes.
    
    Let $u \in V(\GG)$ be a vertex, and consider the slopes of $\alpha$ at $v$ in the $h(u)$-coordinates. Let $\widetilde \cX_u \to \widetilde B_v$ be the surface corresponding to $u$.  Since the model is split and fibers of $\widetilde \cX_u \to \widetilde B_v$ are isomorphic, we have (after a finite base change in $\widetilde B_v$) $\widetilde \cX_u \simeq \mathcal C_u \times \widetilde B_v$ for some curve $\mathcal C_u$. Fix a general section $\{c\} \times \widetilde B_v$.  By Remark~\ref{rem:slopes}, the slopes of $h(u)(m)$ for $m\in M$ are given by the orders of pole along the $\widetilde \cX_{s,u}\simeq \{s\} \times \widetilde B_v$ of a scaled reduction of $f^*(x^m)$ at $\widetilde \cX_u\simeq \mathcal C_u \times \widetilde B_v$ for the special points $s\in \widetilde B_v$. The latter coincide with the orders of pole of the restriction of the scaled reduction to $\{c\} \times \widetilde B_v$ at the points $\{c\} \times \{s\}$, and hence again have to sum up to zero for any $m\in M$.
 	\end{proof}
	
	\begin{rem} \label{rem: harmonic when weightless trivalent}
	Suppose in the setting of Lemma~\ref{lem: harmonic when contracted} that $v \in \Lambda$ is a vertex, such that the fiber $\Gamma_v$ has weightless and $3$-valent underlying graph $\GG$. Since the stratum $M_\GG$ of curves with dual graph $\GG$ in $\overline M_{g,n + |\nabla|}$ is a single point, the lemma in this case implies that $\alpha$ is harmonic at $v$. 
	\end{rem}
	
    Recall that a stratum $M_\Theta$ is called {\em nice} if it is regular (that is, of the expected dimension), and the underlying graph of $\Theta$ is weightless and $3$-valent;  $M_\Theta$ is {\em a simple wall} if it is regular, and the underlying graph is weightless and $3$-valent except for a unique $4$-valent vertex. Note that each simple wall $M_\Theta$ is contained in the closure of exactly $3$ strata, all of them nice. They correspond to the three ways of splitting the $4$-valent vertex into two $3$-valent vertices. The following lemma is a part of \cite[Theorem 4.6]{CHT20a} and applies in particular to simple walls and nice strata:

    \begin{lem}\label{lem:alpha at simple walls and nice strata}
    Let $h\:\Gamma_\Lambda \to N_\RR$ be the tropicalization of a family of parametrized curves $\cX \dashrightarrow S_\Delta$ with respect to $\cX^0 \to B^0$. Let $v \in \Lambda$ be a vertex with $\alpha(v) \in M_{[\Theta]}$, where the underlying graph of $\Theta$ is weightless and $3$-valent except for at most one $4$-valent vertex. Then the induced map $\alpha\: \Lambda \to M_{g,n,\nabla}^{\trop}$ is either harmonic or locally combinatorally surjective at $v$.
    \end{lem}
	
	\begin{proof}[Sketch of the proof]
	Consider the moduli map $\chi\: B^0 \to \overline M_{g,n + |\nabla|}$ induced by the family $\cX^0 \to B^0$. If $\chi$ contracts $\widetilde B_v$, then $\alpha$ is harmonic at $v$ by Lemma~\ref{lem: harmonic when contracted}. Otherwise, the underlying graph $\GG$ of $\Theta$ must contain a unique $4$-valent vertex, since strata $M_{\GG}$ in $\overline M_{g,n + |\nabla|}$ corresponding to weightless, $3$-valent dual graphs $\GG$ have dimension $0$. Thus, $M_{\GG}$ is $1$-dimensional, and $\chi$ maps $\widetilde B_v$ surjectively onto the closure $\overline M_{\GG}\subset \overline M_{g,n + |\nabla|}$. Now one notes that $M_{[\Theta]}$ is contained in the closure of $3$ strata in $M_{g,n,\nabla}^{\trop}$, which correspond to the $3$-splittings of the $4$-valent vertex, and $\overline M_{\GG}$ contains as its boundary the $3$ strata parametrizing curves with the corresponding dual graphs; hence local combinatorial surjectivity follows in this case.
	\end{proof}
	
	\subsection{Flattened cycles}
	The first case we need, that was not covered in \cite{CHT20a}, is a vertex $v$ of $\Lambda$, such that the fiber over $v$ contains two $4$-valent vertices that form the endpoints of a flattened cycle. In this case, we do not get local combinatorial surjectivity of $\alpha$ at $v$, as will become clear in the proof of the next lemma; instead, we get that $\alpha$ is either harmonic at $v$ or $\alpha(\Star(v))$ intersects precisely the maximal adjacent strata in $M_{g,n,\nabla}^{\trop}$.
	
	\begin{lem} \label{lem:alpha at flattened cycle}
	Let $h\: \Gamma_\Lambda \to N_\RR$ be the tropicalization of a family of parametrized curves $\cX \dashrightarrow S_\Delta$ with respect to $\cX^0 \to B^0$. Let $v \in \Lambda$ be a vertex with $\alpha(v) \in M_{[\Theta]}$, where the underlying graph of $\Theta$ is weightless and $3$-valent, except for the two endpoints of a flattened cycle, which are $4$-valent. Assume that for some choice of coordinates, the flattened cycle satisfies the three axioms of Definition~\ref{defn:basic_floor_to_floor} for a basic floor-to-floor elevator. Then either $\alpha$ is harmonic at $v$; or for every $M_{[\Theta']}$ whose closure contains $M_{[\Theta]}$ and where the underlying graph of $\Theta'$ is $3$-valent, there is an edge/leg $e$ in $\Star(v)$ such that $\alpha(e) \cap M_{[\Theta']} \neq \emptyset$.
	\end{lem}
	
	\begin{proof}
     Let $\chi\: B^0 \to \overline M_{g, n + |\nabla|}$ be the moduli map induced by the family $\cX^0 \to B^0$. If $\chi$ contracts $\widetilde B_v$, then $\alpha$ is harmonic at $v$ by Lemma~\ref{lem: harmonic when contracted}. Thus, we assume $\chi$ does not contract $\widetilde B_v$.
    Let $u, u'$ be the endpoints of the flattened cycle $E$ in $\Theta$; that is, $h(E)$ is the line segment between $h(u)$ and $h(u')$.
    Any combinatorial type $\Theta'$ as in the claim of the lemma is obtained by splitting $u$ and $u'$ each into two $3$-valent vertices.

    Since $\cX^0 \to B^0$ is split, there are irreducible components $\widetilde \cX_u \to \widetilde B_v$ and $\widetilde \cX_{u'} \to \widetilde B_v$ of $\widetilde \cX$ that restrict over non-special points of $\widetilde B_v$ to the irreducible component of the fiber corresponding to $u$ and $u'$, respectively. Furthermore, both $\widetilde \cX_u \to \widetilde B_v$ and $\widetilde \cX_{u'} \to \widetilde B_v$ have $4$ marked points, corresponding to the $4$ adjacent edges/legs to $u$ and $u'$, respectively. Thus, we get maps $\chi_u\: \widetilde B_v \to \overline M_{0,4}$ and $\chi_{u'}\: \widetilde B_v \to \overline M_{0,4}$ induced by the two families. Since $\chi$ does not contract $\widetilde B_v$ and a non-special fiber of $\cX^0 \to B^0$ over $\widetilde B_v$ has moduli only at the components corresponding to $u$ and $u'$, at least one of the maps $\chi_u$ and $\chi_{u'}$ is non-constant, and hence surjective. We assume this to be the case for $\chi_u$.

    Surjectivity of $\chi_u$ implies that there is an edge or leg $e\in\Star(v)$ such that $\alpha(e) \cap M_{[\Theta'']} \neq \emptyset$ for any $\Theta''$ obtained by splitting $u$: each splitting of $u$ corresponds to a boundary point of $\overline M_{0,4}$. Let $\gamma_1, \gamma_2$ be the two edges adjacent to $u$ that are contained in $E$ and $\gamma_1', \gamma_2'$ the two edges adjacent to $u'$ contained in $E$ (possibly identical with the $\gamma_i$); see Figure \ref{fig:split}.
    If $\gamma_1$ and $\gamma_2$ are not adjacent to the same vertex of the underlying graph of $\Theta''$, also $u'$ has to split into two $3$-valent vertices in $\Theta''$ by balancing, and the claim of the lemma follows for such $\Theta''$.

\tikzset{every picture/.style={line width=0.75pt}}
\begin{figure}[ht]	
\begin{tikzpicture}[x=0.45pt,y=0.45pt,yscale=-0.9,xscale=0.9]
\import{./}{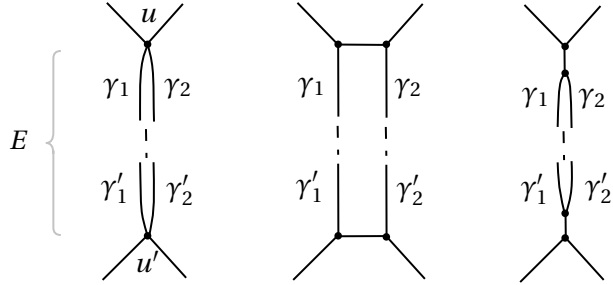}
\end{tikzpicture}		
\caption{On the left, the flattened cycle in a curve of combinatorial type $\Theta$. In the middle, a splitting of $u$ in which $\gamma_1$ and $\gamma_2$ are not adjacent to the same vertex. On the right, the splitting in which they remain adjacent.}
\label{fig:split}
\end{figure}

    Thus, it remains to show the following. If $\gamma_1, \gamma_2$ are adjacent to the same vertex in the underlying graph of $\Theta''$, then the vertex $u'$ necessarily splits in $\Theta''$. If not, then the curves of type $\Theta''$ are not realizable by Proposition~\ref{prop:weight one midpoint}, and therefore can not appear in the tropicalization of a family of parametrized curves.
    \end{proof}

    \subsection{Elliptic components and contracted elliptic tails.}
	We next describe, how to develop a loop or a contracted edge from a $2$-valent, weight $1$ vertex $u$ in the tropicalization $\Gamma_\Lambda \to N_{\RR}$ of a family of parametrized curves. If $\Theta$ is the combinatorial type containing such a vertex $u$, we will consider the following combinatorial types obtained from $\Theta$ (see also Definition~\ref{defn:flattened cycle} and Figure~\ref{fig:flattenedcycle}): let $\Theta'$ be obtained from $\Theta$ by developing a loop based at $u$, replacing $u$ with a weightless vertex $u'$; let $\Theta''$ be obtained from $\Theta$ by developing a contracted edge adjacent to a $1$-valent vertex $u''$ of weight $1$; and let $\Theta'''$ be obtained from $\Theta''$ by developing a loop at $u''$, replacing $u''$ with a weightless vertex.
	
	\begin{lem}\label{lem:alpha at weight one vertex}
	Let $\Gamma_\Lambda \to N_\RR$ be the tropicalization of a family of parametrized curves $\cX \dashrightarrow S_\Delta$ with respect to $\cX^0 \to B^0$. Let $v \in \Lambda$ be a vertex with $\alpha(v) \in M_{[\Theta]}$, where $\Theta$ is $3$-valent and weightless, except for one $2$-valent vertex $u$ of weight $1$. 
	Then, with $\Theta', \Theta''$ and $\Theta'''$ as above the following holds.
	\begin{enumerate}
	    \item Suppose $\mathrm{char}(\widetilde K)$ does not divide the multiplicity of $h$ along the two edges adjacent to $u$. Then $\alpha$ is harmonic at $v$, or there is an edge/leg $e$ in $\Star(v)$ such that $\alpha(e^\circ) \subset M_{[\Theta']}$. 
	    \item Suppose the multiplicity of $h$ along the two edges adjacent to $u$ is a power of $\mathrm{char}(\widetilde K)$. Then $\alpha$ is harmonic at $v$, or there is an edge/leg $e$ in $\Star(v)$ such that $\alpha(e^\circ) \subset M_{[\Theta'']}$.
	    \item  Let $w \in \Lambda$ be a vertex. If $\alpha(w) \in M_{[\Theta']}$ or $\alpha(w) \in M_{[\Theta''']}$, then $\alpha$ is harmonic at $w$. If $\alpha(w) \in M_{[\Theta'']}$, then $\alpha$ is harmonic at $w$, or there is an edge/leg $e$ in $\Star(w)$ such that $\alpha(e^\circ) \subset M_{[\Theta''']}$.
	\end{enumerate}
	\end{lem}
	
	\begin{proof}
	As in the proof of Lemma~\ref{lem:alpha at flattened cycle}, let $\widetilde \cX_u$ be the irreducible component of $\widetilde \cX$ corresponding to $u$. Let $\sigma_1$ and $\sigma_2$ be the two sections of $\widetilde \cX_u \to \widetilde B_v$ corresponding to the two edges $\gamma_1$ and $\gamma_2$ adjacent to $u$ (as before, they are sections since $\cX^0 \to B^0$ is split). Let $\vec \gamma_i$ be the orientation of $\gamma_i$ away from $u$; then  by balancing we have $\frac{\partial h}{\partial \vec \gamma_1} = - \frac{\partial h}{\partial \vec \gamma_{2}}$.
	Let $n \in N$ be a primitive vector for the line spanned by $\frac{\partial h}{\partial \vec \gamma_1}$; that is, $\frac{\partial h}{\partial \vec \gamma_i}$ is $\pm n$ times the multiplicity $\kappa \coloneqq \m(h,\gamma_i)$ of $h$ along the $\gamma_i$. Let $m \in M$ be a primitve lattice element such that $(m,n) = 0$, where $(\cdot,\cdot)$ is the pairing $M \times N \to \ZZ$ defining $M$ and $N$ as dual lattices.
	
	We consider the rational function  
	\[
	G \coloneqq \left(\lambda_{m, u} f^*(x^m)\right)|_{\widetilde \cX_u}
	\]
	on $\widetilde \cX_u$, where $\lambda_{m,u} \in K$ is chosen such that $G$ is regular and invertible at the generic point of $\widetilde \cX_u$. Over non-special points $s \in \widetilde B_v$, $G$ restricts to a rational function $G|_{\widetilde \cX_{s, u}}$ on the fiber $\widetilde \cX_{s, u}$. We have that $G|_{\widetilde \cX_{s, u}}$ is regular and invertible away from the image of $s$ under the two sections $\sigma_1$ and $\sigma_2$. More precisely, we get
	\[
	\mathrm{div}(G|_{\widetilde \cX_{s, u}}) = \kappa \sigma_1(s) - \kappa \sigma_2(s).
	\]
	Viewing $\widetilde \cX_{s,u}$ as an elliptic curve with identity $\sigma_2(s)$, we thus get that $\sigma_1(s)$ is a point of order dividing $\kappa$. Possibly replacing $\kappa$ by a number that divides it, we may assume that $\sigma_1(s)$ has order $\kappa$. Note that this replacement preserves the divisibility condition -- $\mathrm{char}(\widetilde K) \not | \, \kappa$ or $\kappa = \mathrm{char}(\widetilde K)^l$ -- in the assumptions of the lemma. Since $\sigma_2(s) \neq \sigma_1(s)$ and $u$ is of weight $1$, we must have $\kappa > 1$. Furthermore, this order is constant along smooth, non-supersingular fibers of $\widetilde \cX_u \to \widetilde B_v$.
	Let $\chi\: B^0 \to \overline M_{g, n + |\nabla|}$ as before be the moduli map induced by the family $\cX^0 \to B^0$. If $\chi$ contracts $\widetilde B_v$, then $\alpha$ is harmonic at $v$ by Lemma~\ref{lem: harmonic when contracted}. Thus we assume $\chi$ does not contract $\widetilde B_v$. Since the underlying graph of $\Theta$ is $3$-valent and weightless except for $u$, this implies that the family $\widetilde \cX_u \to \widetilde B_v$ contains non-isomorphic fibers.
	
	For the first claim assume $\mathrm{char}(\widetilde K) \not | \, \kappa$. We will show that in this case $\widetilde \cX_u \to \widetilde B_v$ necessarily contains a singular, irreducible fiber, which immediately implies that there is an $e \in \Star(v)$ satisfying the assertion of the first claim.
	Let $X_1(\kappa) \to \Spec(\widetilde K)$ be the compactification of the modular curve $Y_1(\kappa)$ as constructed by Deligne and Rapoport \cite{DR72}, where $Y_1(\kappa)$ parametrizes elliptic curves together with a point of exact order $\kappa$. Let $\mathcal C_1(\kappa) \to X_1(\kappa)$ be the universal curve. The fibers of $\mathcal C_1(\kappa)$ over boundary points of $X_1(\kappa)$ are N\'eron polygons; they are semistable curves with dual graph a weightless cycle with an additional group structure. In the case of $X_1(\kappa)$, the N\'eron polygons that appear as fibers of $\mathcal C_1(\kappa)$ are precisely the ones for which the length of the dual graph divides $\kappa$ (this is an immediate consequence of \cite[Construction IV.4.14, p. 224]{DR72}). In particular, there are fibers of $\mathcal C_1(\kappa)$ that are N\'eron polygons of length $1$, or in other words, $\mathcal C_1(\kappa)$ contains a singular irreducible fiber.
	
	By what we showed above, $\widetilde \cX_u \to \widetilde B_v$ induces a rational moduli map $\chi_{u, \kappa}\: \widetilde B_v \dashrightarrow X_1(\kappa)$, defined at points where the fibers of $\widetilde \cX_u \to \widetilde B_v$ are smooth. Furthermore, over this locus $\widetilde \cX_u \to \widetilde B_v$ is the pullback of $\mathcal C_1(\kappa)$. Recall next that because $\mathrm{char}(\widetilde K) \not | \, \kappa$, $X_1(\kappa)$ is irreducible by \cite[Corollaire IV.5.6, p. 227]{DR72}. Since $\widetilde \cX_u \to \widetilde B_v$ contains non-isomorphic fibers and $X_1(\kappa)$ has dimension $1$, this implies that the moduli map $\chi_{u, \kappa}$ is dominant. The curve $\widetilde B_v$ is smooth and $X_1(\kappa)$ is proper. Thus, after replacing $\widetilde B_v$ with a finite covering, the map $\chi_{u, \kappa}$ extends to the whole curve. To ease the notation, we will assume that $\chi_{u, \kappa}$ is defined on all of $\widetilde B_v$. 
	
	Let $b\in \widetilde B_v$ be any point such that the fiber $C_b$ of $\mathcal C_1(\kappa) \to X_1(\kappa)$ over $\chi_{u, \kappa}(b)$ is singular and irreducible. Since $\mathcal C_1(\kappa) \to X_1(\kappa)$ is a family of stable curves over a neighborhood of $\chi_{u, \kappa}(b)$, so is its pullback to $\widetilde B_v$. But the family $\widetilde \cX_u \to \widetilde B_v$ is also stable, and coincides with the pullback over a dense open subset of $\widetilde B_v$. Therefore, its fiber over $b$ is $C_b$ by the uniqueness of the stable model. Hence, the family $\widetilde \cX_u \to \widetilde B_v$ contains irreducible singular fibers as claimed. 
	
	For the second claim, assume now that $\kappa = \mathrm{char}(\widetilde K)^l$. As before, we obtain a rational map $\chi_{u, \kappa}\: \widetilde B_v \dashrightarrow Y_1(\kappa)$, defined at points where the fibers of $\widetilde \cX_u \to \widetilde B_v$ are smooth; it dominates an irreducible component of $Y_1(\kappa)$. Since for any irreducible component of $Y_1(\kappa)$ the forgetful morphism to $M_{1,1}$ is surjective, there is in particular $s \in \widetilde B_v$ such that $\chi_{u, \kappa}(s)$ parametrizes a supersingular elliptic curve. Since $\kappa = \mathrm{char}(\widetilde K)^l$, the only point of order $\kappa$ on such a supersingular curve is the origin. Thus the images of the two points $\sigma_1(s)$ and $\sigma_2(s)$ in the universal family over $Y_1(\kappa)$ need to coincide. Since $\widetilde \cX_u \to \widetilde B_v$ is a family of stable curves and the $\sigma_i$ sections, we thus get that the fiber of $\widetilde \cX_u \to \widetilde B_v$ over $s$ has two irreducible components: a rational curve containing $\sigma_1(s)$ and $\sigma_2(s)$, which is attached to an elliptic curve along a single node. This is clearly equivalent to the second claim. 

    For the third claim, note first that the underlying graph of $\Theta'''$ is weightless and $3$-valent; thus if $\alpha(w) \in M_{[\Theta''']}$, then $\alpha$ is harmonic at $w$ by Remark~\ref{rem: harmonic when weightless trivalent}. We next consider $\Theta'$. Recall that $u'$ denotes the $4$-valent vertex of the underlying graph of $\Theta'$ that is adjacent to the loop, and let $w \in \Lambda$ be a vertex such that $\alpha(w) \in M_{[\Theta']}$. As before, let $\widetilde \cX_{u'} \to \widetilde B_w$ be the irreducible component of $\widetilde \cX \to \widetilde B$ corresponding to $u'$. Up to an isomorphism, we may assume that the general fiber of $\widetilde \cX_{u'} \to \widetilde B_w$ is $\PP^1$ with $0$ and $\infty$ identified, and one of its two marked points is $1$. Arguing as for the first claim, we get that the second marked point needs to be a $\kappa$-th root of unity in $\PP^1$. In particular, all fibers of $\widetilde \cX_{u'} \to \widetilde B_w$ are isomorphic. Since general fibers of $(\widetilde \cX)|_{\widetilde B_w}$ have by assumption moduli only at $\widetilde \cX_{u'}$, this implies that the global moduli map $\chi\: B^0 \to \overline M_{g, n + |\nabla|}$ contracts $\widetilde B_w$. Thus $\alpha$ is harmonic at $w$ by Lemma~\ref{lem: harmonic when contracted}. 
    Finally, we consider $\Theta''$. Recall that $u''$ denotes the vertex of weight and valence $1$. As before, $u''$ corresponds to a family $\widetilde \cX_{u''} \to \widetilde B_w$ of elliptic curves and induces a moduli map $\chi_{u''}\: \widetilde B_w \to \overline M_{1,1}$. If this map has constant image, the global moduli map $\chi\: B^0 \to \overline M_{g, n + |\nabla|}$ contracts $\widetilde B_w$, since general fibers of $(\widetilde \cX)|_{\widetilde B_w}$ have by assumption moduli only at $\widetilde \cX_{u''}$. In this case $\alpha$ is harmonic at $w$ by Lemma~\ref{lem: harmonic when contracted}. Otherwise, $\chi_{u''}$ is dominant, and the closure of it's image contains the boundary point of $\overline M_{1,1}$. Again by uniqueness of the stable model, this implies that the family $\widetilde \cX_{u''} \to \widetilde B_w$ contains a singular, irreducible fiber, and the claim follows.
 	\end{proof}

\section{Degeneration of curves on polarized toric surfaces}\label{sec:degeneration}
In this section, we prove the main result, Theorem~\ref{thm:main thm}. 
We work over an arbitrary algebraically closed field $K$ and consider $h$-transverse polygons $\Delta$. 
For a line bundle $\mathscr L$ on $S_\Delta$, recall
that the \emph{Severi varieties} $V_{g,\mathscr L}$ and $V_{g,\CL}^{\irr}$ are the following loci of the linear system $|\mathscr L|$:
		\[
	V_{g,\CL}= \left\{[C] \in |\mathscr L|\, | \, C \text{ is reduced, torically transverse, and }\, \pg(C) = g \right\},
		\]
where $\pg(C)$ denotes the geometric genus of $C$, and
{\em torically transverse} means that $C$ contains no zero-dimensional orbits of the torus action. Similarly,
	\[
	V_{g,\CL}^{\irr} = \left\{[C] \in V_{g,\mathscr L} | \, C \text{ is irreducible}\right\}.
		\]
When $\CL=\CL_\Delta$, we will write $V_{g,\Delta}:=V_{g,\mathscr L_\Delta}$ and $V^\irr_{g,\Delta}:=V^\irr_{g,\mathscr L_\Delta}$.

\begin{rem}
The definition of Severi varieties given here is slightly more restrictive than the one in \cite[Definition 2.1]{CHT20a} since we require the curves to avoid {\em all}  zero-dimensional orbits in $S_\Delta$. This restriction is necessary to obtain well-behaved tangency profiles, which will be introduced below. One can check that the closures of the loci defined in the two different ways coincide, thus the modification is harmless for our purposes. 
\end{rem}

\subsection{Tangency conditions}
In Theorem~\ref{thm:main thm} we will allow for some prescribed tangencies with the toric boundary of $S_\Delta$, and we establish the setup necessary for this statement next. Denote by $D_1,\dotsc,D_s\subset S_\Delta$ for $i=1,2,\dotsc,s$ the closures of the codimension-one orbits in $S_\Delta$ corresponding to the sides $\partial \Delta_i$ of $\Delta$. A \emph{tangency profile of} $(S_\Delta, \mathscr L)$ is a collection of multisets 
\[
\underline d:=\left(\left\{d_{i,j}\right\}_{1\leq j\leq a_i}\right)_{1\leq i\leq s}
\]
with $a_i \in \ZZ_{\geq 0}$ and  $d_{i, j} \in \ZZ_{>0}$ such that $D_i \cdot \mathscr L = \sum_{1\leq j\leq a_i}d_{i,j}$.
We say that a pa\-ra\-me\-tri\-zed curve $f\: X \to S_\Delta$ with $f(X)$ torically transverse has tangency profile $\underline d$, if for all $i$ we have $f^*(D_i)= \Sigma_{1\leq j\leq a_i} d_{i, j} p_{i, j}$ as an effective divisor on $X$, where $p_{i,j}$ are distinct points.

\begin{defn}
For a tangency profile $\underline d$ of $(S_\Delta, \mathscr L)$, we define $V_{g,\underline d,\mathscr L}$ (respectively, $V^\irr_{g,\underline d,\mathscr L}$) to be the locus of $V_{g,\mathscr L}$ (respectively, $V^\irr_{g,\mathscr L}$) consisting of curves $C$ such that the normalization $C^\nu\rightarrow C$ induces a parametrized curve $f\colon C^\nu\rightarrow S_\Delta$ with tangency profile $\underline d$.
\end{defn}

As before, we set $V_{g,\underline d,\Delta}:=V_{g,\underline d, \mathscr L_\Delta}$ and $V^\irr_{g,\underline d, \Delta}:=V^\irr_{g,\underline d,\mathscr L_\Delta}$ for the line bundle $\mathscr L_\Delta$ associated to $\Delta$. Notice that in this case $D_i \cdot \mathscr L = \sum_{1\leq j\leq a_i}d_{i,j}$ is the integral length of $\partial \Delta_i$.
We say that a tangency profile $\underline d$ is \textit{trivial} on $D_i$ if $d_{i,j}=1$ for all $1\leq j\leq a_i$. If a tangency profile is trivial on all orbits $D_i$ we simply call it trivial, and denote it by $\underline d_0$; that is, $\underline d_0$ is the tangency profile for $(S_\Delta, \mathscr L)$ with $a_i = D_i \cdot \mathscr L$ and $d_{i,j} = 1$ for all $i, j$.

\begin{rem}\label{rem:generalized severi}
If $S_\Delta = \PP^2$ and $\underline d$ is trivial on two of the codimension-one  orbits, then both $V_{g,\underline d,\Delta}$ and $V^\irr_{g,\underline d, \Delta}$ are \textit{generalized Severi varieties} in the sense of \cite{caporaso1998counting}. On the other hand, in \emph{loc. cit}. the authors also discuss the loci obtained by fixing the point of given tangency with the toric boundary, which we do not consider here.
\end{rem}

\begin{rem}
Clearly we have $V^\irr_{g,\underline d_0,\mathscr L}\subset V^\irr_{g,\mathscr L}$. On the other hand, $V^\irr_{g,\mathscr L}\subset \overline V^\irr_{g, \underline d_0, \mathscr L}$ by \cite[Proposition 2.7]{CHT20a}.
Therefore $\overline V^\irr_{g,\mathscr L} = \overline V^\irr_{g,\underline d_0,\mathscr L}$ and we are free to switch between $V^\irr_{g,\underline d_0,\mathscr L}$ and $ V^\irr_{g,\mathscr L}$ in the study of degenerations of irreducible curves in $S_\Delta$.
\end{rem}
 
We write $|\underline d| :=\sum_ia_i$; that is, $|\underline d|$ denotes the number of distinct points in the preimage under $f$ of the boundary divisor of $S_\Delta$ for any parametrized curve $f \: X \to S_\Delta$ with tangency profile $\underline d$. 

\begin{prop}\label{prop:tangency}
Let $\underline d$ be a tangency profile of $(S_\Delta, \mathscr L)$. Then $V_{g,\underline d,\mathscr L}$ and $V^\irr_{g,\underline d,\mathscr L}$ are either empty or constructible subsets of $|\mathscr L|$ of pure dimension $|\underline d|+g-1$.
\end{prop}
\begin{proof}
By \cite[Lemma~2.6]{CHT20a}, the Severi variety $V_{g,\mathscr L}$ is a locally closed subset in the linear system $|\mathscr L|$, and by \cite[Proposition~2.7]{CHT20a}, it has pure dimension $-\mathscr L\cdot K_{S_\Delta}+g-1$. To estimate the codimension of $V_{g,\underline d,\mathscr L}$ in $V_{g,\mathscr L}$ we follow the approach of \cite{caporaso1998rational}. 

Let $V\subseteq V_{g,\mathscr L}$ be an irreducible component, $\eta\in V$ its generic point, $\overline{\eta}$ the corresponding geometric point, and $C_{\overline{\eta}}$ the corresponding curve. Since the geometric genus of $C_{\overline{\eta}}$ is $g$, it follows that there exists an irreducible finite covering $F_V\:U\to V$ such that the geometric genus of the fiber over $F^{-1}(\eta)$ is $g$. Therefore, the pullback $C_U\to U$ of the universal curve to $U$ is equinormalizable by \cite[Theorem~4.2]{CL06}, since normalization commutes with localization. Set $W$ to be the disjoint union of $U$'s for different irreducible components. Then $F\: W\to V_{g,\mathscr L}$ is a finite surjective map, and the pullback $C_W\to W$ of the universal curve to $W$ is equinormalizable.

Consider the relative Hilbert scheme of points $H\to W$ on the normalized family $C_W^\nu\to W$ of degree $\mathscr L\cdot K_{S_\Delta}$, and the natural section $\psi\:W\to H$ that associates to $w\in W$ the pullback of the boundary divisor $\sum D_i$ under the map $C_w^\nu\to C_w\to S_\Delta$. Since the (relative) Hilbert scheme of points on a smooth (relative) curve is smooth, the standard dimension-theoretic arguments provide a lower bound on the dimensions of the irreducible components of $V_{g,\underline d,\mathscr L}$ as follows. The locus $Z\subset H$ of subschemes of the form $\sum_{i,j}d_{i, j} p_{i, j}$ is locally closed of pure codimension $\sum_{i,j}\left(d_{i,j}-1\right)$. Therefore, $\psi^{-1}(Z)$ is locally closed and has codimension at most 
$$\sum_{i,j}\left(d_{i,j}-1\right)=-\mathscr L\cdot K_{S_\Delta}-|\underline d|$$
at any point. Finally, since $V_{g,\underline d,\mathscr L}$ is a union of components of $F\circ\psi^{-1}(Z)$ and $F$ is finite, it follows that $V_{g,\underline d,\mathscr L}$ is constructible and the dimension of any of its components is at least $$\left(-\mathscr L\cdot K_{S_\Delta}+g-1\right)-\left(-\mathscr L\cdot K_{S_\Delta}-|\underline d|\right)=|\underline d|+g-1.$$

On the other hand, let $V$ be an irreducible component of $V_{g,\underline d,\CL}$, and $[C] \in V$ a general point. Suppose $C$ has $m$ irreducible components $C_1,\dotsc, C_m$. Set $g_k:=\pg(C_k)$, $\CL_k:=\CO_{S_\Delta}(C_k)$, and let $\underline d_k$ be the tangency profile of the parametrized curve induced by normalizing $C_k$. Consider the map $\prod_{k=1}^m V_{g_k,\underline d_k, \CL_k}\to V_{g,\underline d,\CL}$. Its fiber over $[C]$ is finite, and the map is locally surjective. Thus, $\dim_{[C]}(V)=\sum_{k=1}^m\dim_{[C_k]}(V_{g_k,\underline d_k, \CL_k})$. However, $\dim_{[C_k]}(V_{g_k,\underline d_k, \CL_k})\le |\underline d_k|+g_k-1$ by \cite[Theorem~1.2~(1)]{Tyo13}, and hence
$$\dim_{[C]}(V)\le\sum_{k=1}^m\left(|\underline d_k|+g_k-1\right)=|\underline d|+g-1.$$		
Thus, $V$ has dimension $|\underline d|+g-1$, which completes the proof for $V_{g,\underline d,\mathscr L}$. The assertion for $V^\irr_{g,\underline d,\mathscr L}$ now follows too, since $V^\irr_{g,\underline d,\mathscr L}$ is a union of components of $V_{g,\underline d,\mathscr L}$.
\end{proof}

\subsection{The degeneration result}
We are now ready to state our main result. Denote by $\Delta^\circ$ the interior of the polygon $\Delta$ and recall that we denote by $\mathtt w(\Delta)$ the width of $\Delta$. Recall furthermore, that a smooth curve in the linear system $|\mathscr L_\Delta|$ has genus $\#|\Delta^\circ\cap M|$.

\begin{thm}\label{thm:main thm}
Let $\Delta$ be an $h$-transverse polygon and suppose $\mathrm{char}(K)=0$ or $\mathrm{char}(K)>\mathtt w(\Delta)/2$. Let $0 \leq g\leq \#|\Delta^\circ\cap M|$ be an integer and $\underline d$ a tangency profile of $(S_\Delta, \mathscr L_\Delta)$, which is trivial on toric divisors corresponding to non-horizontal sides of $\Delta$. Then $V^\irr_{0,\underline d,\Delta} \subset V$ for every irreducible component $V$ of $\overline V^\irr_{g,\underline d,\Delta}$. In particular, each irreducible component of $\overline V^\irr_{g,\Delta}$ contains $V^\irr_{0,\Delta}$.
\end{thm}

The proof of Theorem~\ref{thm:main thm} occupies much of the remainder of this section. It proceeds by induction on $g$. In Lemma~\ref{lem:irreducibility in rational case}, we prove that $V^\irr_{0,\underline d,\Delta}$ is non-empty and irreducible. To prove the induction step, we show that $V$ must contain an irreducible component of the Severi variety $V^\irr_{g-1,\Delta}$. To do so, we impose $\dim(V) - 1$ general point constraints on curves in $V$ and reduce its dimension to $1$. We consider a one-parameter family of parametrized curves $\cX \to B$ associated to an irreducible component of the obtained locus and analyze its tropicalization $\Gamma_\Lambda \to N_\RR$ by investigating the induced map $\alpha$ from $\Lambda$ to the moduli space of parametrized tropical curves.

Lemma~\ref{lem:local surjectivity} provides us with good control over the image of $\alpha$. In particular, it is used in Lemma~\ref{lem:reduce to two elevators in a cycle} to show that $V$ must contain a codimension-one locus of curves $[C]$, whose tropicalization contains a cycle $O$ satisfying combinatorial properties, that are tailored to be able to develop a contracted edge or loop from it. We complete the proof in \S~\ref{subsubsec: proof main thm}, by showing that for an appropriate choice of point constraints, $\Lambda$ necessarily contains a leg parametrizing tropical curves with a fixed image in $N_\RR$ and having a contracted edge or loop, whose length grows to infinity. Furthermore, one of the components of the complement of this edge/loop is a parametrized tropical curve of genus $g-1$ satisfying $\dim(V) - 1$ general point constraints. This allows us to prove that $V$ must contain a codimension-one locus of curves of genus $g-1$. 

To use tropicalizations, we have to work over a valued field. To this end, note that the statement of Theorem~\ref{thm:main thm} is compatible with base field extension. Hence, by passing to the algebraic closure of $K((t))$, we work in the proof over the algebraic closure of a complete DVR, whose residue field is algebraically closed and has characteristic either 0 or greater than $\mathtt w(\Delta)/2$.

\subsubsection{Point constraints} \label{subsubsec: point constraints}
Let $\nabla$ be a tropical degree dual to $\Delta$; see \S~\ref{sec:nottropcur}. Then $\nabla$ induces a tangency profile $\underline d$ of $(S_\Delta, \mathscr L_\Delta)$ as follows. For each $1\le i\le s$, let $a_i$ be the number of slopes in $\nabla$ corresponding to the side $\partial_i\Delta\subset\Delta$, and  $\left\{d_{i,1},\dotsc,d_{i,a_i}\right\}$ the multiset of multiplicities of these slopes. Vice versa, $\underline d$ determines a tropical degree $\nabla$ dual to $\Delta$ uniquely up-to ordering of the slopes. Notice also that if $f \: X \to S_\Delta$ is a parametrized curve of tangency profile $\underline d$, then the degree $\nabla$ of its tropicalization induces the tangency profile $\underline d$ by the very construction of the tropicalization. We set once and for all
\[
n:=|\underline d|+g-1=|\nabla|+g-1. 
\]

Denote by $\ev\: M^\trop_{g,k,\nabla}\to N_\RR^k$ the evaluation map 
\[
\ev\colon (\Gamma, h)\mapsto (h(l_1),\dotsc,h(l_k)),
\]
where $h\colon \Gamma\rightarrow N_\mathbb R$ is a parametrized tropical curve with  contracted legs $l_1,\dotsc,l_k$. 
If $\Theta$ is a combinatorial type with $k$ contracted legs, we define $\ev_\Theta$ to be the composition $\overline{M}_{\Theta}\to M_{g, k, \nabla}^\trop\to N_{\RR}^{k}$. Then for $k$ points $q_1,\dotsc,q_k$ in $N_\mathbb R$, the preimage $\ev_\Theta^{-1}(q_1,\dotsc,q_k)$ is a polyhedron cut out in $\overline{M}_\Theta$ by an affine subspace.

\begin{defn}\label{defn:general position to a curve}
For $k \in \ZZ_{> 0}$ fix points $p_1,\dotsc,p_k\in S_\Delta$ and $q_1,\dotsc,q_k\in N_\RR$. Denote by $H_i\subset|\mathscr L_\Delta|$ the hyperplane parametrizing curves that pass through the point $p_i$.
\begin{enumerate} 
    \item Let $[C] \in V^\irr_{g,\underline d,\Delta}$ pass through all of the $p_i$'s. We say that the $p_i$'s are in \textit{general position with respect to} $[C]$, if the intersection $V^\irr_{g,\underline d,\Delta}\cap \left(\bigcap_{i=1}^k H_i\right)$ has dimension $n-k$ at $[C]$.
    \item Let $(\Gamma,h)\in M^\trop_{g,k,\nabla}$. We say that the $q_i$'s are in \textit{tropical general position with respect to} $(\Gamma,h)$ if the preimage $\ev^{-1}(q_1,\dotsc,q_k)$ has dimension $n-k$ at $(\Gamma,h)$.
\end{enumerate}
\end{defn}

\begin{rem}\label{rem:local general position}
Let $f\: X \to S_\Delta$ be a parametrized curve with $C := f(X)$ passing through points $p_1, \dotsc, p_k$ and $[C] \in V^\irr_{g,\underline d,\Delta}$. Assume that its tropicalization $(\Gamma, h)$ is weightless, $3$-valent, and immersed except at contracted legs. Set $q_i:=\trop(p_i)$. If $q_1,\dotsc,q_k$ are in general position with respect to $(\Gamma, h)$, then so are $p_1,\dotsc,p_k$ with respect to $[C]$.
Indeed, let $M_\Theta$ be the stratum of $M^\trop_{g,0,\nabla}$ that contains the curve obtained from $(\Gamma, h)$ by forgetting the contracted legs and stabilizing. Then $M_\Theta$ is maximal and has dimension $n$ by \cite[Proposition 2.23]{Mik05}.
Thus we may choose $n-k$ points $q_{k+1},\dotsc,q_{n}$ contained in $h(\Gamma)$, 
such that $(\Gamma, h)$ is an isolated point in $\ev^{-1}(q_1,\dotsc,q_{n})\subset M_{g,n,\nabla}^{\trop}$.
Let $p_{k+1},\dotsc,p_{n}\in C$ be $n-k$ points such that $\trop(p_i)=q_i$. We claim that then also $[C]$ is an isolated point in $V^\irr_{g,\underline d,\Delta}\cap (\bigcap_{i=1}^{n} H_i)$, and therefore $V^\irr_{g,\underline d,\Delta}\cap (\bigcap_{i=1}^k H_i)$ has dimension $n-k$ at $[C]$, since $V^\irr_{g,\underline d,\Delta}$ is equidimensional of dimension $n$. Indeed, any irreducible curve $B\subseteq V^\irr_{g,\underline d,\Delta}\cap (\bigcap_{i=1}^{n} H_i)$ containing $[C]$ gives rise to a $1$-dimensional family of parametrized tropical curves $\Gamma_\Lambda \to N_\RR$ (cf. Construction~\ref{constr:tropical family} below). The family $\Gamma_\Lambda \to N_\RR$ then induces a
map $\alpha\: \Lambda \to M^\trop_{g,n,\nabla}$ satisfying $(\Gamma, h) \in \alpha(\Lambda) \subset \ev^{-1}(q_1,\dotsc,q_{n})$. Since the curves over $B$ dominate $S_\Delta$, their tropicalizations cover a dense subset of $N_\mathbb R$ and $\alpha$ is not constant. Since $\Lambda$ and hence also $\alpha(\Lambda)$ is connected, this contradicts the fact that $(\Gamma, h)$ is an isolated point in $\ev^{-1}(q_1,\dotsc,q_{n})$.
\end{rem}

Let $V \subset \overline V^\irr_{g,\underline d,\Delta}$ be an irreducible component, $[C]\in V$ a curve, and $p_1,\dotsc,p_{n-1}\in S_\Delta$ points in general position with respect to $[C]$. Pick an irreducible component $Z\subseteq V\cap \left(\bigcap_{i=1}^{n-1} H_i\right)$ containing $[C]$, and let us briefly explain how to associate a family of parametrized curves to $Z$, which we then can study by investigating its tropicalization. We refer to Step 2 of the proof of \cite[Theorem 5.1]{CHT20a} for technical details.

\begin{constr}\label{constr:tropical family}
First, we pull back the tautological family of curves from $Z$ to the normalization $Z^\nu$. After replacing $Z^\nu$ with a finite covering, we may assume that the resulting family is generically equinormalizable by \cite[\S 5.10]{dJ96}. Normalizing its total space, and restricting the obtained family to an open dense subset of the base, we get a family $f \colon \CX \to S_\Delta$ of genus $g$ parametrized curves over the complement of finitely many points $\tau_\bullet$ in a smooth projective base curve $B$.
Next, after replacing $B$ with a finite covering, we may assume that the marked points $\sigma_\bullet$ of $\mathscr X$ are the points that are mapped to $\{p_i\}_{i=1}^{n-1}$ and to the boundary divisors of $S_\Delta$. Furthermore, we may assume that the family $\mathscr X \to B\setminus\tau_\bullet$ extends to a split stable model $\mathscr X^0 \to B^0$, where $B^0$ is a prestable model of $(B,\tau_\bullet)$ over $K^0$.
In particular, for $b \in B$ that gets mapped to a general $[C] \in Z$ under the natural projection $B\rightarrow Z$, the fiber $\cX_b$ of $\cX \to B$ is the normalization of $C$, equipped with the natural map to $S_\Delta$.
\end{constr}

Recall from \S~\ref{subsec:simple walls and nice strata} that we consider two types of regular strata $M_\Theta$ in $M_{g, n-1, \nabla}^\trop$: nice strata and simple walls. The former have the expected dimension $2n - 1$, and the latter $2n - 2$. Since both are automorphism free, we have $M_\Theta=M_{[\Theta]}$ and may consider $M_\Theta$ as a subset of $M^\trop_{g,n-1,\nabla}$. 

\begin{lem}\label{lem:local surjectivity}
Let $\alpha\colon \Lambda\rightarrow M^\trop_{g,n-1,\nabla}$ be the moduli map induced by the tropicalization of the family $f\: \cX \dashrightarrow S_\Delta$ given by imposing $n - 1$ general point constraints $p_i$ on $V$ as in Construction~\ref{constr:tropical family}. Set $q_i = \trop(p_i)$.
Then $\alpha$ is not constant and the following hold:
\begin{enumerate}
\item Suppose $\alpha$ maps a vertex $v\in V(\Lambda)$ to a simple wall $M_\Theta$, and $M_\Theta\cap\ev_\Theta^{-1}(q_1,\dotsc,q_{n-1})=\alpha(v)$ is a single point.
Then there exists a vertex $w\in V(\Lambda)$ such that $\alpha(w)=\alpha(v)$ and the map $\alpha$ is locally combinatorially surjective at $w$.

\item Suppose $M_{\Theta'}$ is a nice stratum, $\alpha(\Lambda)\cap M_{\Theta'}\ne\emptyset$ and $\dim \left(M_{\Theta'}\cap\ev_{\Theta'}^{-1}(q_1,\dotsc,q_{n-1})\right)=1$.
Then $M_{\Theta'}\cap\ev_{\Theta'}^{-1}(q_1,\dotsc,q_{n-1})=\alpha(\Lambda)\cap M_{\Theta'}$ is an interval, whose boundary in $\overline M_{\Theta'}$ is disjoint from $M_{\Theta'}$.
\end{enumerate}
\end{lem}
\begin{proof}
By construction, $\alpha$ is obtained by tropicalizing a family of parametrized curves $f\: \cX \dashrightarrow S_\Delta$ over a curve $B$.
Since $f$ is dominant onto $S_\Delta$, the image of $\Gamma_\Lambda \to N_\RR$ covers a dense subset of $N_\mathbb R$. Hence the image of $\Lambda$ in $M^\trop_{g,n,\nabla}$ is not constant.

(1)  Since $\alpha$ is not constant and $\Lambda$ connected, we can find $w\in V(\Lambda)$ such that $\alpha(w)=\alpha(v)$ and $\alpha$ is non-constant at $w$. Since $\alpha(\Lambda)\subset \ev^{-1}(q_1,\dotsc,q_{n-1})$, we have $\alpha(\Lambda)\cap M_\Theta=\alpha(v)$. It follows that $\alpha$ does not induce a map $\Star(w) \to M_\Theta$. By definition, $\alpha$ is thus not harmonic at $w$. Hence by Lemma~\ref{lem:alpha at simple walls and nice strata}, $\alpha$ is locally combinatorially surjective at $w$.

(2) Since $M_{\Theta'}$ is a nice stratum of $M^\trop_{g,n-1,\nabla}$, the underlying graph $\GG$ of $\Theta'$ in particular is weightless and $3$-valent. Thus by Remark~\ref{rem: harmonic when weightless trivalent}, $\alpha$ is harmonic at vertices of $\Lambda$ that map to $M_{\Theta'}$.
On the other hand, since 
$M_{\Theta'}\cap\ev_{\Theta'}^{-1}(q_1,\dotsc,q_{n-1})$ is the intersection of finitely many (open) half-spaces and hyperplanes, it is an interval. Since $\alpha(\Lambda)\cap M_{\Theta'}$ is contained in $M_{\Theta'}\cap\ev_{\Theta'}^{-1}(q_1,\dotsc,q_{n-1})$, the harmonicity of $\alpha$ shows that $M_{\Theta'}\cap\ev_{\Theta'}^{-1}(q_1,\dotsc,q_{n-1})=\alpha(\Lambda)\cap M_{\Theta'}$, and its boundary in $\overline M_{\Theta'}$ is disjoint from $M_{\Theta'}$. Compare to \cite[Lemma 3.9 (3)]{CHT20a}.
\end{proof}

\subsubsection{Cycles in floor decomposed curves} \label{subsubsec:cycles in floor decomposed curves} We specify some notation that is convenient for the later proofs. Given a stable, 3-valent, and floor decomposed curve $(\Gamma,h)$, and a cycle $O$ in $\Gamma$, we denote by $n(O)$ the number of basic floor-to-floor elevators of $\Gamma$ that are contained in $O$, and call it the \textit{vertical complexity} of $O$. Let $n(\Gamma)$ be the minimal vertical complexity among all cycles of $\Gamma$, which we call the \textit{vertical complexity} of $\Gamma$.
A point in $h(\Gamma)$ is called \textit{special} if it is the image of a vertex or a contracted leg on a floor.

Recall that $V$ denotes an irreducible component of $\overline V^\irr_{g,\underline d,\Delta}$. Since curves $[C]$ in $V^\irr_{g,\underline d,\Delta}$ are torically transverse, we can view the map $C^\nu \to S_\Delta$ from the normalization of $C$ as a parametrized curve by marking the preimage of the toric boundary. We denote by $\trop(C) = (\Gamma_C, h_C)$ the tropicalization of this parametrized curve. Let $V^{FD}\subset V \cap V^\irr_{g,\underline d,\Delta}$ 
be the locus of curves $[C]$ such that the tropicalization $\trop(C)$ is floor decomposed, weightless, $3$-valent, and immersed except at contracted legs (in particular, contained in a nice stratum in the moduli of parametrized tropical curves by \cite[Proposition 2.23]{Mik05}). Given $[C]\in V^{FD}$, we define the \textit{vertical complexity} of $C$ to be the vertical complexity of $\trop(C)$. 
Note that by the construction of $V^{FD}$, all basic floor-to-floor elevators of $\trop(C)$ have genus $0$, and $h_C$ is injective on any basic floor-to-floor elevator of $\Gamma_C$ away from contracted legs.
The following lemma provides the first part of the degeneration we construct in order to prove Theorem~\ref{thm:main thm}.

\begin{lem}\label{lem:reduce to two elevators in a cycle}
Suppose $V$ is a component of $\overline V^\irr_{g,\underline d,\Delta}$ with $g\geq 1$ and $\underline d$ is a tangency profile trivial on
all toric divisors of $S_\Delta$ corresponding to non-horizontal sides of $\Delta$ as in Theorem \ref{thm:main thm}. Then there exists $[C]\in V^{FD}$ whose tropicalization  $\trop(C) = (\Gamma_C,h_C)$
contains two elevators $E$ and $E'$ adjacent to the same floor $F$ and contained in the same cycle $O$ of $\Gamma_C$ such that either
\begin{enumerate}
\item 
$\frac{\partial h_C}{\partial \vec E}=-\frac{\partial h_C}{\partial \vec E'}$, where $\vec E$ and $\vec E'$ are oriented outwards from $F$; 
or
\item $O$ has vertical complexity 2.
\end{enumerate}
\end{lem}

\begin{proof}
We first show that $V^{FD}$ is not empty. Since $\Delta$ is $h$-transverse, for any $n$ vertically stretched points $q'_1,\dotsc,q'_{n}\subset N_\mathbb R$ that are in general position, all curves in $\ev^{-1}(q'_1,\dotsc,q'_{n})\subset M^\trop_{g,n,\nabla}$ are weightless, $3$-valent, floor decomposed and immersed by \cite[\S 5.1]{BM08}. Note that while in {\it loc. cit.} only tropical degrees $\nabla$ associated to trivial 
$\underline d$ are considered, the argument there works also in our case. 
For the coordinate-wise tropicalization map $\trop\: T \to N_\RR$ on the dense torus $T \subset S_\Delta$, the preimage $\trop^{-1}(q_i')$ is Zariski dense (though not algebraic) in $T$. Thus, we can choose $p'_i$ in $S_\Delta$, such that $\trop(p'_i)=q'_i$ and there are curves in $V \cap V^\irr_{g,\underline d,\Delta}$ that pass through all of the $p_i'$. Since any such curve is contained in $V^{FD}$, it follows in particular that $V^{FD}$ is not empty.

Pick $[C]\in V^{FD}$ with minimal vertical complexity; we will show that $C$ satisfies the claim of the lemma.
Let $\mathtt h$ be the height of $\Delta$, namely $\mathtt h=\max_{q\in\Delta}y_q-\min_{q\in\Delta}y_q$. Denote the floors of the tropicalization $\trop(C)$ from the bottom to the top by $F_1,\dotsc, F_\mathtt h$. Pick points $q_1,\dotsc,q_{n}\in N_\mathbb R$ such that (1) each floor and elevator of $\trop(C)$ contains exactly one of the $q_i$ in its image and (2) the preimage $x_i=h^{-1}_C(q_i)\in\Gamma_C$ is a single point contained in the interior of an edge or leg.

For each $1\leq i\leq n$, let $\Theta'_i$ be the combinatorial type in $M^\trop_{g,n-1,\nabla}$ obtained from $\trop(C)$ by adding one contracted leg at $x_j$ for each  $j\neq i$. Note that each contracted leg is either on a floor or adjacent to two elevators, and $M_{\Theta'_i}$ is a nice stratum. In this way, we can view $[\trop(C)]$ as a point in $M_i:=M_{\Theta'_i}\cap\ev_{\Theta'_i}^{-1}(q_1,\dotsc,\hat q_i,\dotsc,q_{n})$. 
Since $\Gamma_C\setminus \{x_i\}_{1\leq i\leq n}$ is a disjoint union of trees, each containing exactly one non-contracted leg, the points $(q_1,\dotsc,\hat q_i,\dotsc,q_{n})$ are in general position with respect to $\trop(C)$ 
and $\dim M_i=1$
(compare to \cite[Lemma 4.20]{Mik05}).

\textbf{Step 1. The reduction to the case of vertically stretched points.} For each $1\leq i\leq n$, pick $p_i\in C$ such that $\trop(p_i)=q_i$. It follows from Remark~\ref{rem:local general position} that, for each $i$, the points $\{p_j\}_{j\neq i}$ are in general position with respect to $[C]\in V$. Let $\alpha_i\colon \Lambda_i\rightarrow M^\trop_{g,n-1,\nabla}$ be induced by the family of parametrized tropical curves obtained from Construction~\ref{constr:tropical family} with respect to the $n-1$ points $(p_1,\dotsc,\hat p_i,\dotsc,p_{n})$.
We have $M_i\subset \alpha_i(\Lambda_i)$ by Lemma~\ref{lem:local surjectivity} (2), and the curves in $M_i$ are obtained from $\trop(C)$ by either moving the image of a floor vertically or moving the image of an elevator  horizontally within a suitable distance (until some two adjacent vertices collapse); see Figure~\ref{fig:perturb}.  Since $\trop(C)$ is floor decomposed, the curves in  $M_i$ have the same vertical complexity as $\trop(C)$.

\tikzset{every picture/.style={line width=0.75pt}}
\begin{figure}[ht]	
\begin{tikzpicture}[x=0.45pt,y=0.45pt,yscale=-0.9,xscale=0.9]
\import{./}{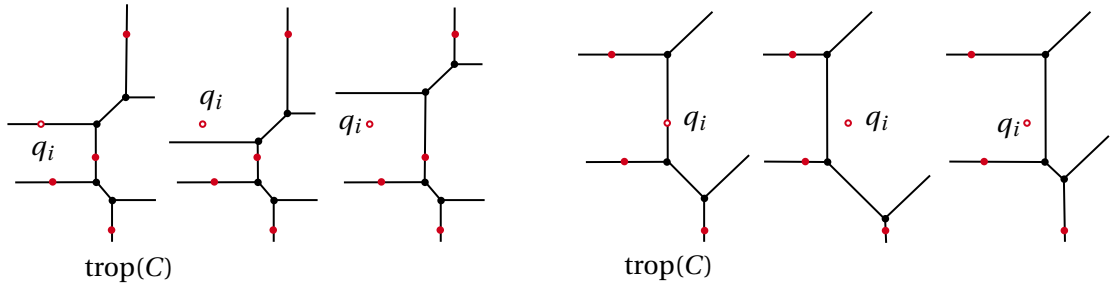}
\end{tikzpicture}		
\caption{The curves in $M_i$. The left is when $q_i$ is contained in the image of a floor of $\trop(C)$, while on the right $q_i$ is contained in the image of an elevator.}
\label{fig:perturb}
\end{figure}

Assume now that $q_i$ is contained in the image of $F_i$ for $1\leq i\leq \mathtt h$. We deform all the points $\{q_i\}_{1\leq i\leq n}$ so that they are vertically stretched and in general position with respect to a curve that has the same vertical complexity as $\trop(C)$, as follows. We start with the top floor $F_{\mathtt h}$. On any upward elevator adjacent to $F_{\mathtt h}$, we may move the point $q_j$ it contains up the elevator for any distance; that is, any point $q_j' = q_j + t (0, 1)$ with $t \in \RR_{\geq 0}$ is still contained in $h_C(\Gamma_C)$. By first moving the $q_j$ up an appropriate distance, we then can move up $F_{\mathtt h}$ any distance, without changing the combinatorial type of $\trop(C)$; that is, we consider the family of curves, starting with $\trop(C)$, that is cut out by replacing $q_{\mathtt h}$ with $q_{\mathtt h} + t (0,1)$ for varying $t \in \RR_{\geq 0}$. By definition, the curves obtained in this way are contained in $M_j$ or $M_{\mathtt h}$. Repeating this construction for each floor $F_i$, with $i$ going down from $\mathtt h$ to $1$, we see that we can increase the vertical distance between any two points in $\{q_i\}_{1\leq i\leq n}$ arbitrarily while preserving their horizontal distance. In particular, we can move the $\{q_i\}_{1\leq i\leq n}$ until they are vertically stretched.

Since the curves obtained by moving the $q_i$'s as above remain in $M_i$ and $M_i\subset \alpha_i(\Lambda_i)$, the parametrized tropical curve we obtain after stretching the $q_i$ is still the tropicalization of a curve $C'$ in $V^{FD}$. Furthermore, $C'$ has the same vertical complexity as $C$, and the image $h_{C'}(\Gamma_{C'})$ contains $n$ vertically stretched points $q_i$ that are in general position with respect to $\trop(C')$. To ease notation, we assume that $C$ already satisfies these properties.

\textbf{Step 2. The curve $C$ satisfies the assertion of the Lemma.} Recall that $q_i$ is contained in the image of $F_i$ for $1\leq i\leq \mathtt h$.
Let $O$ be a cycle in $\trop(C)$ with minimal vertical complexity. 
Let $F_k$ be the highest floor containing an edge in $O$ and pick an elevator $E\subseteq O$ that is adjacent to $F_k$. Let $F_s$ be the other floor
adjacent to $E$. Suppose $O$ is oriented from $E$ towards an edge in $F_s$, and let $E'\subseteq O$ be the first elevator adjacent to $F_s$ that appears after $E$; see Figure~\ref{fig:wc1}. 
Reordering $\{q_i\}_{i=1}^{n}$ if necessary, we may assume that $E$ contains $q_{n}$.
As before, pick $p_1,\dotsc,p_{n-1}\in C$ such that $\trop(p_i)=q_i$, and consider $\trop(C)$ as a curve in the family of parametrized tropical curves $h\:\Gamma_\Lambda \to N_\RR$ obtained from Construction~\ref{constr:tropical family} for the points  $p_1,\dotsc,p_{n-1}$.  Let $\Theta'$ be the combinatorial type of $\trop(C)$. Then $M_{\Theta'}$ is a nice stratum of $M_{g, n-1, \nabla}^\trop$ by assumption.

\tikzset{every picture/.style={line width=0.75pt}}
\begin{figure}[ht]	
\begin{tikzpicture}[x=0.5pt,y=0.5pt,yscale=-0.8,xscale=0.8]
\import{./}{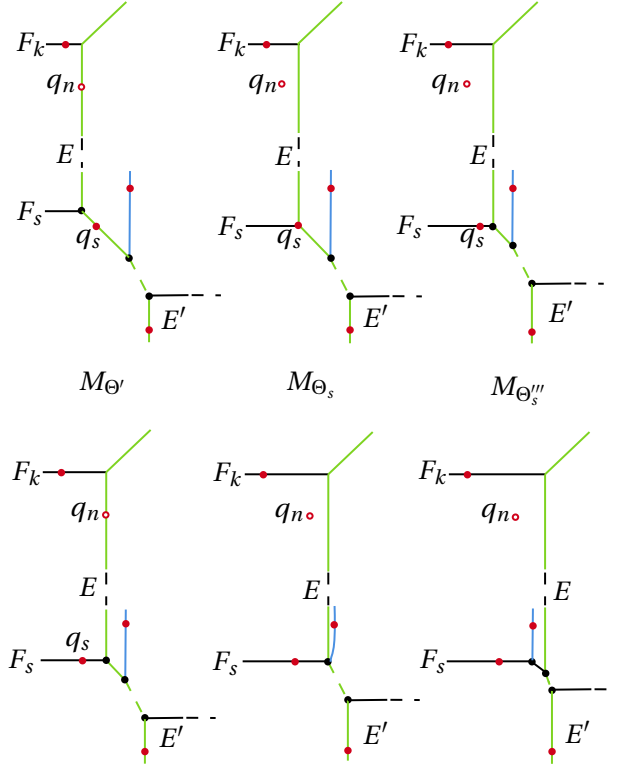}
\end{tikzpicture}		
\caption{Moving $E$ towards $E'$. The left column is $\trop(C)$. In the top row passing through $q_s$, in the bottom row through a vertex in $F_s$. The case of passing through $q_k$ or a vertex in $F_k$ is analogous. The edges in $O$ are colored green.}
\label{fig:wc1}
\end{figure}

We claim that, by moving $E$ towards $E'$ and passing through the special points on the way, the problem reduces to the case where there is no special point in $F_k\cup F_s$ with $x$-coordinate between the $x$-coordinates of $E$ and $E'$.
To prove this, assume without loss of generality that the $x$-coordinate of $E$ is less than that of $E'$, and there are exactly $r$ special points in $F_k\cup F_s$ with $x$-coordinates between those of $E$ and $E'$. We will proceed by induction on $r$. The base case is trivial, so we assume $r>0$. We will describe the inductive step only for the case that, among the $r$ special points, the one with $x$-coordinate closest to that of $E$ is the marked point $q_s \in F_s$, the image of a contracted leg $l_s$. The argument for the other cases is similar, see Figure~\ref{fig:wc1}.

Recall that $\alpha\colon  \Lambda\rightarrow M^\trop_{g,n-1,\nabla}$ is the map induced by the family of parametrized tropical curves $\Gamma_\Lambda \to N_\RR$. Moving $E$ horizontally towards $E'$, until its lower vertex hits $q_s$, we get a tropical curve represented by a boundary point of $M_{\Theta'}\cap\ev_{\Theta'}^{-1}(q_1,\dotsc,q_{n-1})$, which is contained in a simple wall $M_{\Theta_s}$ such that $E$ is adjacent to the leg $l_s$ contracted to $q_s$. By Lemma~\ref{lem:local surjectivity} (2), $\alpha(\Lambda)$ contains $M_{\Theta'}\cap\ev_{\Theta'}^{-1}(q_1,\dotsc,q_{n-1})$, hence also contains this boundary point. Let $\Theta'''_s\neq \Theta'$ be the other nice stratum in the star of $M_{\Theta_s}$ such that $E$ and the leg $l_s$ contracted to $q_s$ are not adjacent.
Hence in $\Theta'''_s$
the $x$-coordinate of $E$ is greater than that of $q_s$ (whereas in $\Theta'$ it is the other way around). By Lemma~\ref{lem:local surjectivity} (1), there is a vertex $w$ of $\Lambda$ such that $\alpha(w)\in M_{\Theta_s}\cap \ev_{\Theta_s}^{-1}(q_1,\dotsc,q_{n-1})$ and $\alpha$ is locally combinatorially surjective at $w$. Thus, there is an edge $e$ of $\Lambda$ such that $\alpha(e^\circ)\subset M_{\Theta'''_S}\cap \ev_{\Theta'''_s}^{-1}(q_1,\dotsc,q_{n-1})$, where $e^\circ$ denotes the interior of $e$. Now we may replace $\trop(C)$ with any curve represented by a point in $\alpha(e^\circ)$, which contains $r-1$ special points on $F_k\cup F_s$ with $x$-coordinates between those of $E$ and $E'$, and has the same vertical complexity as $\trop(C)$ by construction. This completes the inductive step.

Now assume that there is no special point of $\trop(C)$ on $F_k\cup F_s$ with $x$-coordinate between those of $E$ and $E'$. Let $q_j$ be the marked point contained in $h_C(E')$, and $E''$ the second elevator for which $q_j\in h_C(E'')$. Let $F_{s'}$ be the floor adjacent to $E''$.
Since there is no special point in $F_k\cup F_s$ with $x$-coordinate between those of $E$ and $E'$, there is a boundary point of $M_{\Theta'}\cap\ev_{\Theta'}^{-1}(q_1,\dotsc,q_{n-1})$ that is contained in the simple wall $M_\Theta$ in which $E$ and $E'$ get adjacent to the same 4-valent vertex $u$. There are two cases to consider:
\begin{enumerate}
    \item $s'<s$, i.e., $E'$ lies below $F_s$; and
    \item $s'>s$, i.e., $E'$ lies above $F_s$.
\end{enumerate}
See Figure~\ref{fig:wc2} and Figure~\ref{fig:wc3}, respectively, for an illustration. 
\tikzset{every picture/.style={line width=0.75pt}}
\begin{figure}[ht]	
\begin{tikzpicture}[x=0.4pt,y=0.4pt,yscale=-0.8,xscale=0.8]
\import{./}{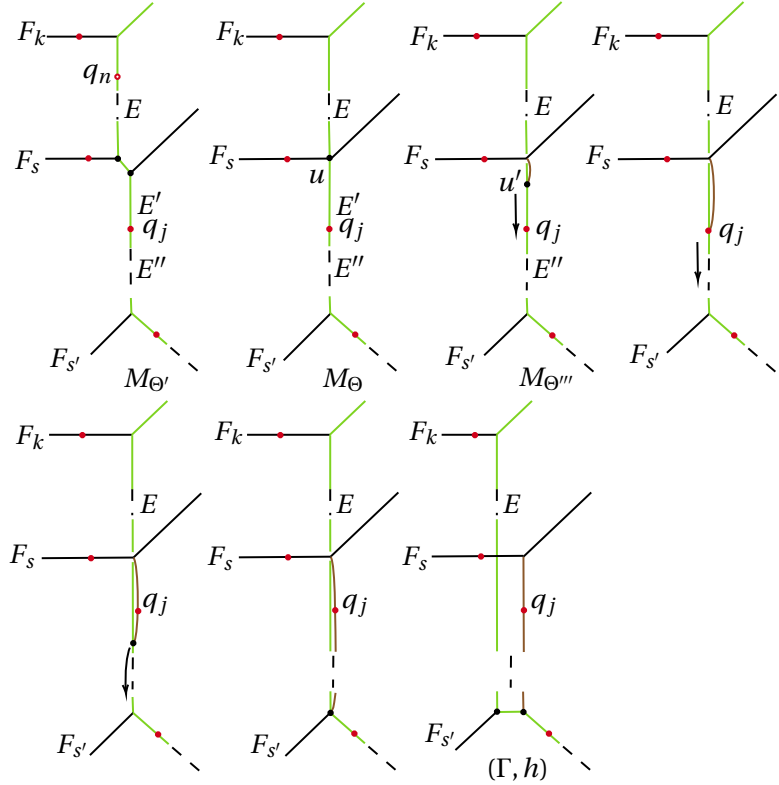}
\end{tikzpicture}		
\caption{Moving $E$ down the elevators $E'$ amd $E''$ with $\m(E') > \m(E)$ (Case 1 in the proof of Lemma~\ref{lem:reduce to two elevators in a cycle}). The edges in the cycle $O$ are colored green.}
\label{fig:wc2}
\end{figure}
In both cases, let $M_{\Theta'''}$ denote the nice stratum in the star of $M_{\Theta}$ in which the elevators $E$ and $E'$ get attached to the same $3$-valent vertex $u'$ obtained from splitting $u$ into two $3$-valent vertices. By Lemma~\ref{lem:local surjectivity} (1), there is a vertex $w \in V(\Lambda)$ and an edge $e \in E(\Lambda)$ adjacent to $w$, such that $\alpha(w)\in M_\Theta$ and
$\alpha(e^\circ)\subset M_{\Theta'''}$.

Case 1: $s'<s$\textit{.} If $\m(E)=\m(E')$, then assertion (1) is satisfied and we are done. Otherwise, we may assume, without loss of generality, that $\m(E)<\m(E')$. 

It follows from the balancing condition that  $u'$ lies below $F_s$, see Figure~\ref{fig:wc2}. We keep using Lemma~\ref{lem:local surjectivity}, to deduce that all curve types illustrated in Figure~\ref{fig:wc2} belong to the image of $\alpha$. Starting with the nice cone $M_{\Theta'''}$ and applying Lemma~\ref{lem:local surjectivity} (2), we conclude that there is a vertex of $\Lambda$ corresponding to the simple wall illustrated on the top right of Figure~\ref{fig:wc2} and parametrizing curves with a $4$-valent vertex $u'$ attached to $E,E',E'',$ and to the contracted leg corresponding to $q_j$. As before, by Lemma~\ref{lem:local surjectivity} (1), there exists an edge $e'\in E(\Lambda)$ parametrizing curves illustrated on the bottom left of Figure~\ref{fig:wc2}, and therefore, there also exists a vertex of $\Lambda$ corresponding to the simple wall illustrated on the bottom middle of Figure~\ref{fig:wc2} and parametrizing curves with a $4$-valent vertex $u'$ attached to $E,E',$ and the two edges on the floor $F_{s'}$. Finally, applying once again Lemma~\ref{lem:local surjectivity} (1), we conclude that the image of $\alpha$ contains curves $(\Gamma,h)$ of the desired type illustrated on the bottom right of Figure~\ref{fig:wc2}. 

The cycle in $(\Gamma, h)$ has vertical complexity $n(O)-1$ and is obtained from $O$ by replacing the elevators $\{E,E',E''\}$ in $\trop(C)$ with the elevator $E$ in $\Gamma$ (note that $E'$ and $E''$ together form a basic floor-to-floor elevator in $\trop(C)$, and thus contribute $1$ to the vertical complexity). This gives a contradiction, since we chose the curve $C$ with minimal vertical complexity in $V^{FD}$.

\tikzset{every picture/.style={line width=0.75pt}}
\begin{figure}[ht]	
\begin{tikzpicture}[x=0.4pt,y=0.4pt,yscale=-0.8,xscale=0.8]
\import{./}{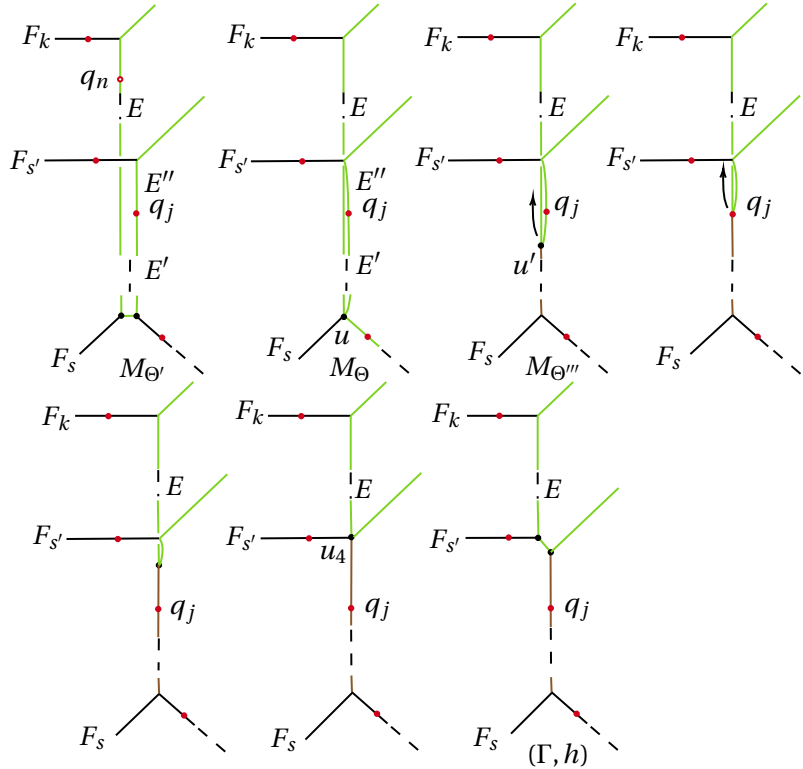}
\end{tikzpicture}		
\caption{Moving $E$ up through the elevators $E'$ and $E''$ (Case 2 in the proof of Lemma~\ref{lem:reduce to two elevators in a cycle}). The edges in the cycle $O$ are colored green.}
\label{fig:wc3}
\end{figure}
Case 2: $s'>s$\textit{.} If $s'=k$, then $\{E,E',E''\}$ induces a cycle of vertical complexity 2, which is assertion (2) and we are done. Otherwise, by a similar argument as in Case 1, we get a tropical curve in $\alpha(\Lambda)$ with vertical complexity less than $\trop(C)$. The only difference is that we are moving up $E$ along $E'$ and $E''$ instead of moving down. See Figure~\ref{fig:wc3} for an illustration. This provides a contradiction and we are done.
\end{proof}

\subsubsection{The proof of Theorem~\ref{thm:main thm}} \label{subsubsec: proof main thm}
We now complete the proof of Theorem~\ref{thm:main thm} by induction on $g$. The base of induction follows from the first assertion of the following Lemma.

\begin{lem}\label{lem:irreducibility in rational case}
Let $\Delta$ and $\underline d$ be as in Theorem \ref{thm:main thm}, and $[C]\in V^\irr_{0,\underline d,\Delta}$. Pick a parametrization of the normalization $\PP^1\simeq C^\nu$, and let $f\colon \mathbb P^1\rightarrow S_\Delta$ be the composition $\PP^1\to C^\nu\to C\to S_\Delta$. Then,
\begin{enumerate}
\item The Severi variety $V^\irr_{0,\underline d,\Delta}$ is non-empty and irreducible, regardless of the characteristic. Furthermore, if $[C]$ is general, then $f$ is injective on the preimage of the boundary of $S_\Delta$;
\item If $[C]$ is general and either $\mathrm{char}(K)=0$ or $\mathrm{char}(K)>\mathtt w(\Delta)$, then $f$ is locally an immersion. In particular, the singular locus of $C$ is contained in the maximal torus. 
\end{enumerate}
\end{lem}

\begin{proof}
(1) Let $D_i\subset S_\Delta$ be the irreducible toric divisors for $1\leq i\leq s$, and $n_i\in N$ the primitive inner normal vectors to the corresponding sides $\partial\Delta_i$ of $\Delta$. Let $p_{i,j}\in \PP^1$ be the points such that $f^*(D_i)=\sum_{1\le j\le a_i}d_{i,j}p_{i,j}$. Without loss of generality, we may assume that $p_{i,j}\ne\infty$ for all $i$ and $j$. Let $t$ be the coordinate on $\PP^1\setminus \{\infty\}$, and set $c_{i,j}:=t(p_{i,j})$. Since $C$ is torically transverse of tangency profile $\underline d$, it follows that the pullback under $f$ of the monomial functions is given by
$$f^*(x^m)=\chi(m)\prod_{i,j}(t-c_{i,j})^{d_{i,j}(n_i,m)},$$
where $(\cdot,\cdot)$ is the pairing $M \times N \to \ZZ$ on the dual lattices and $\chi\:M\to K^\times$ is a character. 

Vice versa, for any choice of a character $\chi$ and distinct points $p_{i,j}$ in $\PP^1\setminus \{\infty\}$, the formula above defines a morphism $f\colon \PP^1\rightarrow S_\Delta$, whose image is torically transverse. It remains to show that for a general choice of the $p_{i,j}$'s, the map $f\colon \PP^1\rightarrow S_\Delta$ is injective on the preimage of the boundary of $S_\Delta$. Indeed, by the assumption, the tangency profile $\underline d$ is trivial on the non-horizontal sides of $\Delta$. Therefore, the injectivity of $f$ over the boundary of $S_\Delta$ implies that $f$ is birational onto its image, and hence the image of $f$ belongs to $V^\irr_{0,\underline d,\Delta}$. Thus, $V^\irr_{0,\underline d,\Delta}$ is non-empty and is dominated by an open subset of ${\rm Hom}(M,K^\times)\times (\mathbb P^1)^{|\underline d|}$. The irreducibility now follows.

Let us show the asserted injectivity. Since $f^*(D_k)=\sum_{1\le j\le a_k}d_{k,j}p_{k,j}$, it is sufficient to show that for each $k$ and each $j_1\ne j_2$, there exists $m\in M$ such that $f^*(x^m)(c_{k,j_1})\ne f^*(x^m)(c_{k,j_2})$. Let $m_k\in M$ be a primitive integral vector along the side $\partial \Delta_k$. Then $(n_k,m_k)=0$, and hence the function $f^*(x^{m_k})$ is a non-trivial rational function in $t$ that does not depend on the parameters $c_{k,j_1}$ and $c_{k,j_2}$. Therefore, for a general choice of $c_{k,j}$'s, $f^*(x^m)(c_{k,j_1})\ne f^*(x^m)(c_{k,j_2})$ as claimed.

(2) To prove that $f$ is locally an immersion, we shall analyze the differential of $f$. For any $m\in M$, set
\begin{equation}\label{eq:derivative}
\partial_m\log(f):=\frac{1}{f^*(x^m)}\frac{d(f^*(x^m))}{dt}=\sum_{1\leq i\leq s}\sum_{1\leq j\leq a_i}\frac{d_{i,j}\cdot( n_i,m)}{t-c_{i,j}}.
\end{equation}
to be the log-derivative of $f^*(x^m)$. 

Let us first show that for each $k$, the log-derivative $\partial_{m_k}\log(f)$ is a non-constant rational function in $t$ that does not depend on the $c_{k,j}$'s. Indeed, since $( n_k,m_k)=0$, the  function $\partial_{m_k}\log(f)$ contains no term(s) with $\{c_{k,j}\}_{1\leq j\leq a_k}$. On the other hand, we claim that we can pick a side $\partial\Delta_l$ such that $0<|( n_l,m_k)|\leq \mathrm{w}(\Delta)$. If this is true, then by the assumptions on the characteristic of $K$, we get in particular, that $(n_l,m_k)$ is not divisible by $\mathrm{char}(K)$, and neither is $d_{l,j}$. It follows that $\partial_{m_k}\log(f)$ contains a non-trivial summand $\frac{d_{l,j}\cdot ( n_l,m_k)}{t-c_{l,j}}$ in expansion~(\ref{eq:derivative}) for any $1\leq j\leq a_l$. Thus, $\partial_{m_k}\log(f)$ depends non-trivially on the $c_{l,j}$'s, and hence, for a general choice of the $c_{i,j}$'s, it is a non-constant rational function in $t$.

Let us now explain how to choose the side $\partial\Delta_l$ as above.
If $\partial\Delta_k$ is horizontal, then any non-horizontal side works. Otherwise,
let $y=a$ be a horizontal line in $M_\mathbb R$ with non-integral $y$-coordinate that intersects $\partial\Delta_k$. Let $\partial\Delta_{k'}$ be the other side of $\Delta$ that intersects this line. If $\partial\Delta_{k'}$ is parallel to $\partial\Delta_k$, let $\partial\Delta_l$ be a side of $\Delta$ that is adjacent to $\partial\Delta_k$; otherwise let $\partial\Delta_l:=\partial\Delta_{k'}$. Then $|(n_l,m_k)|$ is twice the area of the lattice triangle with vertices $(0,0),m_k$ and $m_l$, which is non-zero in each case. Furthermore, since $\Delta$ is $h$-transverse, this lattice triangle has a horizontal edge, it is of height $1$, and can be embedded into $\Delta$ via a translation; hence its area does not exceed $\mathtt w(\Delta)/2$. Therefore,  $|(n_l,m_k)|\leq \mathtt w(\Delta)$. This justifies our choice of $\partial\Delta_l$.

Finally, we are ready to prove that $f$ is locally an immersion. Let $t_0\in \mathbb P^1$ be any point. We shall show that there exists an $m\in M$ such that  $\partial_m\log(f)(t_0)\neq 0$. Assume first, that $t_0=c_{k,j}$. Since the $c_{i,j}$'s are general and the function $\partial_{m_k}\log(f)$ is non-constant and does not depend on $c_{k,j}$, it follows that $t=c_{k,j}$ is not a zero of $\partial_{m_k}\log(f)$ as needed. Assume next, that $t_0\ne c_{i,j}$ for all $i$ and $j$. In particular, $f^*(x^m)$ is defined at $t_0$, and $f^*(x^m)(t_0)\neq 0$ for all $m\in M$. Pick $1\le k\le s$, and let $m'_k\in M$ be such that $(n_k,m'_k)=1$. We will show that either $\partial_{m_k}\log(f)(t_0)\neq 0$ or $\partial_{m'_k}\log(f)(t_0)\neq 0$. We have seen above that $\partial_{m_k}\log(f)$ is not identically zero, and does not depend on the $c_{k,j}$'s. Thus, the zero set $Z\subset \mathbb P^1$ of $\partial_{m_k}\log(f)$ is finite and is independent of the choice of the $c_{k,j}$'s. On the other hand, the log derivative $\partial_{m'_k}\log(f)$ contains a non-trivial summand
$\frac{d_{k,j}\cdot( n_k,m'_k)}{t-c_{k,j}}$ for all $1\leq j\leq a_k$. Thus, no point of $Z$ is a zero of $\partial_{m'_k}\log(f)$ since the $c_{k,j}$'s are general. Hence either $\partial_{m_k}\log(f)(t_0)\neq 0$ or $\partial_{m'_k}\log(f)(t_0)\neq 0$, and we are done.
\end{proof}

\begin{proof}[Proof of Theorem~\ref{thm:main thm}]
Recall that we assume that $K$ is the algebraic closure of a complete DVR whose residue field $\widetilde K$ is algebraically closed and has characteristic either $0$ or greater than $\mathtt w(\Delta)/2$. As mentioned, we prove Theorem~\ref{thm:main thm} by induction on $g$. The base case $g = 0$ follows from Lemma~\ref{lem:irreducibility in rational case} (1), and we can assume $g>0$ in the sequel.
By the inductive hypothesis, it is enough to show that each irreducible component $V$ of $\overline V^\irr_{g,\underline d,\Delta}$ contains an irreducible component of $\overline V^\irr_{g-1,\underline d,\Delta}$, which further reduces to showing that the locus of curves of geometric genus $g-1$ in $V$ has dimension $n-1=|\underline d|+g-2$, the dimension of $\overline V^\irr_{g-1,\underline d,\Delta}$ by Proposition~\ref{prop:tangency}.

Let $[C] \in V^{FD}$ be a curve as in Lemma~\ref{lem:reduce to two elevators in a cycle}. As in Step 1 of the proof of Lemma~\ref{lem:reduce to two elevators in a cycle}, after replacing $[C]$ with another curve in $V^{FD}$, we can find vertically stretched points $q_1,\dotsc,q_n\in h_C(\Gamma_C)$ that are in general position with respect to $\trop(C)$; here we view $\trop(C)$ as an element of $M^\trop_{g,n,\nabla}$ by adding contracted legs over the $q_i$, as before. Pick $p_1,\dotsc,p_n\in C$ such that $\trop(p_i)=q_i$. Then $p_1,\dotsc,p_n$ are in general position with respect to $C$ by Remark~\ref{rem:local general position}. Let $E$ and $E'$ be the elevators of $\trop(C)$ as in Lemma~\ref{lem:reduce to two elevators in a cycle}, and assume without loss of generality that $q_n\in h_C(E)$.

Let $h\:\Gamma_\Lambda \to N_\RR$ be the family of parametrized tropical curves over $\Lambda$ obtained from Construction~\ref{constr:tropical family} for the points $p_1, \dotsc, p_{n-1}$ and the curve $C$. Let $\alpha\colon \Lambda\rightarrow M^\trop_{g,n-1,\nabla}$ be the induced map. Recall that this family is obtained from tropicalizing a family of parametrized curves $f\colon \mathscr X\rightarrow S_\Delta$ over a base $(B,\tau_\bullet)$, where $B$ admits a finite cover to a component $Z\ni [C]$ of the locus in $V$ of curves passing through $p_1,\dotsc,p_{n-1}$. To prove the theorem, it suffices to show that there is a $\tau\in B(K)$ such that $\left[f(\mathscr X_\tau)\right]\in V$ is an integral curve of geometric genus $g-1$, since varying the $p_1,\dotsc,p_{n-1}$ then produces an $n-1$ dimensional locus of such curves in $V$.

By the definition of a family of parametrized tropical curves (cf. \S~\ref{subsec: moduli and families of parametrized tropical curves}), for any edge $e$ of $\Lambda$ the underlying graph of the fiber of $\Gamma_\Lambda$ over any point in the interior $e^\circ$ of $e$ is a fixed graph $\GG_e$.

 \begin{cl}\label{cl:develop a contracted edge}
$\Lambda$ contains a leg $l$ such that the map $h$ is constant on all vertices of $\GG_l$ along $l$ and one of the following holds:
\begin{enumerate}
    \item \label{cl:develop a contracted edge1} The lengths of all but one edge $\gamma$ of $\GG_l$ are constant in the family $\Gamma_\Lambda$, and the graph obtained from $\GG_l$ by removing $\gamma$ is connected; or
    \item \label{cl:develop a contracted edge2} the length of all edges of $\GG_l$ is constant in the family $\Gamma_\Lambda$, except for at least one edge of a contracted elliptic tail $O$.
\end{enumerate}

\end{cl}
Let $l$ be as in Claim~\ref{cl:develop a contracted edge}, and $\tau\in B(K)$ the marked point corresponding to $l$. We will show that $\tau$ is the desired point, namely that  $\left[f(\mathscr X_\tau)\right]\in V$ is an integral curve of geometric genus $g-1$. 
Set $D_{\cX^0}:=\widetilde \cX \cup \left ( \bigcup_i \cX_{\tau_i} \right) \cup \left( \bigcup_j \sigma_j \right)$, where $\tau_i$ are the marked points of $B$ and $\sigma_j$ are the marked points of $\cX \to B$. First, let us show that $f$ is defined on $\cX_\tau$.

To see that the rational map $f$ is defined at the generic points of $\cX_\tau$, consider the pullbacks $f^*(x^m)$ for $m\in M$. Let $u$ be a vertex of $\GG_l$ such that the corresponding component of $\cX_{\widetilde \tau}$ belongs to the closure of a given generic point $\eta\in\cX_\tau$. By Remark~\ref{rem:slopes}, the order of pole of $f^*(x^m)$ at $\eta$ is equal to the slope $\frac{\partial h(u)(m)}{\partial \vec l}$, which has to vanish since $h(u)$ is constant along $l$. Therefore, $f$ is defined at the generic points of $\cX_\tau$ and maps them to the dense orbit of $S_\Delta$. Next, let us show that $f$ is defined on all of $\cX_\tau$. Notice that $f^*(x^m)$ is regular and invertible away from $D_{\cX^0}$, and it is regular and invertible in codimension one on $\cX_\tau$. Since $\cX^0$ is normal,  it follows that $f$ is defined on $\cX_\tau\setminus \left( \bigcup_j \sigma_j \right)$. Pick any $\sigma_j$, and let us show that $f$ is defined at $\sigma_j(\tau)\in \cX_\tau$, too. Let $S_j\subseteq S_\Delta$ be the affine toric variety consisting of the dense torus orbit and the orbit of codimension at most one containing the image of $\sigma_j(b)$ for a general $b\in B$. Then the pullback of any regular monomial function $x^m\in \mathcal O_{S_j}(S_j)$ is regular in codimension one in a neighborhood of $\sigma_j(\tau)$, and hence regular at $\sigma_j(\tau)$ by normality of $\cX$. Thus, $f$ is defined at $\sigma_j(\tau)$, and maps it to $S_j$.

To complete the proof, assume that assertion (2) of Claim ~\ref{cl:develop a contracted edge} is satisfied. The case of assertion (1) can be treated in a similar way. Let $\gamma \in E(\mathbb G_l)$ be an edge that corresponds to a node $z \in  \cX_{\widetilde\tau}$. Then $\cX^0$ is given \'etale locally near $z$ by $xy = g_z$ for some regular function $g_z$ defined on a neighbourhood of $\widetilde \tau\in B^0$. If the length of $\gamma$ varies over $l$, then $g_z$ vanishes along $\tau$ by Remark~\ref{rem:slopes}, and thus $z$ does not get smoothed in $\mathscr X_\tau$. 

If the only edge of varying length in $O$ is a loop, then $\mathscr X_\tau$ is irreducible, since $\mathbb G_l$ is still connected even if we remove this loop. By the argument above, $f$ is defined on $\cX_\tau$ and maps its generic point to the dense orbit. Moreover, the image of $\mathscr X_\tau$ intersects the boundary divisor at $|\underline d|$ distinct points, and some of them are points of simple intersection by the assumption on $\underline d$. Thus, $f|_{\cX_\tau}$ is birational onto its image, and $\left[f(\cX_\tau)\right]\in V$ is an integral curve of geometric genus $g-1$. 
Assume next that $O$ contains a non-loop edge of varying length. Then $\cX_\tau$ is a union of a smooth irreducible component of genus $g-1$ and an irreducible component of arithmetic genus $1$, and without marked points. Denote the components by $\cX_\tau'$ and $E$, respectively. Then $f(E)$ is disjoint from the toric boundary, and therefore $f$ contracts $E$ to a point in the dense orbit of $S_\Delta$. In particular, $f(\cX_\tau)=f(\cX_\tau')$ is irreducible. On the other hand, we get arguing as above, that $f|_{\cX_\tau'}$ is birational onto its image. Thus, the normalization of $f(\cX_\tau)$ is equal to $\cX_{\tau}'$, and $\left[f(\cX_\tau)\right]\in V$ is an integral curve of geometric genus $g-1$.
\end{proof}

\begin{proof}[Proof of Claim~\ref{cl:develop a contracted edge}]
Recall that $C$ satisfies the assertion of Lemma~\ref{lem:reduce to two elevators in a cycle}. The points $q_1,\dotsc,q_n$ are in general position with respect to $\trop(C)=(\Gamma_C,h_C)$ and vertically stretched. The points  $p_1,\dotsc,p_n\in C$ are points such that $\trop(p_i)=q_i$. The edges $E$ and $E'$ are the elevators of $\trop(C)$ contained in a cycle $O$ and adjacent to a floor $F$ as in Lemma~\ref{lem:reduce to two elevators in a cycle}, and $q_n\in h_C(E)$.  Moreover, $h\:\Gamma_\Lambda \to N_\RR$ is the family of parametrized tropical curves obtained from Construction~\ref{constr:tropical family} for $p_1,\dotsc, p_{n-1}$, and $\alpha\colon \Lambda\rightarrow M^\trop_{g,n-1,\nabla}$ the induced map.

As before, we may consider $\trop(C)$ as a curve in the family $h\:\Gamma_\Lambda \to N_\RR$ by marking the preimages of $q_1,\dotsc,q_{n-1}$ in $\Gamma_C$. Then $E$ contains no marked point and $E'$ contains one marked point which we denote by $q_j$. 
As in the proof of Lemma~\ref{lem:reduce to two elevators in a cycle}, we may assume that on the floors adjacent to $E$ there is no special point that has $x$-coordinate between those of $E$ and $E'$. Let $M_{\Theta'}\subset M^\trop_{g,n-1,\nabla}$ be the nice stratum that contains $\trop(C)$. According to the two conditions of Lemma~\ref{lem:reduce to two elevators in a cycle}, there are two cases to consider.

Case 1: \textit{$\frac{\partial h_C}{\partial\vec E}=-\frac{\partial h_C}{\partial\vec E'}$, where $\vec E$ and $\vec E'$ are oriented away from $F$.} As in the proof of Lemma~\ref{lem:reduce to two elevators in a cycle}, we move $E$ towards $E'$ till we hit a simple wall $M_\Theta$ that parametrizes curves $(\Gamma,h)$ containing a $4$-valent vertex $u$ adjacent to both $E$ and $E'$. See Figure~\ref{fig:wc4} for an illustration. As before, all parametrized tropical curves along this deformation belong to $\alpha(\Lambda)$, and hence $[(\Gamma,h)]=\alpha(w)$ for a vertex $w$ of $\Lambda$. Furthermore, by Lemma~\ref{lem:local surjectivity} (1), we may assume that $\alpha$ is locally combinatorially surjective at $w$. Let $M_{\Theta'''}$ be the nice stratum in the star of $M_\Theta$ in which $E$ and $E'$ are adjacent to the same vertex $u'$ of the splitting of $u$ into two $3$-valent vertices.
Let $e\in \Star(w)$ be an edge such that $\alpha(e^\circ)\subset M_{\Theta'''}$. Since $\m(E)=\m(E')$, it follows from the balancing condition that the third edge adjacent to $u'$ gets contracted by the parametrization map. Thus, the locus $M_{\Theta'''}\cap\ev_{\Theta'''}^{-1}(q_1,\dotsc,q_{n-1})$ 
is an {\em unbounded} interval. Furthermore, the images of all curves in this locus agree with $h(\Gamma)$. 

\tikzset{every picture/.style={line width=0.75pt}}
\begin{figure}[ht]	
\begin{tikzpicture}[x=0.5pt,y=0.5pt,yscale=-0.9,xscale=0.9]
\import{./}{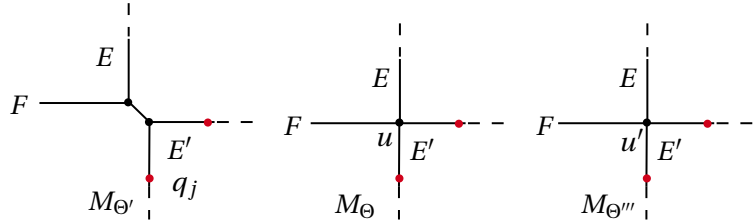}
\end{tikzpicture}		
\caption{Developing a contracted edge in Case 1 of the proof of Claim~\ref{cl:develop a contracted edge}.}
\label{fig:wc4}
\end{figure}

By Lemma~\ref{lem:local surjectivity} (2),  $M_{\Theta'''}\cap\ev_{\Theta'''}^{-1}(q_1,\dotsc,q_{n-1})=\alpha(\Lambda)\cap M_{\Theta'''}$, and thus there exists a leg $l$ of $\Lambda$ such that $\alpha(l)\subset M_{\Theta'''}\cap\ev_{\Theta'''}^{-1}(q_1,\dotsc,q_{n-1})$ is not bounded. The leg $l$ satisfies the assertion of the claim: the graph $\GG_l$ is the underlying graph of $\Theta'''$, and the 
edge $\gamma$ in the claim is just the contracted edge adjacent to $u'$ as above. The connectivity of $\GG_l\backslash \gamma$ follows from the fact that $E$ and $E'$ are contained in the same cycle, and hence assertion \eqref{cl:develop a contracted edge1} of Claim~\ref{cl:develop a contracted edge} is satisfied.

Case 2: \textit{$O$ has vertical complexity 2.} See left picture in Figure~\ref{fig:wc5}.
\tikzset{every picture/.style={line width=0.75pt}}
\begin{figure}[ht]	
\begin{tikzpicture}[x=0.5pt,y=0.5pt,yscale=-0.9,xscale=0.9]
\import{./}{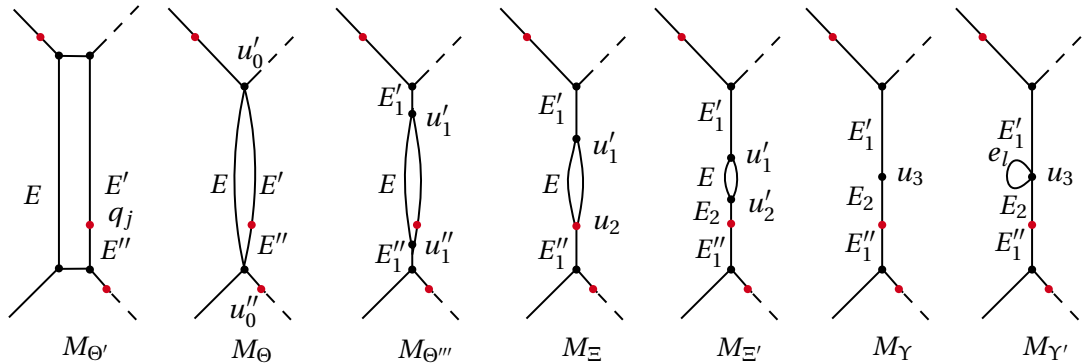}
\end{tikzpicture}		
\caption{Developing a contracted loop in Case 2 of the proof of Claim~\ref{cl:develop a contracted edge}.}
\label{fig:wc5}
\end{figure}
We may assume that $E'$ lies above $q_j$. Let $E''$ be the other elevator that contains $q_j$. 
By moving $E$ towards $E'$ as before, we find a boundary point $(\Gamma',h')$ of $M_{\Theta'}\cap \ev_{\Theta'}^{-1}(q_1,\dotsc,q_{n-1})$ that is contained in a stratum $M_\Theta$. In this case, $M_\Theta$ is not a simple wall. Instead, $(\Gamma',h')$ is weightless and $3$-valent except for two $4$-valent vertices $u'_0$ and $u''_0$, with $u'_0$ adjacent to $E$ and $E'$, and $u''_0$ adjacent to $E$ and $E''$; in other words, $E, E', E''$ form a flattened cycle between $u'_0$ and $u''_0$. By Lemma~\ref{lem:local surjectivity} (2), $\alpha(\Lambda)$ contains $M_{\Theta'}\cap \ev_{\Theta'}^{-1}(q_1,\dotsc,q_{n-1})$, hence also contains $(\Gamma',h')$.

We keep deforming the obtained tropical curve by shrinking the flattened cycle, cf. Figure~\ref{fig:wc5}: let $M_{\Theta'''}$ be the stratum in the star of $M_{\Theta}$ in which both $u'_0$ and $u''_0$ split into two $3$-valent vertices, such that each pair of elevators adjacent to $u'_0$ (resp., $u''_0$) remains adjacent to vertices $u'_1$ (resp., $u''_1$); that is, the flattened cycle is preserved. Denote the new edges adjacent to $u'_1$ and $u''_1$ by $E'_1$ and $E''_1$, respectively. Then $\alpha(\Lambda)\cap M_{\Theta'''}$ is non-empty by  Lemma~\ref{lem:alpha at flattened cycle}.
Note that, unlike before, $M_{\Theta'''}\cap \ev_{\Theta'''}^{-1}(q_1,\dotsc,q_{n-1})$ has dimension $2$: one can perturb the images of $u'_1$ and $u''_1$ independently. Nevertheless, the locus of realizable curves in $M_{\Theta'''}\cap \ev_{\Theta'''}^{-1}(q_1,\dotsc,q_{n-1})$ is one-dimensional by Proposition~\ref{prop:weight one midpoint}; more precisely, we get that $\alpha(\Lambda)\cap M_{\Theta'''}$ is contained in the sublocus $L_{\Theta'''}$ of $M_{\Theta'''}\cap \ev_{\Theta'''}^{-1}(q_1,\dotsc,q_{n-1})$ consisting of curves such that $\ell(E'_1)=\ell(E''_1)$, where $\ell(\bullet)$ denotes the edge lengths.  Plainly, $L_{\Theta'''}$ is a line segment in $M_{\Theta'''}$. Since the underlying graph of $\Theta'''$ is weightless and $3$-valent, $\alpha$ is harmonic at vertices of $\Lambda$ that are mapped to $M_{\Theta'''}$ by Remark~\ref{rem: harmonic when weightless trivalent}.
It follows that $\alpha(\Lambda)$ contains $L_{\Theta'''}$.

By the balancing condition, we have $\m(E'_1)=\m(E''_1)$. This means that, along $L_{\Theta'''}$, the images of $u'_1$ and $u''_1$ are moving (one downwards and the other upwards) with the same speed. We may assume that $q_j$ lies strictly below the center of the flattened cycle. When $u_1''$ reaches $q_j$, we obtain a boundary point of $L_{\Theta'''}$ that is not contained in $M_{\Theta'''}$; namely, a curve which contains a $4$-valent vertex $u_2$ with the following adjacent edges/legs: $E''_1$, the leg $l_j$ contracted to $q_j$, and the two elevators that lie above it. 
Let $\Xi$ be the combinatorial type of this curve. 
Note that this parametrized tropical curve has a non-trivial automorphism induced by switching the two elevators above $u_2$ when these two elevators have the same multiplicity.  

Let $\Xi'$ be the combinatorial type obtained from $\Xi$ by splitting $u_2$ into two $3$-valent vertices, such that $l_j$ and $E''_1$ remain incident.
This shrinks the flattened cycle further, and we denote by $u'_2$ its lower endpoint and by $E_2$ the downward elevator adjacent to $u_2'$.
By Lemma~\ref{lem:alpha at simple walls and nice strata}, $\alpha(\Lambda) \cap M_{[\Xi']} \neq \emptyset$, and, arguing as above, $\alpha(\Lambda)$ contains (the image in $M_{[\Xi']}$ of) the line segment $L_{\Xi'}\subset M_{\Xi'}\cap\ev_{\Xi'}^{-1}(q_1,\dotsc,q_{n-1})$ consisting of curves with $\ell(E'_1)=\ell(E_2)+\ell(E''_1)$ .

The second boundary point of $L_{\Xi'}$, obtained by shrinking the flattened cycle to a point, contains a $2$-valent vertex $u_3$ of weight $1$ adjacent to $E'_1$ and $E_2$. Let $\Upsilon$ be its combinatorial type. Recall that $\mathrm{char}(\widetilde K)$ is either $0$ or greater than $\mathtt w(\Delta)/2$.\footnote{What follows is the only point in the proof of Theorem~\ref{thm:main thm} at which we use the assumption on the characteristic, and we need it to apply Lemma~\ref{lem:alpha at weight one vertex}.} By Lemma~\ref{lem:max stretching factor}, the multiplicity of the edge $E'_1$ (and, by balancing, of $E_2$) is at most $\mathtt w(\Delta)$. Thus either the multiplicity is not divisible by $\mathrm{char}(\widetilde K)$ or equal to $\mathrm{char}(\widetilde K)$. In the first case we develop a contracted loop, and in the second case a contracted elliptic tail as follows.

First, assume that $\m(E_1')$ is not divisible by $\mathrm{char}(\widetilde K)$.
Let $M_{[\Upsilon']}$ be a stratum in the star of $M_{[\Upsilon]}$ such that $\Upsilon'$ is obtained from $\Upsilon$ by adding a contracted loop at $u_3$ and setting the weight of $u_3$ to $0$. 
By Lemma~\ref{lem:alpha at weight one vertex} (1) we have that $\alpha(\Lambda)\cap M_{[\Upsilon']}$ is non-empty.
By Proposition~\ref{prop:weight one midpoint}, we have that $\alpha(\Lambda)\cap M_{[\Upsilon']}$ is contained in (the image of) the ray $L_{\Upsilon'}\subset M_{\Upsilon'}\cap\ev_{\Upsilon'}^{-1}(q_1,\dotsc,q_{n-1})$ consisting of curves such that $\ell(E'_1)=\ell(E_2)+\ell(E''_1)$. By Lemma \ref{lem:alpha at weight one vertex} (3), $\alpha$ is harmonic at vertices of $\Lambda$ that are mapped to $M_{[\Upsilon']}$, and hence, similar as above,
$L_{\Upsilon'}$ is contained in $\alpha(\Lambda)$. Thus, there is a leg $l$ of $\Lambda$ such that $\alpha(l)\subset L_{\Upsilon'}$ is not bounded. Then $l$ satisfies the assertion of the claim: the graph $\GG_l$ of the claim is the underlying graph of $\Upsilon'$, and the edge $\gamma$ of the claim is the contracted loop at $u_3$ (in particular, $\GG_l\backslash \gamma$ is connected). Thus assertion \eqref{cl:develop a contracted edge1} of Claim~\ref{cl:develop a contracted edge} is satisfied. 

Finally, suppose $\m(E_1') = \mathrm{char}(\widetilde K)$. In this case, let $\Upsilon''$ be the combinatorial type obtained from $\Upsilon$ by adding a contracted edge adjacent to $u_3$, now $3$-valent and of weight $0$, and $u_3'$, a new vertex of weight $1$ and valence $1$; let $\Upsilon'''$ be the combinatorial type obtained from $\Upsilon''$ by developing a loop at $u_3'$ (that is, $u_3$, $u_3'$, the edge between them and possibly the loop form a contracted elliptic tail; cf. Figure~\ref{fig:flattenedcycle}). By Lemma~\ref{lem:alpha at weight one vertex} (2) we have that $\alpha(\Lambda)\cap M_{[\Upsilon'']}$ is non-empty. By Lemma \ref{lem:alpha at weight one vertex} (3), for a vertex $w \in \Lambda$ with $\alpha(w) \in M_{[\Upsilon'']}$, $\alpha$ is harmonic at $w$, or there is an edge/leg $e$ adjacent to $w$ such that $\alpha(e^\circ) \subset M_{[\Upsilon''']}$. If $\alpha$ is harmonic at all vertices mapped to $M_{[\Upsilon'']}$, we argue as above to show that assertion \eqref{cl:develop a contracted edge2} of Claim~\ref{cl:develop a contracted edge} is satisfied; that is, we deduce from Proposition~\ref{prop:weight one midpoint} that $\alpha(\Lambda)$ contains the ray $L_{\Upsilon''}$, defined analogous to $L_{\Upsilon'}$. Otherwise, we get that $\alpha(\Lambda)\cap M_{[\Upsilon''']}$ is not empty. Since $\alpha$ is harmonic at vertices mapped to $M_{[\Upsilon''']}$ by Lemma \ref{lem:alpha at weight one vertex} (3), we can argue again as above to obtain that assertion \eqref{cl:develop a contracted edge2} of Claim~\ref{cl:develop a contracted edge} is satisfied; that is, we deduce from Proposition~\ref{prop:weight one midpoint} that $\alpha(\Lambda)$ contains a ray in the intersection of $M_{[\Upsilon''']}$ with the plane given by the condition $\ell(E'_1)=\ell(E_2)+\ell(E''_1)$. 
\end{proof}

\subsection{A counterexample in low characteristic}\label{sec:counterexample}
\begin{wrapfigure}{r}{0.32\textwidth}
\begin{center}
\tikzset{every picture/.style={line width=0.75pt}}
\begin{tikzpicture}[x=0.8pt,y=0.8pt,yscale=-0.7,xscale=0.7]
\import{./}{figdiamond.tex}
\end{tikzpicture}		
\caption{$k=k'=3$.}
\label{fig:diamond}
\end{center}
\end{wrapfigure}
We conclude our discussion of degenerations on $S_\Delta$ with an example that shows that Theorem~\ref{thm:main thm} may fail without the assumption on the characteristic of $K$.
Recall that we defined kites in Example~\ref{ex:h-transverse polygon} following \cite{LT20} as polygons with vertices $(0,0), (k, \pm 1)$ and $(k + k', 0)$ for non-negative integers $k, k'$ with $k' \geq k$ and $k' > 0$. As we noted above, kites are $h$-transverse, and by definition we have $\mathtt w (\Delta) = k + k'$.

\begin{prop} \label{prop:low characteristic counterexample}
    Suppose $\mathrm{char}(K) = p > 0$ and let $\Delta$ be the kite with 
    \begin{equation*}
    k = k' = \left\{
    \begin{array}{l}
    3 \text{ if } p = 2  \\
    p \text{ otherwise}
  \end{array}.
    \right.
    \end{equation*}
    Consider $V_{1, \Delta}^{\irr}$ parametrizing irreducible curves of geometric genus $1$ on $S_\Delta$ with trivial tangency profile. Then there is an irreducible component $V$ of $\overline V_{1, \Delta}^{\irr}$, such that $V \cap V_{0, \Delta}^{\irr} = \emptyset.$
\end{prop}

\begin{proof}
We assume $k = k' = p > 2$, the argument for $p = 2$ and $k = k' = 3$ is completely analogous. We first follow the description in \cite[\S 3]{LT20} to obtain some properties of the curves para\-me\-trized by $V^{\irr}_{1, \Delta}$. For a general point $[C] \in V^{\irr}_{1, \Delta}$, let $f\: E \to S_\Delta$ be the parametrized curve given by normalizing $C$. Let $O, P, Q, R$ be the preimages of the toric divisors $D_i$ corresponding to the sides of $\Delta$ with outer normals $(-1, -p), (- 1, p), (1, - p), (1,p)$, respectively. Then the divisors on $E$ of the pullbacks of the monomials $x$ and $y$ are given by $\mathrm{div}(f^*(y)) = -p O + p P - p Q + p R$ and $\mathrm{div}(f^*(x)) = - O - P + Q + R$. Using these equations, we get in the group law of the elliptic curve $(E, O)$:
\begin{equation} \label{eq:torsion diamond}
P = R + Q \quad \text{ and } \quad 2p R = 0.
\end{equation}
Conversely, any smooth genus $1$ curve $E$ and choice of distinct points $O, P, Q, R$ satisfying the two conditions above gives an element in $V^{\irr}_{1, \Delta}$.

For $l \in \NN_{\geq 2}$ let $X_1(l)$ denote the modular curve parametrizing generalized elliptic curves together with an ample Drinfeld $\ZZ/l \ZZ$-structure as in \cite[Definitions~2.4.1 and 2.4.6]{C07}; see also \cite{DR72} and \cite{KM85}. Denote by $Y_1(l)$ the open subset of $X_1(l)$ over which the elliptic curves are smooth. We will consider the forgetful map from $V^{\irr}_{1, \Delta}$ to $Y_1(2p)$ induced by $(E, O, R)$. More precisely, by what we discussed above, we may choose a point $[C] \in V^{\irr}_{1, \Delta}$ such that for the triple $(E, O, R)$, $R$ is a point of exact order $2p$ -- that is, the subgroup scheme generated by multiples of $R$ is isomorphic to $\ZZ/2p \ZZ$. Denote by $V^\circ$ the irreducible component of $V^{\irr}_{1, \Delta}$ containing $[C]$. Possibly after pulling back along a finite surjective map, we may assume that the universal family over $V^\circ$ is equinormalizable by \cite[\S 5.10]{dJ96}. Thus the normalization of the total space of the universal family induces a map $V^\circ \to Y_1(2p)$, whose image is contained in an irreducible component $Y^\circ$ of $Y_1(2p)$. Denote by $Y \subset X_1(2p)$ its closure.

Now let $[C_0]$ be a point in the boundary of $V = \overline {V^\circ}$ and assume contrary to the claim that $[C_0] \in V_{0, \Delta}^{\irr}$.  Choose a discretely valued field $F$ with valuation ring $F^0$ and a map $\Spec(F^0) \to V$ such that the generic point of $\Spec(F^0)$ maps to a point in $V^\circ$, and the special point to $[C_0]$.
Pulling back the normalization of the total space of the universal family over $V^\circ$, we get a smooth genus $1$ curve $\mathcal E \to \Spec(F)$ with the four marked points $O, P, Q, R$. Up to a finite extension of $F$, we may assume that this family extends to a stable model $\mathcal E^0 \to \Spec(F^0)$. Then $f$ extends on the central fiber $\widetilde{ \mathcal E}$ of $\mathcal E^0$ to a map $\widetilde f\: \widetilde{\mathcal E} \to S_\Delta$ with $C_0 = \widetilde f(\widetilde{\mathcal E})$. Since we assume that $C_0$ is torically transverse, $C_0$ intersects the toric boundary in $4$ distinct points. Since $\widetilde{ \mathcal E}$ is stable of arithmetic genus $1$, and $C_0$ is irreducible, this implies the following: 
\begin{center}
    ($\dagger$)\; at least $3$ marked points lie on an irreducible component $E_0$ of $\widetilde{ \mathcal E}$ of geometric genus $0$. 
\end{center}

On the other hand, consider the map $\varphi^\circ\: \Spec(F) \to Y^\circ$ induced by $(\mathcal E, O, R)$ as above. Since $X_1(2p)$ is proper, the map extends uniquely to a map $\varphi\: \Spec(F^0) \to Y$. We denote by $[X_0] \in Y$ the image of the special point under this extension, corresponding to a generalized elliptic curve $X_0$. Let $O_0, P_0, Q_0$ and $R_0$ denote the points on $X_0$ that extend, respectively, $\varphi^\circ(O), \varphi^\circ(P), \varphi^\circ(Q)$, and $\varphi^\circ(R)$ to $X_0$.
Recall that we have two natural maps $\varphi_2\: X_1(2p) \to X_1(p)$ and $\varphi_p\: X_1(2p) \to X_1(2)$, both over $\overline M_{1,1}$. They map a level $2p$ structure to a level $p$ structure (respectively, a level $2$ structure) via multiplication by $2$ (respectively, $p$); on a singular curve, they contract irreducible components, that no longer intersect the level structure, see \cite[Lemma 4.2.3 (ii)]{C07}.

Suppose first $X_0$ is smooth. Since $p \not | 2$, the torsion points parametrized  by $Y_1(2)$ all are distinct from the origin on the elliptic curves. Since $R$ has exact order $2p$, this implies that the map on universal families induced by $\varphi_p$ maps $R_0$ and $O_0$ to distinct points. In particular, $R_0$ and $O_0$ are distinct points on $X_0$. On the other hand, by \eqref{eq:torsion diamond}, we have $Q_0 = O_0 + Q_0$ and $P_0 = R_0 + Q_0$. Thus also $Q_0$ and $P_0$ are distinct points on $X_0$. This implies that the stable model of the curve with marked points $(\varphi^\circ(\mathcal E), \varphi^\circ(O), \varphi^\circ(P), \varphi^\circ(Q), \varphi^\circ(R))$ cannot contain a rational component with $3$ marked points; by uniqueness of the stable model, this contradicts~($\dagger$).
Now suppose $X_0$ is singular.  By \cite[Proposition 2.3, p. 246]{DR72}, $X_1(p)$ has $2$ irreducible components and one of them, $Y'$, parametrizes, away from supersingular curves, generalized elliptic curves together with a fixed inclusion of $\ZZ/p \ZZ$ as a subgroup scheme. Since the multiplicative torus $\GG_m$ does not contain a subgroup scheme isomorphic to $\ZZ/p\ZZ$ in characteristic $p$, the universal family over $Y'$ does not contain singular, irreducible fibers. Since we chose $R$ such that it generates a subgroup scheme isomorphic to $\ZZ/2p \ZZ$, we have that $\varphi_2(Y) \subset Y'$. Thus $X_0$ is not irreducible. By the definition of an ample Drinfeld $\ZZ/2p \ZZ$ structure, it follows that $O_0$ and $R_0$ lie on different irreducible components of $X_0$. Arguing similar as before, we obtain that also $Q_0$ and $P_0$ lie on different irreducible components of $X_0$ by \eqref{eq:torsion diamond}. This contradicts ($\dagger$) by the uniqueness of the stable model. 
\end{proof}

\section{Zariski's Theorem}\label{sec:zar}
As a consequence of Theorem~\ref{thm:main thm}, we prove in this section Zariski's Theorem for $S_\Delta$. Recall that we denote the sides of $\Delta$ by $\partial\Delta_i$ and the corresponding toric divisors of $S_\Delta$ by $D_i$, where $1\leq i\leq s$.

\begin{thm}[Zariski's Theorem]\label{thm:Zariski}
Let $\Delta$ be an $h$-transverse polygon and suppose either $\mathrm{char}(K)=0$ or $\mathrm{char}(K)>\mathtt w(\Delta)$. Let $0 \leq g\leq \#|\Delta^\circ\cap M|$ be an integer and $\underline d$ a tangency profile of $(S_\Delta, \mathscr L_\Delta)$, which is trivial on toric divisors corresponding to non-horizontal sides of $\Delta$.
Let $V\subset V_{g,\underline d,\Delta}$ be a subvariety. Then: 
\begin{enumerate}
    \item $\dim V\leq |\underline d|+g-1$, and
    \item if $\dim V =|\underline d|+g-1$, then for a general $[C]\in V$, the curve $C$ is nodal, and its singular locus is contained in the maximal torus.
\end{enumerate}
\end{thm}

\begin{lem}\label{lem:height 2}
Under the assumptions of Theorem~\ref{thm:Zariski}, suppose that $\Delta$ has height 2 and no horizontal sides. Then for a general point $[C]\in  V^\irr_{0,\Delta}$, the curve $C$ is nodal.
\end{lem}

\begin{proof}
After a translation and a change of coordinates, we may assume that $\Delta$ has vertices $(0,0),$ $(a,1),(a+b,0),(0,-1)$, where $a,b\geq 0$ and $a+b>0$. We may further assume that $\mathtt w(\Delta)\geq 2$ as otherwise $C$ is smooth. Therefore, char$(K)>2$. Let us label the sides of $\Delta$ as illustrated
in Figure~\ref{fig:height2}. 
Note that the case where $\Delta$ is a triangle is included by setting $a=0$.

\begin{figure}[ht]
\tikzset{every picture/.style={line width=0.75pt}}
\begin{tikzpicture}[x=1pt,y=1pt,yscale=-1,xscale=1]
\import{./}{figheight2.tex}
\end{tikzpicture}		
\caption{}
\label{fig:height2}
\end{figure}

As in the proof of Lemma~\ref{lem:irreducibility in rational case}, a general curve $[C]\in V^\irr_{0,\Delta}$ is represented as the image of a map $f\colon \mathbb P^1\rightarrow S_\Delta$.
We may assume that $f^{-1}(D_1)=c$, $f^{-1}(D_2)=1$, $f^{-1}(D_3)=0$ and $f^{-1}(D_4)=\infty$. Then
$$f^*(x)=\alpha\cdot\frac{t}{(t-c)(t-1)}\mathrm{\ and\ }f^*(y)=\beta\cdot t^b(t-1)^a,$$
where $\alpha,\beta\in K^\times$.

By Lemma~\ref{lem:irreducibility in rational case}, the curve $C$ has no unibranch singularities and is smooth on the boundary divisor of $S_\Delta$. On the other hand, since $f^*(x)$ is the quotient of two polynomials of degree at most $2$, each singularity of $C$ has multiplicity at most $2$. Thus it suffices to show that a general curve $[C]\in V^\irr_{0,\Delta}$ does not contain a (generalized) tacnode that lies in the maximal torus.

Suppose to the contrary that $p\in (K^\times)^2$ is a tacnode of $C$, and $f^{-1}(p)=\{t_1,t_2\}$. Then $f^*(x)(t_1)=f^*(x)(t_2)$, from which we get $t_1t_2=c$. On the other hand,
since $\left(\frac{df^*(x)}{dt}\Big|_{t_i},\frac{df^*(y)}{dt}\Big|_{t_i}\right)$ gives the same tangent direction for $i=1,2$, we have
\begin{equation}\label{eq:tangent condition}
    \frac{df^*(x)}{dt}\Big|_{t_1}\cdot \frac{df^*(y)}{dt}\Big|_{t_2}=\frac{df^*(x)}{dt}\Big|_{t_2}\cdot \frac{df^*(y)}{dt}\Big|_{t_1}.
\end{equation}
Let us compute the derivatives:
$$\frac{df^*(x)}{dt}=f^*(x)\Big(\frac{1}{t}-\frac{1}{t-c}-\frac{1}{t-1}\Big) \mathrm{\ \ and\ \ }\frac{df^*(y)}{dt}=f^*(y)\Big(\frac{a}{t-1}+\frac{b}{t}\Big).$$
Substituting these in equation~(\ref{eq:tangent condition}), we get
$$(a+b)\cdot(t_1-c)(t_2-c)+a(c-1)\cdot t_1t_2+bc\cdot (t_1-1)(t_2-1)=0.$$
Since $t_1t_2=c\neq 0$,
we can rearrange terms to obtain
$$(a+2b)\cdot(t_1+t_2)=(2a+2b)\cdot c+2b.$$
As $c$ is general and char$(K)>\max\{2,a+b\}$, this equation can not hold if $a+2b=0$. Hence we may assume $a+2b\neq 0$, which gives $$t_1+t_2=\frac{(2a+2b)\cdot c+2b}{a+2b}.$$
It follows that $t_1$ and $t_2$ are the two roots of the quadratic polynomial
$$F(t):=t^2-\frac{(2a+2b)\cdot c+2b}{a+2b}t+c.$$
On the other hand, since $f^*(y)(t_1)=f^*(y)(t_2)$, $t_1$ and $t_2$ are also two roots of
$$G_\mu(t):=t^b(t-1)^a-\mu$$
for some $\mu\in K^\times$. Therefore, in order to get a contradiction, it suffices to show that for a general $c\in \mathbb P^1$, there is no $\mu\in K^\times$ such that $F(t)$ divides $G_\mu(t)$.

Let $R(t)$ be the residue of $t^b(t-1)^a$ divided by $F(t)$. We can write $R(t)=A(c)\cdot t+B(c)$
where $A$ and $B$ are polynomials in $c$. Now $F(t)$ not dividing $G_\mu(t)$ for any $\mu\in K$  (including $\mu=0$) is equivalent to
$A(c)\neq 0$. It is therefore sufficient to show that $A$ is a non-zero polynomial, which further reduces to showing that $F(t)$ does not divide $G_\mu(t)$ for any $\mu\in K$ when $c=0$. In this case the two roots of $F(t)$ are $0$ and $2b/(a+2b)$. If $a=0$, then $t=0$ and $t=2b/(a+2b)=1$ can not be the roots of $G_\mu=t^b-\mu$ at the same time. If $a\neq 0$ and $b=0$ then $F=t^2$ and there is a term $\pm at$ in $G_\mu$, hence $F$ does not divide $G_\mu$. If $a,b\neq 0$ and $\mu=0$ then $t=2b/(a+2b)\neq 0,1$ is not a root of $G_\mu=t^b(t-1)^a$; if $a,b,\mu\neq 0$ then $t=0$ is not a root of $G_\mu$. In each case $G_\mu$ is not divisible by $F$, hence we are done.
\end{proof}

\begin{proof}[Proof of Theorem~\ref{thm:Zariski}]
The first part of the claim follows immediately from Proposition~\ref{prop:tangency}. We need to show that if $V$ is an irreducible component of $V_{g,\underline d,\Delta}$, then for a general point $[C]\in V$, the curve $C$ is nodal, and all its singular points are contained in the maximal torus.
As in the proof of Proposition~\ref{prop:tangency}, suppose $C$ has $m$ irreducible components $C_1,\dotsc, C_m$. Set $g_k:=\pg(C_k)$, $\CL_k:=\CO_{S_\Delta}(C_k)$, and let $\underline d_k$ be the tangency profile of the parametrized curve induced by the normalization of $C_k$. Then $V$ is dominated by an irreducible component $\prod_{k=1}^mV_k$ of $\prod_{k=1}^m \overline V^\irr_{g_k,\underline d_k, \CL_k}$, where $V_k$ is an irreducible component of $\overline V^\irr_{g_k,\underline d_k,\mathscr L_k}$.

Let $\Delta_k\subset M_\mathbb R$ be the polygon that defines the line bundle $\mathscr L_k$. If $\Delta_k$ is degenerate, then $g_k=0$, $\underline d_k$ is trivial, $|\underline d_k|=1$, and $V_k=\overline V^\irr_{0,\underline d_k,\mathscr L_k}=|\mathscr L_k|$. Otherwise it is easy to check that $\Delta_k$ is still $h$-transverse, and we have a natural map $\pi_k\colon S_\Delta\rightarrow S_{\Delta_k}$ obtained from a sequence of toric blow-ups.  Note that $\pi_k$ induces a natural map $\pi'_k\colon V_k\rightarrow \overline V^\irr_{g_k,\underline d_k,\Delta_k}$. By Proposition \ref{prop:tangency}, we have $\dim V_k=\dim \overline V^\irr_{g_k,\underline d_k,\Delta_k}$. Since a general point of $V_k$ represents a curve in $S_\Delta$ that does not intersect the exceptional divisors of $\pi_k$, the map $\pi'_k$ maps $V_k$ birationally to an irreducible component $V'_k$ of $\overline V^\irr_{g_k,\underline d_k,\Delta_k}$. Now by Theorem~\ref{thm:main thm}, $V'_k$ contains $V^\irr_{0,\underline d_k,\Delta_k}$, hence $V_k$ contains $V^\irr_{0,\underline d_k,\mathscr L_k}$, and $\prod_{k=1}^mV_k$ contains $\prod_{k=1}^mV^\irr_{0,\underline d_k,\mathscr L_k}$. Since the locus of nodal curves is open in $\prod_{k=1}^mV_k$ by \cite[Tag 0DSC]{stacks-project}, and similarly the locus of curves with singular locus contained in the maximal torus of $S_\Delta$ is open, it remains to prove the following claim.
\end{proof}

\begin{cl}
Let $C=\cup_{1\leq k\leq m}C_k$ be a curve such that $[C_k]\in V^\irr_{0,\underline d_k,\mathscr L_k}$ is general. Then $C$ is nodal, and its singular locus is contained in the maximal torus.
\end{cl}
\begin{proof} According to Lemma~\ref{lem:irreducibility in rational case}, the singular locus of each $C_k$ is contained in the maximal torus of $S_\Delta$. On the other hand, as each $V^\irr_{0,\underline d_k,\mathscr L_k}$ admits a torus action, we can move $C_k$ so that it misses the points of intersection of the other components of $C$ with the boundary of $S_\Delta$.
As a result, the intersections of components of $C$ are also contained in the maximal torus, hence the singular locus of $C$ is contained in the maximal torus. It remains to show that $C$ is nodal.

Let $\underline {d}=\left(\left\{d_{i,j}\right\}_{1\leq j\leq a_i}\right)_{1\leq i\leq s}$ and $C\cap D_i= \{p_{i,1},...,p_{i,a_i}\}$ for all $1\leq i\leq s$.
Assume to the contrary that $C$ has a singular point $p$ that is not a node. By Lemma~\ref{lem:irreducibility in rational case}, each $C_i$ contains no unibrach singularities, and thus there are the following two cases to consider.

Case 1: \textit{there are at least three branches of $C$ that pass through $p$.} Again, since the variety $V^\irr_{0,\underline d_k,\mathscr L_k}$ admits a torus action, and $C_k$ is general, these branches must come from the same component of $C$, say $C_1$. Then $\Delta_1$ is non-degenerate and has height at least 2. In particular, we have $|\underline d_1|\geq 4$ as $\underline d_1$ is trivial on toric divisors corresponding to non-horizontal sides of $\Delta$.
Let $U\subset |\mathscr L_\Delta|$ be the locus of curves $C'=\cup_{1\leq k\leq m}C'_k$ such that $[C'_k]\in V^\irr_{0,\underline d_k,\mathscr L_k}$ for all $k$ and $p\in C'_1$ is a singular point with at least three branches.
Then
$$\dim U\geq \dim V^\irr_{0,\underline d_1,\mathscr L_1}-2+\sum_{k=2}^m\dim V^\irr_{0,\underline d_k,\mathscr L_k}=\sum_{k=1}^m\left(|\underline d_k|-1\right)-2.$$

Let $\pi_p\colon S\rightarrow S_\Delta$ be the surface obtained from $S_\Delta$ by blowing up $p$. Let $E$ be the exceptional divisor, and $\widetilde C$ (respectively, $\widetilde C_k$) be the strict transform of $C$ (respectively, $C_k$). It follows that
$$K_S\sim \pi_p^*K_{S_\Delta}+E\mathrm{\ and\ } \widetilde C_1\sim \pi^*C_1-aE,$$
where $a\geq 3$, and $\widetilde C_k\sim \pi^*C_k$ for $k\geq 2$. Taking strict transforms in $S$ gives a rational map $U\dashrightarrow V_{1-m,\mathscr O_S(\widetilde C)}$, since $1-m=p_g(C)$. It is an inclusion on the domain and we denote by $U'$ its image. Then $[\widetilde C]\in U'$ is general, and the normalization $f\colon C^\nu\rightarrow S$ of $\widetilde C$ is immersed.
Hence the normal sheaf $\mathscr N_f$ is a line bundle on $C^\nu$, and its degree on each component $C^\nu_k$ of $C^\nu$ is $-K_S\cdot \widetilde C_k+\deg K_{C^\nu_k}$. In other words, $\mathscr N_f$ has degree $-K_{S_\Delta}\cdot C_1-2-a$ on $C^\nu_1$, and degree $-K_{S_\Delta}\cdot C_k-2$ on $C^\nu_k$ for $k\geq 2$.

Now since char$(K)> \mathtt w(\Delta)$, each $d_{i,j}$ is not divisible by char$(K)$, and the space of first order deformations of $f$ (with fixed target $S_\Delta$) that preserves the tangency with the boundary of $S_\Delta$ is identified with $H^0\left(C^\nu,\mathscr N_f\left(-\sum_{i=1}^s\sum_{j=1}^{a_i}(d_{i,j}-1)p_{i,j}\right)\right)$. Since $|\underline d_1|\geq 4$, we have
\begin{eqnarray*}
\dim T_{[C]}U&=&\dim T_{[\widetilde C]}U'\leq h^0\left(C^\nu,\mathscr N_f\left(-\sum_{i=1}^s\sum_{j=1}^{a_i}(d_{i,j}-1)p_{i,j}\right)\right)= \\ &=& \max \left\{|\underline d_1|-1-a,0 \right\}+\sum_{k=2}^m(|\underline {d_k}|-1)<\dim U.
\end{eqnarray*}
This provides a contradiction.

Case 2: \textit{there are two branches of $C$ passing through $p$ with non-trivial tangency at $p$.} Denote by $\sigma$ the tangent direction of these branches at $p$.
Let $\pi_p\colon S\rightarrow S_\Delta$ be the surface obtained by blowing up $p$ as in Case 1, and $\pi_{\sigma}\colon S'\rightarrow S$ be obtained by blowing up the point in the exceptional divisor $E$ of $S$ corresponding to the direction $\sigma$. Let $\widetilde C$ and $\widetilde C_k$ be the strict transforms of $C$ and $C_k$ in $S'$ respectively, and $f'\colon C^\nu\rightarrow S'$ the map induced by normalizing $\widetilde C$. Let $F_1$ be the pullback of $E$ to $S'$, and $F_2$ the exceptional divisor of $\pi_\sigma$. Then $K_{S'}=\pi_\sigma^*\pi_p^*K_{S_\Delta}+F_1+F_2$

Case 2.1: \textit{the two branches at $p$ come from two different components, say $C_1$ and $C_2$.}  Then at least one of $\Delta_1$ and $\Delta_2$ is non-degenerate.
By the argument above we can assume that $p$ is contained in the maximal torus.
By assumption, we have a rational map from $\prod_{1\leq k\leq m}V^\irr_{0,\underline d_k,\mathscr L_k}$ to the incidence variety $P$ of a point contained in a line in the maximal torus of $S_\Delta$; namely, it is given by mapping $C$ to the point $p$ and the tangent line of $C_1$ at $p$ (in direction $\sigma$). Let $W$ be the fiber of this map for our fixed $p$ and $\sigma$. Since $P$ has dimension $3$, we have 
\[
\dim W\geq \sum_{k=1}^m\dim V^\irr_{0,\underline d_k,\mathscr L_k}-3=\sum_{k=1}^m\left(|\underline d_k|-1\right)-3.
\]
Moreover, if $\Delta_1$ or $\Delta_2$ is degenerate, then $\dim W\geq \sum_{k=1}^m\left(|\underline d_k|-1\right)-2$ since the choice of the tangent direction $\sigma$ is unique, and the image of the map to $P$ has dimension at most $2$.

Similar as in Case 1, passing to strict transforms in $S'$ gives a rational inclusion  $W\dashrightarrow V_{g,\mathscr O_{S'}(\widetilde C)}$. Let $W'$ be the image of this map. By  construction, we have
$$\widetilde C_k\sim \pi_\sigma^*\pi_p^*C_k-F_1-F_2 \mathrm{\ for\ }k=1,2,
\mathrm{\ and\ } \widetilde C_k\sim \pi^*C_k\mathrm{\ for\ }k>2.$$
Hence $\mathscr N_{f'}$ has degree $-K_{S_\Delta}\cdot C_k-4$ on $C^\nu_k$ for $k=1,2$, and degree $-K_{S_\Delta}\cdot C_k-2$ on $C^\nu_k$ for $k\geq 3$. Thus
\begin{eqnarray*}
\dim T_{[C]}W &=& \dim T_{[\widetilde C]}W'\leq h^0\left(C^\nu,\mathscr N_{f'}\left(-\sum_{i=1}^s\sum_{j=1}^{a_i}(d_{i,j}-1)p_{i,j}\right)\right) = \\ &=& \sum_{k=1}^2\max\left\{|\underline d_k|-3,0\right\}+\sum_{k=3}^m\left(|\underline {d_k}|-1\right)<\dim W.
\end{eqnarray*}
This again provides a contradiction.

Case 2.2: \textit{the two branches at $p$ come from the same component, say $C_1$.}  Then $C_1$ is singular and $\Delta_1$ has height at least 2 as in Case 1. By Lemma~\ref{lem:height 2}, we have $|\underline d_1|\geq 5$. Similar to $W$ in Case 2.1, let $Z\subset |\mathscr L_\Delta|$ be the locus of curves  $C'=\cup_{1\leq k\leq m}C'_k$ such that $\left[C'_k\right]\in V^\irr_{0,\underline d_k,\mathscr L_k}$, $p\in C'_1$ is a tacnode, and the tangent direction of $C'_1$ at $p$ is $\sigma$. Arguing as in Case 2.1, we get
\[
\dim Z\geq \sum_{k=1}^m\left(|\underline d_k|-1\right)-3.
\]
Let $Z'$ be the image of the rational inclusion  $Z\dashrightarrow V_{g,\mathscr O_{S'}(\widetilde C)}$.
We have
$$\widetilde C_1\sim \pi^*C_1-2F_1-2F_2\mathrm{\ and\ } C_k\sim \pi^*C_k\mathrm{\ for\ }k>2.$$
Therefore, $\mathscr N_{f'}$ has degree $-K_{S_\Delta}\cdot C_1-6$ on $C^\nu_1$, and degree $-K_{S_\Delta}\cdot C_k-2$ on $C^\nu_k$ for $k\geq 2$. It follows that
\begin{eqnarray*}
\dim T_{[C]}Z &=& \dim T_{[\widetilde C]}Z'\leq h^0\left(C^\nu,\mathscr N_{f'}\left(-\sum_{i=1}^s\sum_{j=1}^{a_i}(d_{i,j}-1)p_{i,j}\right)\right)= \\ &=& \max\left\{|\underline d_1|-5,0\right\}+\sum_{k=2}^m\left(|\underline {d_k}|-1\right)<\dim Z.
\end{eqnarray*}
This is again a contradiction.
\end{proof}

We conclude this section with an example showing that some bound on $\mathrm{char}(K)$ is necessary in order to obtain generically nodal curves in Theorem \ref{thm:Zariski}.
\begin{ex}\label{ex:severi with tangency}
    Suppose $\Delta$ is the parallelogram with vertices $(0,0)$, $(0,-1)$, $(q,0)$ and $(q,1)$, where $q$ is a positive power of the characteristic of $K$. Then according to \cite[Theorem 4.2]{Tyo14}, for any $[C]\in V^\irr_{0,\Delta}$, all singularities of $C$ are unibranch. Moreover, for a general curve $[C']\in V^\irr_{0,\Delta}$, $C'\cap C$ consists of two points, both of multiplicity $q$. In particular, the torus action on $C'$ cannot be used to change the tangency relation between $C$ and $C'$.
\end{ex}

\section{Irreducibility of Severi Varieties}\label{sec:irr}
In this section, we use a result of Lang \cite{lang2020monodromy} on the monodromy action on the nodes of a general rational curve in $V^\irr_{0,\Delta}$ to prove the irreducibility of $V^\irr_{g,\Delta}$ in certain cases. Let $\{m_i\}_{1\leq i\leq s}$ be the set of primitive integral vectors along the sides of $\Delta$. Set
$$l:=5\max_{1\leq i\leq s}||m_i||^2,$$ where $||m_i||$ denotes the Euclidean length of $m_i$.
As in \cite{lang2020monodromy}, we denote by $\mathscr C_{\geq l}(S_\Delta)$ the set of ample line bundles $\mathscr L$ on $S_\Delta$ such that $\deg(\mathscr L|_C)\geq l$ for any curve $C\subset S_\Delta$. Let $N_\Delta\subseteq N$ be the sublattice generated by the primitive normals of the sides of $\Delta$.

\begin{thm}\label{thm:monodromy}
Suppose char$(K)=0$. Let $\Delta$ be an $h$-transverse polygon. If $\mathscr L_\Delta\in\mathscr C_{\geq l}(S_\Delta)$ and $N_\Delta=N$, then $V^\irr_{g,\Delta}$ is irreducible for all $0\leq g\leq \#|\Delta^\circ\cap M|$.
\end{thm}
\begin{proof}
By Lefschetz's principle, we may assume $K=\mathbb C$. According to  Theorem~\ref{thm:Zariski}, a general curve $[C]\in V^\irr_{0,\Delta}$ is nodal and has $\#|\Delta^\circ\cap M|$ nodes. As in the planar case, $\overline V^\irr_{g,\Delta}$ thus has $\binom{\#|\Delta^\circ\cap M|}{g}$ smooth branches through $[C]$, corresponding to smoothing out subsets of $g$ nodes of $C$. Since $\mathscr L_\Delta\in\mathscr C_{\geq l}(S_\Delta)$, according to \cite[Theorem 1 and Proposition 7.4]{lang2020monodromy}, the monodromy group of $V^\irr_{0,\Delta}$ acts on $C^\sing$ as the group of deck transformations of a map from $C^\sing$ to a quotient of the group $\mathrm{Hom}(M_\Delta/M,\mathbb C)$, where $M_\Delta$ is the dual lattice of $N_\Delta$. As $N_\Delta=N$, the group  $\mathrm{Hom}(M_\Delta/M,\mathbb C)$ is trivial and hence the monodromy group acts as the full symmetric group on $C^\sing$. In particular, the action on the set of subsets of g nodes is transitive. Consequently, there is a unique irreducible component of $\overline V^\irr_{g,\Delta}$ containing $[C]$. On the other hand, by Theorem~\ref{thm:main thm} any irreducible component of $\overline V^\irr_{g,\Delta}$ contains $[C]$. Thus $V^\irr_{g,\Delta}$ is irreducible.
\end{proof}

\begin{ex}\label{ex:kites irreduciblity}
    Suppose $\Delta$ is a kite as in Example~\ref{ex:h-transverse polygon} and $\mathscr L$ a sufficiently ample line bundle on $S_\Delta$.
    If $N_\Delta=N$, e.g., $(k,k')=(1,2)$ as in Figure~\ref{fig:htransverse1}, then the Severi varieties $V^\irr_{g,\mathscr L}$ are either empty or irreducible for all $g\ge 0$ by Theorem~\ref{thm:monodromy}. On the other hand, if $N_\Delta\neq N$, e.g., $(k,k')=(3,3)$ as in Figure~\ref{fig:diamond}, and $\mathscr L=\mathscr L_\Delta^{\otimes l}=\mathscr L_{l\Delta}$, then $V^\irr_{g,\mathscr L}$ is known to be reducible for $g=1$ and all $l\ge 1$ by \cite[Theorem A]{LT20} applied to the polygon $l\Delta$.
\end{ex}

In Theorem~\ref{thm:monodromy} we do not allow for non-trivial tangency profiles. Our final example shows that imposing tangency conditions requires a separate analysis of the monodromy action:

\begin{ex}
    Let $\Delta$ be the $h$-transverse polygon with vertices $(0,0), (3,1), (3, -1), (6, 1), (6, -1)$ and $(9,0)$. Suppose $\mathrm{char}(K) = 0$ and let $\underline d$ be the tangency profile given by imposing a point of order $3$ along each of the $2$ toric divisors that correspond to horizontal sides of $\Delta$. Using a similar argument as in the proof of \cite[Proposition 3.1]{LT20}, one can check that $V_{1, \Delta}^\irr$ is irreducible, whereas $V_{1, \underline d, \Delta}^\irr$ contains $2$ irreducible components.
\end{ex}

\providecommand{\bysame}{\leavevmode\hbox to3em{\hrulefill}\thinspace}
\providecommand{\MR}{\relax\ifhmode\unskip\space\fi MR }
\providecommand{\MRhref}[2]{%
  \href{http://www.ams.org/mathscinet-getitem?mr=#1}{#2}
}
\providecommand{\href}[2]{#2}

	\end{document}